\UseRawInputEncoding
\documentclass[11pt,leqno]{amsart}
\usepackage{amsthm,amsfonts,amssymb,amsmath,oldgerm,upgreek,xcolor,marvosym}
\numberwithin{equation}{section}
\usepackage{fullpage}
\usepackage[british]{babel}
\usepackage{amsmath}
\usepackage{amsfonts}
\usepackage{amssymb}
\usepackage{setspace}
\usepackage{mathrsfs}
\usepackage{fancyhdr}
\usepackage{graphicx}
\usepackage[colorlinks = true, linkcolor=blue,
filecolor=darkgreen,
urlcolor=cadmiumgreen,
citecolor=red,pdftex]{hyperref}
\usepackage{psfrag}
\usepackage{listings}
\usepackage{enumitem}
\usepackage{amsmath}
\usepackage{tikz-cd}
\usepackage[normalem]{ulem}

\usepackage[normalem]{ulem}

\definecolor{airforceblue}{rgb}{0.36, 0.54, 0.66}
\definecolor{darkgreen}{rgb}{0.0, 0.2, 0.13}
\definecolor{darkslategray}{rgb}{0.18, 0.31, 0.31}
\definecolor{mediumjunglegreen}{rgb}{0.11, 0.21, 0.18}
\definecolor{prussianblue}{rgb}{0.0, 0.19, 0.33}
\definecolor{warmblack}{rgb}{0.0, 0.26, 0.26}
\definecolor{cadmiumgreen}{rgb}{0.0, 0.42, 0.24}
\definecolor{mauvetaupe}{rgb}{0.57, 0.37, 0.43}
\definecolor{maroon(x11)}{rgb}{0.69, 0.19, 0.38}

\usepackage[top=30mm,bottom=30mm,left=25mm,right=25mm,a4paper]{geometry}

\def\eps{\varepsilon }

\newcommand{\omk}{\raisebox{0.5mm}{\text{\footnotesize $\langle $}} k  \raisebox{0.5mm}{\text{\footnotesize $ \rangle $}}}

\newcommand{\omtilk}{\raisebox{0.5mm}{\text{\footnotesize $\langle $}} \tilde k  \raisebox{0.5mm}{\text{\footnotesize $ \rangle $}}}

\newcommand{\omdifk}{\raisebox{0.5mm}{\text{\footnotesize $\langle $}} k-\tilde k  \raisebox{0.5mm}{\text{\footnotesize $ \rangle $}}}
\newcommand{\omdotk}{\raisebox{0.5mm}{\text{\footnotesize $\langle $}} \cdot  \raisebox{0.5mm}{\text{\footnotesize $ \rangle $}}}

\newcommand\br{\begin{remark}}
\newcommand\er{\end{remark}}
\newcommand\bp{\begin{pmatrix}}
\newcommand\ep{\end{pmatrix}}
\newcommand{\be}{\begin{equation}}
\newcommand{\ee}{\end{equation}}
\newcommand{\ba}[1]{\begin{array}{#1}}
\newcommand{\ea}{\end{array}}

\newcommand{\beg}{\begin{example}}
\newcommand{\eeg}{\end{exaplem}}
\newcommand{\bpr}{\begin{proposition}}
\newcommand{\epr}{\end{proposition}}
\newcommand{\bt}{\begin{theorem}}
\newcommand{\et}{\end{theorem}}
\newcommand{\bc}{\begin{corollary}}
\newcommand{\ec}{\end{corollary}}
\newcommand{\bl}{\begin{lemma}}
\newcommand{\el}{\end{lemma}}
\newcommand{\bd}{\begin{definition}}
\newcommand{\ed}{\end{definition}}
\newcommand{\brs}{\begin{remarks}}
\newcommand{\ers}{\end{remarks}}

\font\dsrom=dsrom10 scaled 1300
\def \indic{\textrm{\dsrom{1}}}

\font\dsrom=dsrom10 scaled 900
\def \indicf{\textrm{\dsrom{1}}}

\newtheorem{theo}{Theorem}
\newtheorem{prop}[theo]{Proposition}
\newtheorem{cor}[theo]{Corollary}
\newtheorem{lem}[theo]{Lemma}
\newtheorem{defi}[theo]{Definition}

\newtheorem{rem}[theo]{Remark}

\numberwithin{equation}{section}
\newtheorem{assump}{Assumption}

\def\eps {\varepsilon}

\def\part{\partial}
\def\part{\partial}

\newcommand{\CC}{{\mathbb C}}

\newcommand{\NN}{{\mathbb N}}

\newcommand{\RR}{{\mathbb R}}

\newcommand\cB{{\mathcal B}}
\newcommand\cC{{\mathcal C}}

\newcommand\cE{{\mathcal E}}
\newcommand\cF{{\mathcal F}}

\newcommand\cI{{\mathcal I}}
\newcommand\cJ{{\mathcal J}}
\newcommand\cK{{\mathcal K}}
\newcommand\cL{{\mathcal L}}

\newcommand\cQ{{\mathcal Q}}

\newcommand\cS{{\mathcal S}}
\newcommand\cT{{\mathcal T}}

\newcommand\cY{{\mathcal Y}}

\newsavebox\foobox
\def\slantvalue{0}
\newcommand{\slantboxengine}[2][\slantvalue]{\mbox{%
        \sbox{\foobox}{#2}%
        \hskip\wd\foobox
        \pdfsave
        \pdfsetmatrix{1 0 #1 1}%
        \llap{\usebox{\foobox}}%
        \pdfrestore
}}
\newcommand\slantbox[2][\slantvalue]{%
  \edef\slantvalue{#1}\expandafter\slantboxhelpA#2 \relax\relax}
\def\slantboxhelpA#1 #2\relax{%
  \slantboxengine{#1}%
  \ifx\relax#2\relax\else\ \slantboxhelpA#2\relax\fi
}

\newcommand\slu{\textit{\c{u}}}
\newcommand\slw{\textit{\c{w}}}

\title{The  Schr\" odinger--Klein--Gordon system revisited \smallskip
\\
{\it \footnotesize Dedicated to the memory of our  colleague, Dimitri Yafaev}}
\date{31  August 2026}
\author[Z. Ammari]{Zied Ammari}
\address[Zied Ammari]{Universite Marie \& Louis Pasteur, Laboratoire de Math\'ematiques LmB, UFR Sciences et 
techniques, 16 route de Gray, 25030 Besan\c{c}on CEDEX \\ France}
\author[C. Cheverry]{Christophe Cheverry}
\address[Christophe Cheverry]{Institut Math\'ematique de Rennes,
Campus de Beaulieu, 263 avenue du G\'en\'eral Leclerc CS 74205
35042 Rennes Cedex\\France}

\begin{document}

\begin{abstract}
We introduce a microlocal formulation of the Schr\"odinger--Klein--Gordon system describing the
interaction between a non-relativistic quantum particle and a Klein--Gordon field through Yukawa coupling.
Instead of working directly with the Schr\"odinger wave function, we represent the quantum
component by its Fourier--Wigner transform, and we rewrite the Klein--Gordon equation in terms of a
complex Fourier variable. Eliminating the field variable yields a closed nonlinear
Fourier--Moyal equation on phase space whose unknown is the Fourier--Wigner distribution associated with the quantum
component.

The resulting formulation provides a refined description of the particle--field interaction at
the microlocal level. In particular, the original cubic coupling is transformed into a quadratic
self-interaction governed by an explicit  bilinear operator. This new  representation
 permits the use of techniques from time-frequency analysis. 
 
 Building on this, we develop a well-posedness theory for the Fourier--Moyal equation in anisotropic 
spaces involving Wiener-type norms and prescribed moduli of continuity. Local existence and
uniqueness are established under general assumptions on the ultraviolet cutoff, together with
propagation of microlocal regularity. We then introduce a new construction of weak
$L^2$-solutions by exploiting compactness properties of Fourier--Wigner transforms and  obtain
global existence for prepared data. The aim  is to provide an alternative analytical framework
for the study of Yukawa-type interactions and  to establish a bridge between nonlinear dispersive
equations and phase-space methods.
\end{abstract}

\maketitle


{\small \parskip=1pt
\tableofcontents
}


\section{Introduction}

The Schr\"odinger--Klein--Gordon (SKG in abbreviated form) system is one of the fundamental nonlinear models
describing the interaction between a quantum particle and a relativistic scalar field. It was originally introduced as a semiclassical approximation of the Yukawa interaction. It  currently occupies
an important position at the crossroads of nonlinear dispersive equations, mathematical quantum
field theory and semiclassical analysis. Besides its physical relevance, it has
served for several decades as a benchmark for the development of analytical techniques for
coupled dispersive systems.

The mathematical theory of the SKG system is today well developed. Classical works have  established
global well-posedness for finite-energy solutions by combining energy estimates and compactness
arguments (see e.g.~\cite{MR778979,MR519635,MR1441722,MR390555,MR515899,MR904544}), while later developments have exploited Strichartz estimates and Bourgain-type spaces to
reach low-regularity regimes (see e.g.~\cite{MR2403699,MR2237673,MR2035502,MR2968610} and Subsection \ref{subsec:Presentationoftheequations}). Scattering theory \cite{MR4382307,MR1381889,MR1275411,MR2159005}, stability of solitary waves \cite{MR2024878,MR2585081,MR1393148,MR1888144}, attractors for
dissipative models \cite{MR1062399,MR1448829}, and connections with quantum field theory \cite{MR3255099,MR3737034,MR4548997}, have also received considerable
attention. Nevertheless, almost all existing approaches share a common feature: they formulate
the dynamics in terms of the physical variables consisting of the Schr\"odinger wave function and
the Klein--Gordon field.

The purpose of the present work is to develop a different point of view based on phase-space
analysis. Rather than studying directly the Schr\"odinger wave function $ u $, we describe the quantum
particle through its Fourier--Wigner transform. Simultaneously, we represent the
Klein--Gordon field by a complex Fourier variable which diagonalizes the free wave evolution.
Combining these two observations leads to a self-contained nonlinear evolution equation on
phase space, referred to throughout this paper as the Fourier--Moyal (FM) equation \eqref{transporteqonw}.

The resulting equation possesses several  remarkable structural properties. First, the nonlinear source term becomes quadratic. Indeed, once the Klein--Gordon component has been
solved explicitly through Duhamel's formula and substituted back into the Schr\"odinger
equation, the particle-field interaction appears as a quadratic self-interaction acting directly
on the Fourier--Wigner distribution. This interaction is encoded by the explicit Yukawa
bilinear operator \eqref{bilitodefine} whose structure reflects the geometry of the underlying phase space.

Second, the Fourier--Moyal formulation separates the transport
mechanism generated by the free Schr\"odinger evolution from the genuinely nonlinear part
associated with the exchange of mesons. The linear propagation is represented by a transport
operator in phase space, whereas the nonlinear interaction appears through an integral operator
whose kernel can be made  explicit. This decomposition provides a natural framework for
microlocal analysis and reveals features that remain hidden in the original variables.

Third, this formulation connects the analysis of SKG, and its generalizations called $ \text{SKG}_\chi $ - see \eqref{PSKG}, with modern
time-frequency analysis. The Fourier--Wigner transform, modulation spaces, Wiener amalgam
spaces and Moyal-type operators become intrinsic objects of the theory rather than auxiliary
technical tools. This allows us to investigate propagation of phase-space regularity and
continuity properties measured through suitable moduli of continuity $ \omega $.

\medskip

The main contributions of this work can be summarized as follows.

\begin{itemize}
\item \textbf{Derivation of the Fourier--Moyal formulation.} 
We explain the origin of the Fourier--Moyal equation and we  prove its
equivalence with the Schr\"odinger--Klein--Gordon system (SKG) for sufficiently regular
prepared data. This derivation also identifies the natural conserved quantities of the
microlocal dynamics and expresses the total energy entirely in terms of phase-space variables, see Section \ref{subsec:TransitionfromtheSKGsystemtotheSKGSequation}.

\item \textbf{Well-posedness theory in anisotropic Banach spaces.}
We develop a well-posedness theory directly for the Fourier--Moyal
equation. We introduce anisotropic Banach spaces adapted to the transport structure of the
equation and prove continuity properties of the Yukawa bilinear operator in these spaces.
These estimates yield local existence and uniqueness together with propagation of prescribed
moduli of continuity, see Theorem  \ref{mildtotal}. As a byproduct, in Section~\ref{sec:L2preparedsol}, one can present a propagated metric, which is related to the microlocal norm $ L^\infty(\cC^{\omega}_b) $
of the Fourier--Wigner transform of $ u $ (see Subsections  \ref{subsec:Notationsfunctionalsetting} and \ref{subsec:propagationinfty}),  enabling the comparison of weak $L^2$-solutions to $ \text{SKG}_\chi $, see Theorem \ref{GESPD}.

\item \textbf{Construction of global weak solutions.}  
We construct weak solutions directly at the level of the
Fourier--Wigner transform. Along these lines, instead of relying on classical compactness arguments for the
Schr\"odinger field, we exploit weak continuity properties of prepared Fourier--Wigner
distributions together with new microlocal estimates. This yields an alternative construction
of global $L^2$-solutions and establishes a  correspondence between weak solutions of the
Fourier--Moyal equation and weak solutions of the original Schr\"odinger--Klein--Gordon system, see Theorem \ref{GESPD} and Proposition \ref{globalexistence}.
\end{itemize}

Beyond the specific Yukawa model considered here, the present work illustrates how nonlinear
particle--field systems may be reformulated as autonomous equations on phase space. We expect
that this perspective will provide a useful framework for the study of other coupled dispersive
systems, semiclassical limits, quantum kinetic models, and problems involving propagation of
microlocal structures.

The paper is organized as follows. In Section~\ref{sec:Introduction} we explain the microlocal formulation of the 
Schr\"odinger--Klein--Gordon system and introduce our functional setting. Section~\ref{sec:KGSsystemrevisited} is devoted to the derivation of the Fourier--Moyal equation and to
its relation with the classical Schr\"odinger--Klein--Gordon system. In Section~\ref{sec:cauchyproblem} we establish
local well-posedness in anisotropic phase-space contexts. Finally, Section~\ref{sec:L2preparedsol} develops the theory
of prepared weak solutions and proves global existence through microlocal compactness
arguments.


\vskip 4mm

\section{Microlocal formulation} \label{sec:Introduction}
As explained above, this article is devoted to a model motivated by Yukawa's pioneering work \cite{Yukawa} on the strong nuclear force. At the semiclassical level, the interaction between the nucleon wave function $u$ and the meson field $A$, which mediates the attractive force binding nucleons within an atomic nucleus, is described (in dimensionless units) by the total energy
\begin{equation}\label{eq.skg.energytr}
\begin{array}{rl}
\mathcal E_{\chi} (t) := \! \! \! & \displaystyle \frac{1}{2} \,  \Bigl(\parallel \nabla_x u
(t,\cdot) \parallel^2_{L^2} + \|A (t,\cdot) \|^2_{L^2} + \|
\nabla_x A (t,\cdot) \|^2_{L^2}+ \|
\part_t A (t,\cdot) \|^2_{L^2} \Bigr) \\
\ & \displaystyle
+ \int_{\RR^d} \chi(D) A (t,x) \ |u(t,x)|^2 \, dx ,
\end{array}
\end{equation}
where $ (t,x) \in \RR \times \RR^d $ with $ d \in \NN^* $, and where $ \chi(D) $ is the Fourier multiplier associated with the symbol $ \chi $ whose action is given by
$$ \cF_x \bigl( \chi(D) A (t,\cdot) \bigr) = \chi (k) \, (\cF_x A) (t,k) , \qquad (\cF_x A) (t,k) = \int_{\RR^d} e^{-i k \cdot x} A(t,x) \, dx . $$
Throughout this text, we work with bounded even functions $ \chi $, thus  satisfying
\begin{equation} \label{voevenchi}
\chi \in L^\infty (\RR) \, , \qquad \chi(k) = \chi(-k) \, , \quad \forall k \in \RR^d \, .
\end{equation}
The Schr\" odinger--Klein--Gordon system $ \text{SKG}_{\chi} $ is the Euler--Lagrange equation derived from (\ref{eq.skg.energytr}).
It is composed of the two
following coupled scalar equations
\begin{subequations}\label{PSKG}
\begin{eqnarray}
& & \displaystyle i \partial_t u =  - \frac{1}{2} \Delta u + \chi(D) A  \times  u , \label{PSKG1} \\
& & \displaystyle (\Box +1 ) A = - \cF_k^{-1} (\chi) * \vert u \vert^2 .   \label{PSKG2}
\end{eqnarray}
\end{subequations}

\noindent In \eqref{PSKG}, the symbol $ \Box $ stands for the d'Alembertian $ \Box = \part^2_{tt} - \Delta_x $ (with $ \Delta_x $
being the Laplacian), $ \times $ is the multiplication, $ * $ is the convolution, and  $ \cF_k^{-1} $ is the inverse Fourier transform. Due to \eqref{voevenchi}, the distribution $ \cF_k^{-1} (\chi) $ is real-valued.
The system (\ref{PSKG1})-(\ref{PSKG2}) is complemented with initial data
\begin{equation}\label{inidataSKGinif}
u(0,\cdot) = u_0 (\cdot) , \qquad A(0,\cdot) = A_0 (\cdot), \qquad \part_t A(0,\cdot) = A_1 (\cdot).
 \end{equation}
We assume that both $ A_0 $ and $ A_1 $ are real-valued. By this way, the same applies to $ A(t,\cdot) $. Thus, the system $ \text{SKG}_{\chi} $ describes the coupled time evolution of the complex-valued function $ u $ and of  the real-valued function $ A $. For $  \chi $ constant equal to one, that is $ \chi \equiv \indic_\RR $, the interactions in the right hand sides of \eqref{PSKG} are purely local. 

Then, by taking $\chi=\indicf$, we recover the classical Schr\"odinger--Klein--Gordon system (see, e.g., \cite{MR778979,MR519635,MR515899}), which we simply denote by 
\[
\mathrm{SKG}\equiv \mathrm{SKG}_{\indicf},
\]
namely,
$$ \left \lbrace \begin{array}{l}
\displaystyle i \partial_t u =  - \frac{1}{2} \Delta u + A  \,  u , \\
\displaystyle (\Box +1 ) A = - \vert u \vert^2 .
\end{array} \right.
\leqno(\text{SKG})
$$
For $ \chi \not \equiv \indic_\RR $,
the interplay in \eqref{PSKG} between $ u $ and $ A $ becomes non local. The cutoff function $ \chi $ is introduced as an ultraviolet regularization motivated by renormalization in quantum field theory (QFT). Its role is to suppress, or at least attenuate, the contribution of high-frequency modes ∣$ \vert k \vert \gg 1 $, thereby preventing ultraviolet divergences in QFT.

\subsection{Change of perspective}\label{chnageperspective}
In Paragraph \ref{subsec:microlocalap}, we explain how to replace a Schr\"odinger equation like \eqref{PSKG1} by a Moyal equation which is written on the phase space and which gives access to the underlying microlocal properties. In Paragraph \ref{subsec:Kleing}, we look at \eqref{PSKG2} as an ordinary differential equation with parameter $ k \in \RR^d $. This furnishes a practical way to compute $ A $ on the Fourier side.


\subsubsection{The Moyal interpretation of the Schr\"odinger part} \label{subsec:microlocalap}
To emphasize our main perspective, we first recall the formal derivation of the Moyal equation from the first principles of Quantum Mechanics according to the seminal work of Moyal and Baker \cite{MR95025,MR29330}. To simplify the presentation, we restrict our attention in this paragraph to the one-dimensional case $d=1$. Consider a quantum particle represented
by the time-dependent wave function $ u(t,x) $. Given some
Hamiltonian $ \hat H $, the motion of $ u $  is governed by the Schr\"odinger equation
\begin{equation}
\label{intro.eq.schrod}
i\hslash\partial_t u(t)=\hat H u(t).
\end{equation}
Let $x$ and $\xi $ be the position-momentum canonical variables. We assume that $\hat H$ is the Weyl quantization of a real-valued function $H(x,\xi)$. Then, the Wigner transform $ v(t) $ of the nucleon wave function $ u(t)$, defined by
\begin{equation}
\label{intro.eq.wigner}
v(t,x,\xi)=\frac{1}{\pi \hslash} \int_{\mathbb{R}} \overline{u(t,x+y)} \, u
(t,x-y) e^{2i \xi y/\hslash} dy,
\end{equation}
satisfies the Moyal equation
\begin{equation}
\label{intro.eq.moyal}
i\hslash\partial_t v(t)=  H\star v(t)- v(t)\star  H ,
\end{equation}
where the Moyal star product $\star$ is given by the symbolic composition formula
\begin{equation}
\label{intro.eq.starprod}
f\star g=\exp\bigl(i\hslash/2 \, (\partial_x \partial_{\tilde \xi} -\partial_\xi \partial_{\tilde x})\bigr) \bigl( f(x,\xi) g(\tilde x,\tilde \xi) \bigr)_{\mid(x,\xi)=(\tilde x,\tilde \xi)} .
\end{equation}
In \eqref{intro.eq.starprod}, the exponential is understood through its formal power series expansion. The derivatives are taken with respect to the variables $x,\xi,\tilde x$, and $\tilde \xi $, after which one sets $(x,\xi)=(\tilde x,\tilde \xi)$.
So, the Moyal equation \eqref{intro.eq.moyal} describes the evolution of the time-dependent Wigner function $ v(t,x,\xi)$ of a quantum particle
in the phase space. It can be reformulated using the Moyal bracket into
\begin{equation}
\label{intro.eq.moyal.bis}
\partial_t v (t)= \{H,v(t)\}_{\mathrm{Moyal}},
\end{equation}
where
\begin{equation}
\label{intro.eq.moyalbracket}
\{f,g\}_{\mathrm{Moyal}}= \frac{1}{i\hslash} (f\star g-g\star f)=\frac{2}{\hslash} \sin\bigl(\hslash/2\, (\partial_x \partial_{\tilde \xi} -\partial_\xi \partial_{\tilde x})\bigr) \bigl( f(x,\xi) g(\tilde x,\tilde \xi) \bigr)_{\mid(x,\xi)=(\tilde x,\tilde \xi)}.
\end{equation}
The preceding expression involving the sine function is interpreted via its formal power series expansion, exactly as in the definition of the star product \eqref{intro.eq.starprod}. According to Baker \cite{MR95025}, a Fourier integral representation of the Moyal bracket is
$$
\{f,g\}_{\mathrm{Moyal}}(x,\xi)=\frac{2}{\hslash\pi ^{2}}\int_{\mathbb{R}^4} f(x+x',\xi+\xi')g(x+x'',\xi+\xi'')\sin \bigl({\frac{2}{\hslash}}(x'\xi''-x''\xi')\bigr) dx'\, d\xi'\, dx'' \, d\xi'' .
$$
The Moyal equation furnishes a  straightforward  access to a microlocal phase-space  analysis of the  Schr\"odinger equation \eqref{intro.eq.schrod}. Within this alternative viewpoint, one can for instance  address the problem of propagation of microlocal regularities  (versus singularities) or the problem of scattering, see e.g.~\cite{MR393884,MR783182,MR1982788,MR1603761}. Such subject is related to quantum optics where it is common to use the
Wigner quasi-probability distribution in order to derive quantum kinetic equations.   It is used for instance to describe Rabi oscillations and spontaneous emission phenomena \cite{MR789707,MR2271425}. Furthermore, the Wigner distribution is particularly well adapted to the study of atomic motion in electromagnetic fields. Since its introduction in the 1940s, it has become a fundamental tool in quantum mechanics, both from the theoretical and experimental viewpoints. Notably, it gives rise to an equivalent formulation of quantum mechanics, the \emph{Wigner representation}, and naturally leads to the Moyal evolution equation \cite{MR95025,MR29330}.


\subsubsection{The complex  interpretation of the Klein--Gordon part} \label{subsec:Kleing}
Instead of working directly with the real-valued field $A$, with $ \omk $ as in \eqref{Japanesebracket}, it is convenient to introduce the associated complex meson field
\begin{equation} \label{introdealpha}
\alpha (t,k) := (2 \pi)^{-d/2} \bigl \lbrack \omk^{1/2}  \cF_x \bigl( A(t,\cdot) \bigr) (k) + i \omk^{-1/2} \cF_x \bigl( \part_t A(t,\cdot) \bigr) (k) \bigr \rbrack \, .
\end{equation}
At time $ t = 0 $, we find
\begin{equation} \label{SKGreinter2ini}
\alpha_0 (k) := (2 \pi)^{-d/2}\bigl \lbrack  \omk^{1/2}  \cF_x ( A_0 ) (k) + i \omk^{-1/2} \cF_x (A_1) (k) \bigr \rbrack \, .
\end{equation}
As explained in Paragraph~\ref{mesonfield}, the Klein--Gordon equation~\eqref{PSKG2} can be reformulated as the following ordinary differential equation for the complex scalar valued field $ \alpha(t,k) $, where $ k \in \RR^d $ is a parameter and $t\in \RR$:
\begin{equation}\label{SKGSyequibis}
\partial_t \alpha + i \omk \alpha + i (2 \pi)^{-d/2}  \omk^{-1/2} \, \chi (k) \, \cF_x \bigl( \vert u(t,\cdot) \vert^2 \bigr)(k) = 0 .
\end{equation}


\subsection{Fourier--Moyal interpretation of $ \text{SKG}_{\chi} $} \label{subsec:FourierMoyalin}
The purpose of this work is to initiate a microlocal analysis of the system $\text{SKG}_{\chi}$, which is substantially more involved than the Schr\"odinger equation~\eqref{intro.eq.schrod}. Our approach combines the microlocal framework presented in Paragraph~\ref{subsec:microlocalap} with the complex formulation of the Klein--Gordon field introduced in Paragraph~\ref{subsec:Kleing}.


\subsubsection{Fourier--Moyal analysis with $ A $ incorporated} \label{subsubsec:Aincor} The equation \eqref{PSKG1} is clearly of the form  \eqref{intro.eq.schrod} with $ \hslash = 1 $ and with the effect of the meson field $ A $ entering in the definition of the Schr\"odinger Hamiltonian $ \hat H $ through
$$ \hat H \equiv \hat H_A := - \frac{1}{2} \Delta + \chi(D) A \, . $$
We follow the same guiding idea as above.. Let $u(t,x)$ be the Schr\"odinger wave function solving the equation \eqref{PSKG1}. Let $v(t,x,\xi)$ be its associated Wigner function \eqref{wignerforv}. In the case of $ \hat H_A $, the Moyal equation \eqref{intro.eq.moyal.bis} can be formulated as
\begin{equation}\label{KGB1}
\part_t v + \xi \cdot \nabla_x v + \cK \bigl( \chi(D) A \bigr) *_\xi v  = 0 ,
\end{equation}
where the action of $ \cK $ is given by \eqref{formulapourOp}.
As a substitute for $ v $, we can consider the Fourier transform of $ v (t,\cdot) $ with respect to $ (x,\xi) $, which is denoted by $ w (t,k,\eta) $. More precisely, on the phase space $ \RR^d \times \RR^d $, we look at the Fourier--Wigner transform $ w (t,\cdot )$ of $u(t,\cdot)$ which is
 \begin{equation}\label{defunknownw}
\quad w(t,k,\eta) := \frac{1}{(2 \pi)^d} \int_{\RR^d} \! \int_{\RR^d} \! \int_{\RR^d} e^{-ik\cdot x - i \eta \cdot \xi - i \xi \cdot y} \,
u \Bigl( t, x+ \frac{y}{2} \Bigr) \bar u\Bigl( t, x- \frac{y}{2} \Bigr) \ dx dy d\xi .
 \end{equation}
 Recalling that $\mathrm{W}$ denotes the Wigner transform, and using the notation introduced in~\eqref{crosswigner}, we equivalently define\footnote{For simplicity, we set $\hslash=1$. It is more convenient to work with the Fourier--Wigner transform $w(t,\cdot)$ associated with $u(t,\cdot)$ rather than with its Wigner transform $\mathrm{W}(u(t,\cdot))$. Our definition of the Fourier--Wigner transform differs slightly from the standard convention adopted in~\cite{MR1639461}. We also note that $w(t,\cdot)$ can equivalently be defined by the formula~\eqref{autreformulapourwbis}.}
 \begin{equation}\label{defunknownwbis}
\quad w(t,\cdot) := (2 \pi)^{-d} \, \cF_{x,\xi} \circ \text{W} \bigl( u(t,\cdot)\bigr) \, , \qquad \text{W}(u):= \mathbb W(u,u) \, .
 \end{equation}
 From \eqref{KGB1}, we can deduce the transport equation (of Vlasov type)
 \begin{equation}\label{KGB1intermediaire}
\part_t w - (k \cdot \nabla_\eta) w + (2 \pi)^{-d} \cF_{\! x,\xi} \bigl\lbrack \cK \bigl( \chi(D) A \bigr) \bigr\rbrack *_{k} w = 0 .
\end{equation}


\subsubsection{The Yukawa bilinear map} \label{subsubsec:Yuka}
Working with $w$ amounts to replace the wave function $u$ by its Fourier--Wigner transform in \eqref{PSKG2} or, equivalently, in \eqref{SKGSyequibis}. Furthermore, owing to the identity
\[
\mathcal{F}_x\bigl(|u(t,\cdot)|^2\bigr)(k)=w(t,k,0),
\]
which follows from~\eqref{recallwigff}, the Klein--Gordon equation in~\eqref{SKGSyequibis} becomes \emph{linear} in the unknown $w$, whereas it is quadratic in $u$ in the original formulation~\eqref{PSKG2}. Consequently, the field $A(t,\cdot)$ can be expressed explicitly as the image of $(\alpha_0,w)$ under a linear operator. This observation is a key ingredient of our approach. Indeed, after substituting this representation of $A(t,\cdot)$ into~\eqref{KGB1intermediaire}, the interaction term becomes the sum of a bilinear operator acting on $(\alpha_0,w)$ and a quadratic operator acting on $(w,w)$. 

To simplify the notation, we introduce the auxiliary cutoff function
\begin{equation}
\label{cutofftichi}
\tilde{\chi}(k):=\omk^{-1}\chi(k)^2.
\end{equation}
It follows from~\eqref{voevenchi} that
\begin{equation}
\label{voevenchitilde}
\tilde{\chi}\in L^p(\mathbb{R}^d),
\qquad
\forall\, p\in(d,\infty].
\end{equation}

The {\it Yukawa bilinear map} is defined  by
 \begin{equation}\label{bilitodefine}
\qquad \cB_{\tilde \chi} \lbrack w_1 , w_2 \rbrack (k,\eta) := \frac{2}{(2 \pi)^d} \int_{\RR^d} \sin (\tilde k \cdot \eta /2 ) \, \tilde \chi(\tilde k) \, \cI w_1(\tilde k,0) \,
\cI w_2 (k-\tilde k, \eta) \, d \tilde k ,
\end{equation}
 where
 \begin{equation}\label{symetrydeIenketa}
 \cI w_j (k,\eta) := \frac{1}{2} \bigl \lbrack w _j(k,\eta) + \bar w_j (-k,-\eta)  \bigr \rbrack , \quad \text{ for } j=1,2.
 \end{equation}
 We show in Proposition \ref{Equivalenceofsystems}
 that all smooth rapidly decreasing solutions $ (u,A) $ to (\ref{PSKG})-(\ref{inidataSKGinif}) give rise to solutions $ w $ to the following \emph{Fourier--Moyal (FM) equation}
\begin{equation}\label{transporteqonw}
\part_t w - (k \cdot \nabla_\eta) w + \cL_t w + \cQ_t (w) = 0 \, ,
 \end{equation}
 completed with the initial data
 \begin{equation}\label{inidatatransporteqonw}
\qquad \quad w(0,\cdot) = w_0 (\cdot) := \frac{1}{(2 \pi)^d} \int_{\RR^d} \! \int_{\RR^d} \! \int_{\RR^d} e^{-ik\cdot x - i \eta \cdot \xi - i \xi \cdot y} \,
u_0 \Bigl( x+ \frac{y}{2} \Bigr) \bar u_0 \Bigl( x- \frac{y}{2} \Bigr) \, dx dy d\xi \, .
 \end{equation}
 Given some initial data $\alpha_0$ which is expressed in terms of $(A_0,A_1)$ as in  \eqref{SKGreinter2ini} and which can be put in the form  
 \begin{equation}\label{def.alpha.tilde.alpha}
  \alpha_0 = \chi \, \tilde \alpha_0 ,
  \end{equation}
   the action $ \cL_t w $ is linear with respect to $ w $. It is given by\footnote{Retain that $ \tilde \chi $ (and not $ \chi $) appears in the definition of $ \cL_t $. The map $ \cL_t $ may be defined by either of the formulas below, depending on the context. The right part is more general because it does not assume a factorization of $ \alpha_0 $ by $ \chi $; the left part is more practical for our future goals.}
 \begin{equation}\label{int.defdecLcQt}
 \ \cL_t w := \cB_{\tilde \chi} \bigl \lbrack f_{\tilde \alpha_0}(t,\cdot) , w  (t,\cdot) \bigr \rbrack = \cB_{\omk^{-1} \chi} \bigl \lbrack f_{\alpha_0}(t,\cdot) , w  (t,\cdot) \bigr \rbrack \, , 
 \end{equation}
with the convention
\begin{equation}\label{int.defdecLcQtcomp} 
f_{\tilde \alpha_0}(t,k,\eta) := (2 \pi)^{d/2}e^{-i \omk t} \, \omk^{1/2} \, \tilde \alpha_0 (k) \,  .
\end{equation}
 On the other hand, the expression $ \cQ_t (w) $ is quadratic in $ w $, defined as
 \begin{equation}\label{ilfautajouter}
\qquad \ \cQ_t (w) := - \int_0^t \cB_{\tilde \chi} \bigl \lbrack g_w(s,t,\cdot), w (t,\cdot) \bigr \rbrack \, ds \, ,
\qquad g_w(s,t,k,\eta) := i e^{i \omk (s-t)} w(s,k,0) \, .
\end{equation}
For clarity, we have to explain the origin and the role of $ \cL_t $ and $ \cQ_t $. In fact, the solution $ A $ to \eqref{PSKG2} can be split into two parts. It can be decomposed into a sum $ A = A_l + A_n $ with $ A_l $ and $ A_n $ defined as indicated below:
\begin{itemize}
  \item The component $ A_l $ is obtained by solving the homogeneous wave equation $ (\Box +1 ) A_l = 0 $ with initial data $ A_0 $ and $ A_1 $ as in \eqref{inidataSKGinif}. After interpretation in terms of $ \alpha $, it generates the function $ f_{\tilde \alpha_0} (t,\cdot) $ which can directly be computed from $ \tilde \alpha_0 $ as in \eqref{int.defdecLcQt}. Thus, at the level of \eqref{transporteqonw}, the impact of $ A_l $ is displayed by the {\it linear} action $ \cL_t w $.
  \item The second component $ A_n $ is extracted by looking at the inhomogeneous wave equation $ (\Box +1 ) A_n = - \cF^{-1} (\chi) * \vert u \vert^2 $ with zero initial datum. After substitution inside \eqref{PSKG1}, it yields the {\it bilinear} integral operator $ \cQ_t (w) $. The presence of $ \cQ_t (w) $ comes from
  the interplay (in general non local)  between the nucleon $ u $ and the meson field $ A_n $. Viewed from the perspective of \eqref{transporteqonw}, these exchanges take the form of a self-interaction  (since they now involve the sole unknown $ w $).
\end{itemize}

\noindent The scalar equation \eqref{transporteqonw} together with \eqref{int.defdecLcQt} and \eqref{ilfautajouter} is self-contained. It is indexed by $ \tilde \chi $ and $ \tilde \alpha_0 $, and therefore it can be denoted by $ \text{FM}^{\tilde \alpha_0}_{\tilde \chi} $, or simply $ \text{FM} $.


\subsubsection{The microlocal self-interaction viewpoint} \label{subsubsec:microself}
\noindent 
The Fourier--Moyal formulation~\eqref{transporteqonw} of the system
\eqref{PSKG} exhibits two distinct features. The first is linear, through the transport operator
\[
\partial_t-k\cdot\nabla_\eta
\]
and the action of the linear operator $\mathcal L_t$. The second aspect is quadratic, appearing through the nonlinear operator $\mathcal Q_t$. The introduction of $\mathcal Q_t$ is advantageous for two reasons. First, it simplifies the nonlinear structure of the problem: the original cubic nonlinearity in $u$ is reduced to a quadratic one in $w$. Second, through the bilinear operator $\mathcal B_{\tilde\chi}$, it provides a finer microlocal description of the nonlinear interactions.

\smallskip

\noindent
By construction, the Fourier--Moyal equation~\eqref{transporteqonw} is derived from~\eqref{PSKG}. The functions $w$ obtained from $u$ via~\eqref{defunknownw} are called \emph{prepared} (see Definition~\ref{Prepareddata}). For prepared data, solutions of~\eqref{transporteqonw} correspond directly to solutions of~\eqref{PSKG}. This naturally raises the question of solving the Cauchy problem~\eqref{inidatatransporteqonw}--\eqref{KGB1intermediaire} for prepared initial data $w_0$ arising from wave functions $u_0$ with lower regularity than is usually assumed.

More importantly, the Fourier--Moyal equation~\eqref{transporteqonw} is a self-contained evolution equation. It therefore makes sense to study its Cauchy problem for arbitrary, possibly unprepared, initial data, thereby providing a natural extension of the original Schr\"odinger--Klein--Gordon dynamics. Motivated by this perspective, in the present work we investigate the well-posedness of the initial value problem~\eqref{transporteqonw}--\eqref{inidatatransporteqonw}.
The answer depends crucially on two ingredients:
\begin{itemize}
\item the choice of the cutoff function $\tilde\chi$ (see \eqref{cutofftichi}), which determines the strength of the ultraviolet regularization;
\item the functional framework, which is dictated by the mapping properties of the bilinear operator $\mathcal B_{\tilde\chi}$.
\end{itemize}


\subsection{Notations and functional setting}\label{subsec:Notationsfunctionalsetting} This paragraph collects the different norms that will appear. Together with Appendix \ref{aboutwiener}, it should serve as a reference point for their definitions.
Let $ \cS (\RR^d ; \CC) $ be  the Schwartz space, whose dual is the space  of tempered distributions $ \cS' (\RR^d ; \CC)$. The Fourier transform $ \cF $ and its inverse $ \cF^{-1} $
are defined by
\begin{subequations}\label{Fourierandinverse}
\begin{eqnarray}
& &  \displaystyle \forall u \in \cS (\RR^d ; \CC) , \qquad \hat u (k) \equiv (\cF u) (k) :=  \int_{\RR^d} u(x) e^{-i k \cdot x} \,
dx , \label{Fourierandinverse1} \\
& & \displaystyle \forall u \in \cS (\RR^d ; \CC) , \qquad  u (x) \equiv (\cF^{-1} \hat u) (x) :=  \frac{1}{(2\pi)^d} \int_{\RR^d} \hat u(k)
e^{i k \cdot x} \, dk . \label{Fourierandinverse2}
\end{eqnarray}
\end{subequations}
A function $ u \in \cS (\RR^d ; \CC) $ can always be decomposed into $ u = \cI u + \cJ u $ with
\begin{equation}\label{defdeJI}
\cI u (x) : = \frac{1}{2} \bigl \lbrack u(x) + \bar u (-x) \bigr \rbrack , \qquad \cJ u (x) : = \frac{1}{2} \bigl \lbrack u(x) -
\bar u (-x) \bigr \rbrack .
\end{equation}
Remark that
\begin{equation}\label{propdecIcJ}
 \cI^2 = \cI , \quad \cJ^2 = \cJ , \quad \cI \cJ = \cJ \cI = 0 , \quad \cI i = i \cJ , \quad \cJ i = i \cI .
 \end{equation}
 Retain that
 \begin{equation}\label{pasaoublier}
 \parallel \alpha \parallel_{L^2}^2 =
 \parallel \cI \alpha \parallel_{L^2}^2 + \parallel \cJ \alpha \parallel_{L^2}^2 .
 \end{equation}
For all $ p \in[1,+\infty] $, the operators $ \cI $ and $ \cJ $ are continuous on $ L^p (\RR^d ; \CC) $ with norm $ 1 $. Observe that $ \cI (\cF u) = \cF u $ when the function $ u(\cdot) $ is real valued.
On the other hand, the condition $ \cI u = u $ is satisfied if and only if $ \bar u (-x) = u(x)$. In what follows, we will often deal with functions $ u $ depending on various variables like $ t $, $ x $, $ k $, $ \xi $ or $ \eta $. Then, operators like $ \cI $, $ \cJ $, $ \cF $, $ \cF^{-1} $ or the convolution $ * $ may be applied with respect to a specific set of
these variables (while the others are viewed as parameters). To clearly specify that their action involves a special group of variables among others, like $ x $, $ k $, $ \xi $ or
both $ k $ and $ \eta $, we will use the notations $ \cF_x $, $ \cF^{-1}_{k} $, $ *_\xi $ or $ \cI_{k,\eta} $.

\smallskip


Let $B$ be a Banach space. Throughout this article, we work with weighted Lebesgue spaces on $\RR^d$ with values in $B$. We shall repeatedly use the Japanese bracket
\begin{equation} \label{Japanesebracket}
\omk^s := \bigl( 1 + \vert k \vert^2 \bigr)^{s/2} , \qquad s \in \RR , \qquad \vert k \vert := \Bigl( \sum_{i=1}^d k_i^2
\Bigr)^{1/2} .
 \end{equation}
 
 Let $s\in\RR$ and $p\in[1,+\infty]$. The weighted Lebesgue space $L_s^p(B)$ consists of all measurable functions
$w:\RR^d\to B$ such that (with the usual modification when $p=\infty$)
\begin{equation} \label{normsobolevcondition}
\parallel w \parallel_{L^p_s (B)} \, := \, \parallel \omk^s \, w(k) \parallel_{L^p (B)} = \Bigl( \int_{\RR^d} \omk^{s p} \parallel w (k) \parallel_B^p dk \Bigr)^{1/p}   < + \infty .
 \end{equation}
 When $s=0$, we simply write
\[
L_0^p(B)=L^p(B).
\]
In the  case $p=2$, viewing $k$ as the Fourier variable dual to $x$, the norm $\|\cdot\|_{L_s^2}$ coincides with the Sobolev norm $\|\cdot\|_{H^s}$ of the inverse Fourier transform, namely
\[
\|w\|_{L_s^2}
=
\|\mathcal F_k^{-1}w\|_{H^s}.
\]
Finally, we adopt the standard notation
\begin{equation}
\label{Lp+-}
\qquad L_{s+}^p(B)
:=
\bigcup_{\tilde s>s}
L_{\tilde s}^p(B),
\qquad
L_{s-}^p(B)
:=
\bigcap_{\tilde s<s}
L_{\tilde s}^p(B) \, , \qquad L_{s+}^p(B) \subsetneq L_{s}^p(B) \subsetneq L_{s-}^p(B) \, .
\end{equation}

  \smallskip
 
 Denote by $\mathcal C^0(B)$ the space of continuous $B$-valued functions on $\mathbb R^d$, equipped with the supremum norm $\|\cdot\|_{L^\infty}$. We use the subscripts $b$ and $c$, as in $  \cC^0_b (B) $ and $ \cC^0_c (B) $, to indicate that the functions are bounded and compactly supported, respectively.

  \smallskip

\noindent Given $ \tilde \imath \in ]0,1] $, let $ \mathscr W_{\tilde \imath} $ be the set of concave moduli of continuity\footnote{This means that $ \omega : \RR_+ \rightarrow
  \RR_+ $ is a concave increasing function vanishing at $ 0 $ and continuous at $ 0 $.} $ \omega $ satisfying
\begin{equation} \label{modconw}
\exists \, (c,\bar z) \in  (\RR_+^*)^2 \, ; \qquad
0 \leq c \, z^{\tilde \imath} \leq \omega (z) \, , \qquad \forall z \in [0,\bar z] \, .
    \end{equation}
The space $ \cC^{\omega}_b (B) $ is the subset of $  \cC^0_b (B) $ which is defined by the norm condition\footnote{When $ w $ is as in \eqref{defunknownwbis} and $ B $ is just $ L^\infty $ (instead of 
$ \cC^{\omega}_b $), we recover the 
Sj\" ostrand's class, see Appendix \ref{aboutwiener}.}
  \begin{equation} \label{defdenormcomegab}
  \parallel w \parallel_{\cC^{\omega}_b (B)} := \parallel w \parallel_{L^\infty (B)} + \, \vert w \! \downharpoonright_\omega \, < + \infty ,
    \end{equation}
 where
  \begin{equation} \label{defdenormcomegabdepare}
\qquad  \vert w \! \downharpoonright_\omega :=
\sup_{\eta_1 \not = \eta_2} \ \frac{\parallel w (\eta_1) - w (\eta_2) \parallel_B}{\omega( \vert \eta_1 - \eta_2 \vert)} =
\sup_{\eta \in \RR^d} \
\sup_{\vert h \vert \not = 0} \ \frac{\parallel w (\eta + h) - w (\eta) \parallel_B}{\omega( \vert h \vert)} \, .
  \end{equation}
  The set $ \cC^{\omega}_b (B) $ is therefore the Banach space \cite{MR244747} containing all bounded uniformly  continuous functions with respect to the modulus of
  continuity $ \omega $.
  The option $ \omega (y) = y^\jmath $ with $ \jmath \in ]0,1] $ describes H\"older continuity
  and, in this case, we  simply note $ \cC^{\jmath}_b (B) $. For $ \jmath = 1 $,
  we recover the set $ \cC^1_b (B) \equiv \text{Lip} (B) $ of bounded  Lipschitz continuous functions. The endpoint $ \jmath = 0 $ is excluded because the constant function $ \omega = s^0 \equiv 1 $ is not an admissible modulus of continuity\footnote{For this reason, extending the notation $ \parallel \cdot \parallel_{C^\jmath_b(B)}$ to the case $ \jmath = 0 $ would be misleading. We therefore use the notation $  \parallel \cdot \parallel_{L^\infty(B)}$ instead.}. Other choices for $ \omega $ are of course possible, for instance to
  express the Dini continuity.
Briefly, given a modulus of continuity $ \omega $, the exponents $ \imath $ and $ \jmath $ are respectively related to a minoration (by $ z^{\tilde \imath} $) and a majoration (by $ z^\jmath $) of $ \omega $.

\begin{rem}[About the choice of $ \tilde \imath $]
\label{remtildeimath}
Since every modulus of continuity is concave, the condition~\eqref{modconw} is automatically satisfied with $ \tilde \imath =1 $. Moreover, the inequality
$ \tilde \omega \leq \omega $ implies the continuous embedding $ \cC_b^{\tilde\omega}(B)\hookrightarrow \cC_b^{\omega}(B) $. Hence, given $ w\in \cC_b^{\tilde\omega}(B) $ and any exponent
$ \tilde \imath\in ]0,1[ $, we may replace $ \tilde\omega $ by the smallest concave modulus of continuity $ \omega $ dominating both $  z\mapsto \tilde\omega(z) $ and $ z\mapsto z^{\tilde\imath} $. 
Then $ w\in \cC_b^{\omega}(B) $, and the modulus $ \omega $ satisfies~\eqref{modconw}.
\end{rem}

\noindent    When $ B = \CC $, we omit
  mention of $ B $, so that $ L^p_s \equiv L^p_s (\CC) $ and $ \cC^{\omega}_b \equiv \cC^{\omega}_b (\CC) $.

  \smallskip  

We work with functions $ w (k,\eta) $ defined on the phase space $ \RR^d \times \RR^d $, that is depending on both variables $ k \in \RR^d $ and $ \eta \in \RR^d $. Given $ k \in \RR^d $, we measure the regularity of $ w(k,\cdot) $ with respect to $ \eta $ in terms of $ \cC^\omega_b $. Then, we want to control separately through  $ L^p $-norms the integrability of the  functions $ k \mapsto \parallel w(k,\cdot) \parallel_{L^\infty} $ and $ k \mapsto \vert w (k,\cdot)
  \! \downharpoonright_\omega  $. To this end, we implement $ L^p $-norms of  weighted versions of \eqref{defdenormcomegab}, where the supnorm and the seminorm are multiplied by $ \omk^{s+\imath} $ and $ \omk^\imath $. With this in mind, define the space $ L^{p}_{s,\imath} (\cC^\omega_b) $ by the norm condition
\begin{equation} \label{calculnormcontinuouscondition}
\begin{array}{rl}
\parallel w \parallel_{L^{p}_{s,\imath} (\cC^\omega_b)} \! \! \! & := \, \parallel k \mapsto \omk^{\imath} \parallel w (k,\cdot) \parallel_{L^\infty} + \, \vert w (k,\cdot)
  \! \downharpoonright_\omega \parallel_{L^p_s} \smallskip \\
  \ & \ = \, \parallel k \mapsto \omk^{s+\imath} \parallel w (k,\cdot) \parallel_{L^\infty} + \, \omk^s \ \vert w (k,\cdot)
  \! \downharpoonright_\omega \parallel_{L^p} \, < + \infty \,  .
  \end{array}
 \end{equation}
 With this convention, the space $ L^p_{s,0} ( \cC^{\omega}_b) $ coincides with $ L^p_s (B) $ where $ B = \cC^{\omega}_b $, with exactly the same norm
 $$ \parallel w \parallel_{L^{p}_{s,0} (\cC^\omega_b)} = \parallel w \parallel_{L^{p}_{s} (\cC^\omega_b)} . $$
 Still, for $ \imath \not = 0 $, the space $ L^p_{s,\imath} ( \cC^{\omega}_b) $ is different from $ L^p_{s} ( \cC^{\omega}_b) $.
In the sequel, it is understood that the spaces $ \cC^{\omega}_b $ and $ L^p $ refer, unless otherwise specified, to functions of the variables $\eta$ and $k$.
 It is clear that
\begin{equation} \label{clearthat}
\parallel w \parallel_{L^{p}_{s,\imath} (\cC^\omega_b)} \leq
\parallel w \parallel_{L^{p}_{\tilde s,\tilde \imath} (\cC^\omega_b)} \, , \qquad s \leq \tilde s \, , \quad \imath \leq \tilde \imath \, . \qquad
 \end{equation}
But we also have
\begin{equation} \label{clearthatbis}
\quad \ \parallel w \parallel_{L^{p}_{s,\imath} (\cC^\omega_b)} \leq
\parallel w \parallel_{L^{p}_{\tilde s,\tilde \imath} (\cC^\omega_b)} \, , \qquad
s \leq \tilde s \, , \quad s + \imath = \tilde s + \tilde \imath  \, .
 \end{equation}
In contrast, for $ \cC^0_b $ the situation is different, since
    \begin{equation} \label{defdenormcomegabimathaajouter}
  \parallel w \parallel_{L^p_{s+\imath} (L^\infty)} \leq \, \parallel w \parallel_{L^p_{s,\imath} (\cC^{\omega}_{b})} , \qquad \forall (s,\imath) \in \RR^2 \, .
  \end{equation}

Let $ w $ be as  in \eqref{defunknownw}. By the definition of the Fourier--Wigner transform, multiplication of $w$ by $\omk$ corresponds to differentiating $u$ with respect to the spatial variable $x$. Consequently, for $\imath\in\RR_+^*$, the condition
\[
w\in L^p_{s,\imath}(\cC_b^\omega)
\]
imposes an additional degree of spatial regularity on $u$ compared to the weaker assumption
\[
w\in L^p_{s,0}(\cC_b^\omega).
\]
On the other hand, differentiation of $w$ with respect to the phase-space variable $\eta$ also corresponds to differentiation of $u$ with respect to $x$. It is therefore natural to measure higher regularity by combining derivatives in $\eta$ with powers of the Fourier weight $\omk$. For $l\in\NN$, we thus introduce the higher-regularity counterpart of \eqref{calculnormcontinuouscondition},
 \begin{equation} \label{calculnormcontinuousconditionregular}
\parallel w \parallel_{W^{l,p}_{s,\imath} (\cC^{\omega}_b) } := \sum_{\vert \alpha \vert \leq l} \parallel \omk^{l-\vert \alpha \vert} \,
\part^\alpha_\eta w (k,\eta) \parallel_{L^p_{s,\imath} (\cC^{\omega}_b) } .
 \end{equation}


\section{Fourier--Moyal equation} \label{sec:KGSsystemrevisited}

In Subsection~\ref{subsec:Presentationoftheequations}, we review the state of the art on Schr\"odinger--Klein--Gordon systems and place our work in the context of the existing literature. This also provides the opportunity to highlight its connections with previous results. In Subsection~\ref{subsec:TransitionfromtheSKGsystemtotheSKGSequation}, we give a detailed derivation of the Fourier--Moyal equation~\eqref{transporteqonw} starting from the Schr\"odinger--Klein--Gordon system~\eqref{PSKG}. Subsection~\ref{subsec:Conservedquantities} is devoted to the conserved quantities associated with~\eqref{transporteqonw}.

\noindent 
Throughout this section, we fix an ultraviolet cutoff
\[
\chi\in \cC_c^\infty(\RR^d; \RR),
\]
and do not repeat this assumption thereafter. It is introduced here to ensure that the corresponding solutions $ u $ and $ A $ are smooth and rapidly decreasing.


\subsection{Schr\" odinger--Klein--Gordon system, a panorama}\label{subsec:Presentationoftheequations}
There are various ways to deal with the initial value
 problem (\ref{PSKG})-(\ref{inidataSKGinif}), which mainly depend on the functional framework that is selected and on the methods
that are involved.

A first series of advances concerning regular solutions to SKG has been produced in the late 1970s by Baillon-Chadam and Fukuda-Tsutsumi, see \cite{MR519635,MR515899}
and references therein. They rely on classical $ L^2 $-based energy estimates and  Galerkin
methods.

Another typical result relying  on the regularization properties of the Schr\"odinger and wave operators has been obtained by A.~Bachelot in
\cite[Theorem 3]{MR778979}. Given
\begin{equation}\label{regularinidata}
u_0 \in H^1(\RR^3;\CC) , \qquad A_0 \in H^1(\RR^3;\RR) , \qquad A_1 \in L^2(\RR^3;\RR) ,
\end{equation}
there is a unique global solution to SKG  such that
\begin{equation}\label{regularpourinstantt}
u \in L^\infty \bigl(\RR ; H^1(\RR^3;\CC) \bigr) , \qquad A \in L^\infty \bigl(\RR ; H^1(\RR^3;\RR) \bigr) .
\end{equation}
To construct solutions as in (\ref{regularpourinstantt}), A.~Bachelot develops in \cite{MR778979} some adequate iterative scheme, and then he exploits  compactness arguments.

Today, it is well known \cite{MR2237673,MR2035502} that the existence and uniqueness may be granted for less regular data than in
 (\ref{regularinidata}). Progress in this direction has been made from the 2000's, by following a scheme introduced by J. Bourgain
 \cite{MR1616917} and by using bilinear Strichartz estimates. More recent contributions \cite{MR3128940,MR3725728,MR2852202}
 further examine certain aspects of (\ref{PSKG}) or explore the case of similar systems, always in $ L^2 $-based Sobolev contexts.

In the spirit of~\cite{MR3737034}, an alternative approach to the Schr\"odinger--Klein--Gordon system consists in considering the modified model $\mathrm{SKG}_{\chi}$. When $\chi$ is compactly supported, $\mathrm{SKG}_{\chi}$ may be viewed as a frequency-truncated approximation of SKG that is compatible with the conservation of the regularized total energy $\mathscr E_\chi$. More generally, allowing $\chi$ to be merely bounded leads to a natural extension of the classical SKG system.

Furthermore, the microlocal formulation developed in this work makes it possible to investigate the well-posedness of~\eqref{PSKG}--\eqref{inidataSKGinif} in function spaces beyond the classical Sobolev framework, such as modulation spaces (Appendix~\ref{aboutmodulation}), Wiener amalgam spaces (Appendix~\ref{aboutwiener}), and related spaces. Although this perspective has become well established for nonlinear Schr\"odinger equations~\cite{MR4027018,Chaichenets}, to the best of our knowledge it has not previously been explored for the Schr\"odinger--Klein--Gordon system, nor for its generalized version $\mathrm{SKG}_{\chi}$.

Finally, it is worth emphasizing that weak $L^2$-solutions $w$ of the Fourier--Moyal equation~\eqref{transporteqonw} naturally extend the notion of weak $L^2$-solutions $u$ of~\eqref{PSKG1}; see Lemma~\ref{Propertiesofprepareddata}.

\smallskip


\subsection{Derivation of the Fourier--Moyal equation}\label{subsec:TransitionfromtheSKGsystemtotheSKGSequation}
In Paragraph \ref{nucleonfield}, we first focus on the nucleon field $ u $. In
Paragraph \ref{mesonfield}, we explain the origin of the complex meson field $ \alpha $. In
Paragraph \ref{Yukawabilinearmapdef}, we present the Yukawa bilinear map. This yields in Paragraph
\ref{KGSsystemrevisited} the
Moyal--Klein--Gordon system and then, in Paragraph
\ref{SKGSequationenfin}, the Fourier--Moyal equation. 


\subsubsection{Microlocal formulation  of the nucleon equation}\label{nucleonfield} The Wigner function
that is associated at time $ t $ with $ u (t,\cdot) $ is just (see Paragraph \ref{aboutwigner} for the detailed definition
of $ \mathbb W $ and $ {\rm W} $)
\begin{equation}\label{wignerforv}
v \bigl( t,x,\xi ) := \mathbb W \lbrack u(t,\cdot) , u(t,\cdot) \rbrack (x,\xi) = {\rm W} \bigl( u(t,\cdot) \bigr) (x,\xi) .
\end{equation}
The function $ v (\cdot) $ furnishes the quasi-probability distribution
that is induced by $ u(t,\cdot) $ in the phase space $ \RR^d \times \RR^d $. The Wigner transform $ {\rm W} : u \longmapsto {\rm W}
(u) = \mathbb W [u,u] $ maps $ \cS (\RR^d;\CC) $ into $ \cS (\RR^d \times \RR^d;\CC) $. It is injective but it is not surjective.

\begin{defi}[$ B $-prepared data] \label{Prepareddata}
Let $ B = \cS (\RR^d;\CC) $ or $ B = L^p (\RR^d;\CC) $ with $ 1 \leq p \leq + \infty $.
We say that a function $ w (k,\eta) $ is $ B $-{\it prepared} when it belongs to the subset $ \cF_{x,\xi} \circ {\rm W} (B) $.
\end{defi}

\noindent As explained in \cite{MR2502626}, the real valued function $ v(\cdot) $ given by (\ref{wignerforv}),
with $ u(\cdot) $ satisfying (\ref{PSKG1}), solves the transport equation \eqref{KGB1} where the action of the operator $ \cK : \cS (\RR^d ) \rightarrow \cC^\infty_b (\RR^d \times \RR^d )$ is determined by
\begin{equation}\label{formulapourOp}
\left. \begin{array} {rcl}
\cK(A) (x,\xi) & := & \displaystyle \frac{i}{(2 \pi)^d} \int_{\RR^d} \Bigl \lbrack A \Bigl( x + \frac{y}{2} \Bigr) - A \Bigl( x - \frac{y}{2} \Bigr)
\Bigr \rbrack e^{-i \xi \cdot y} \, dy , \\
\ & = & \displaystyle \frac{i}{\pi^d} \Bigl \lbrack \hat A (2\xi) e^{2 i \xi \cdot x} - \hat A (-2\xi) e^{- 2 i \xi \cdot x} \Bigr \rbrack .
 \end{array} \right.
\end{equation}
According to (\ref{defunknownw}) and (\ref{wignerforv}), we have
\begin{equation} \label{introdew}
w (t,k,\eta) = (2 \pi)^{-d} \, \cF_{x,\xi} \bigl( v (t,\cdot) \bigr) (k,\eta) .
\end{equation}
By construction, the function $ w(t,\cdot) $ is $ \cS $-prepared. In particular, the same applies to $ w_0 $. Moreover, in view of the definition \eqref{defunknownw}, we have
\begin{equation}\label{paritedew}
w (t,k,\eta) = \bar w (t,-k,-\eta) , \qquad
w (t,k,\eta) = \cI_{k,\eta} w (t,k,\eta) .
\end{equation}
This means that the real (resp. imaginary) part of $ w (t,\cdot) $ (and of $ w_0 $) is an even (resp. odd) function on
$ \RR^d \times \RR^d $. Apply the Fourier transform in both variables $ x $ and $ \xi $  to (\ref{KGB1}) to get \eqref{KGB1intermediaire}
where we can compute
\begin{equation}\label{computeFKA}
\cF_{\! x,\xi} \bigl\lbrack \cK (\chi(D) A \bigr) \bigr\rbrack (t,k,\eta) = 2 \, \sin (k \cdot \eta /2 ) \, \chi(k) \, \cF_x \bigl( A(t,\cdot) \bigr) (k) .
\end{equation}
Compared with the wave function $ u(t,x) $, the Fourier--Wigner transform $ w(t,k,\eta) $ is expected to encode finer phase-space information through~\eqref{KGB1intermediaire}. In particular, many estimates and localization arguments formulated in terms of the phase-space variables $ (k,\eta) $ have no counterpart when working solely with the physical variable $ x $, and therefore provide a more refined analytical framework.


\subsubsection{Complex meson fields}\label{mesonfield}
In view of (\ref{recallwigff}), the equation (\ref{PSKG2}) is the same as
\begin{equation}\label{KGB2}
(\Box +1 ) A = - (2 \pi)^{-d} \, \cF^{-1} (\chi) * \Bigl( \int_{\RR^d} v(t,\cdot,\xi) d \xi \Bigr) .
\end{equation}
Now, introduce the complex field $ \alpha $ given by \eqref{introdealpha}.
Since $ A(\cdot) $ and $ \part_t A(\cdot) $ are real valued functions, with $ \cI $ and $ \cJ $ as in (\ref{defdeJI}), we have
\begin{subequations}\label{llllllllll}
\begin{eqnarray}
& & \displaystyle \cI_k \alpha (t,k) = (2 \pi)^{-d/2} \omk^{1/2} \cF_x \bigl( A(t,\cdot) \bigr) (k) , \label{llllllllll1} \\
& & \displaystyle \cJ_k\alpha (t,k) = i (2 \pi)^{-d/2} \omk^{-1/2} \cF_x \bigl( \part_t A(t,\cdot) \bigr) (k) . \label{llllllllll2}
\end{eqnarray}
\end{subequations}

\noindent From \eqref{pasaoublier} and Plancherel theorem, we can deduce that
\begin{equation}\label{pasaoublierbis}
\begin{array}{rl}
 \parallel \omk^{1/2} \alpha (t,\cdot) \parallel_{L^2}^2 \! \! \! & \displaystyle = \,
 \parallel \omk^{1/2} \cI_k \alpha (t,\cdot) \parallel_{L^2}^2 + \parallel \omk^{1/2} \cJ_k \alpha (t,\cdot) \parallel_{L^2}^2 \smallskip \\
 & \displaystyle = \, (2 \pi)^{-d} \parallel
  \omk \cF_x \bigl( A(t,\cdot) \bigr) \parallel_{L^2}^2 + (2 \pi)^{-d} \parallel
  \cF_x \bigl( \part_t A(t,\cdot) \bigr)  \parallel_{L^2}^2 \smallskip \\
 & \displaystyle = \, \parallel A(t,\cdot)  \parallel_{L^2}^2 + \parallel \nabla_x A(t,\cdot)  \parallel_{L^2}^2 + \parallel \part_t A(t,\cdot) \parallel_{L^2}^2 .
 \end{array}
 \end{equation}
The last line of \eqref{pasaoublierbis} appears in the definition of the total energy, see   \eqref{eq.skg.energytr}.
The factor $ (2 \pi)^{-d/2} $ in \eqref{introdealpha} is  calibrated in order to make this (part of the) total energy coincide with the square of the $ L^2_{1/2} $-norm of $ \alpha $.

\noindent Of course, we can recover $ A (t,\cdot) $ and $ \part_t A (t,\cdot) $ (resp. $ A_0 $ and $ A_1 $) from $ \alpha (t,\cdot) $
(resp. $ \alpha_0 $) through
\begin{subequations}\label{recoveralpha0}
\begin{eqnarray}
& & \displaystyle \quad A (t,\cdot) = (2 \pi)^{d/2} \cF_k^{-1} \bigl( \omk^{-1/2} \cI \alpha (t,\cdot) \bigr) , \qquad \quad \ \, A_0 = (2 \pi)^{d/2} \cF_k^{-1} \bigl( \omk^{-1/2}
\cI \alpha_0 \bigr) , \label{recoveralpha01} \\
& & \displaystyle \quad \part_t A (t,\cdot) = - i (2 \pi)^{d/2} \cF_k^{-1} \bigl( \omk^{1/2} \cJ \alpha (t,\cdot ) \bigr) , \qquad A_1  = - i (2 \pi)^{d/2} \cF_k^{-1} \bigl( \omk^{1/2} \cJ \alpha_0 \bigr) .
\label{recoveralpha02}
\end{eqnarray}
\end{subequations}

\noindent Apply $ \part_t $ to (\ref{llllllllll1}) and compare the resulting relation to (\ref{llllllllll2}) in order to get
\begin{equation} \label{eqourcIalpha}
\part_t \cI_k\alpha = - i \omk \cJ_k\alpha .
\end{equation}
On the other hand, using (\ref{KGB2}) and (\ref{llllllllll2}), we find
\begin{equation} \label{eqourcJalpha1}
\begin{array}{rl}
\part_t \cJ_k\alpha = \! \! \! & \displaystyle + i (2 \pi)^{-d/2} \omk^{-1/2} \cF_x \bigl \lbrack (\Delta -1)  A(t,\cdot) \bigr \rbrack \\
\ & \displaystyle - i \, (2 \pi)^{-3d/2} \, \omk^{-1/2} \, \chi(k) \, \cF_x \Bigl( \int_{\RR^d}
v(t, \cdot ,\xi) d \xi \Bigr) .
\end{array}
\end{equation}
Then, from (\ref{llllllllll1}) and (\ref{introdew}), we obtain
\begin{equation} \label{eqourcJalpha2}
\part_t \cJ_k\alpha = - i \omk \cI_k\alpha - i (2 \pi)^{-d/2} \omk^{-1/2} \, \chi (k) \, w(t,k,0) .
\end{equation}
By adding (\ref{eqourcIalpha}) and (\ref{eqourcJalpha2}), since $ \alpha = \cI_k\alpha + \cJ_k\alpha $, we recover
\begin{equation}\label{SKGSyequi}
\partial_t \alpha + i \omk \alpha + i (2 \pi)^{-d/2}  \omk^{-1/2} \, \chi (k) \, w(t,k,0) = 0 ,
\end{equation}
which is the same as \eqref{SKGSyequibis}.
In this way, we rewrite the Klein--Gordon equation~\eqref{PSKG2} in terms of the complex field $\alpha$. This transformation is classical (see, e.g.,~\cite{MR3737034}). We adopt this formulation because the integral representation of $\alpha$ will be used repeatedly throughout the paper. Indeed, unlike the original Klein--Gordon equation~\eqref{PSKG2}, the equivalent equation~\eqref{SKGSyequi} is a non-homogeneous linear equation and therefore admits the explicit representation
\begin{equation}\label{integrateddealpha}
\alpha(t,k)
=
e^{-i\omk t}\,\alpha_0(k)
-i(2\pi)^{-d/2}
\int_0^t
e^{i\omk(s-t)}
\omk^{-1/2}\,
\chi(k)\,
w(s,k,0)\,ds.
\end{equation}


\subsubsection{The Yukawa bilinear map}\label{Yukawabilinearmapdef}
In view of \eqref{computeFKA}
and \eqref{llllllllll1}, because $ w = \cI w $, we have
\begin{equation}\label{defedeqy}
(2 \pi)^{-d} \cF_{\! x,\xi} \bigl\lbrack \cK \bigl( \chi(D) A \bigr) \bigr\rbrack *_{k} w = (2 \pi)^{d/2} \, \cB_\chi [ \omk^{-1/2} \alpha , w ] .
\end{equation} 
By this way, the  Yukawa Bilinear map 
(or $ \cY \cB $-map in abbreviated form) of \eqref{bilitodefine} comes into play. It is easy to check that
\begin{equation}\label{bilitodefine0}
\begin{array} {rcl}
\cB_{\chi} \, : \, \cS (\RR^d \times \RR^d; \CC) \times \cS (\RR^d \times \RR^d; \CC) & \longrightarrow & \cS (\RR^d \times \RR^d; \CC) \\
(w_1,w_2) & \longmapsto & \cB_{\chi} [ w_1, w_2 ] .
\end{array}
\end{equation}
Observe that
\begin{equation}\label{paritedewforB}
\cB_{\chi} [ w_1, w_2 ] = \cB_{\chi} [ \cI w_1, \cI w_2 ]  = \cI \bigl( \cB_{ \chi} [ w_1, w_2 ] \bigr) .
\end{equation}
The evaluation of the bilinear operator $\cB_{\chi}[w_1,w_2]$ requires the trace of $w_1$ at $\eta=0$. Consequently, some minimal regularity of $w_1$ with respect to the variable $\eta$ is needed in a neighborhood of the origin. The most natural assumption is continuity. Recalling the convention~\eqref{normsobolevcondition} together with~\eqref{calculnormcontinuouscondition}, one may choose either the space $L_k^p(\cC^0_{b,\eta})$ or $\cC^0_{b,\eta}(L_k^p)$. We adopt the former, since it is better suited to the propagation properties of the transport operator
\[
\partial_t+k\cdot\nabla_\eta.
\]
Accordingly, throughout the sequel we assume
\begin{equation}\label{regularitydewcritique}
w \in L_k^p(\cC^0_{b,\eta})
\hookrightarrow
L_k^p(L_\eta^\infty),
\qquad
1\le p\le+\infty,
\end{equation}
where the continuous embedding follows from~\eqref{calculnormcontinuouscondition}.

\begin{lem}[Basic continuity properties of the $ \cY \cB $-map] \label{continuityYBbasic}
 Fix indices $ 1\leq p, \tilde p, \tilde q , r
\leq + \infty $  such that $ 1/p + 1/\tilde p = 1 + 1/ r $.
The $ \cY \cB $-map  extends to a continuous bilinear map
\[ \cB_{\chi} : L^p_k(\cC^0_{b,\eta}) \times L^{\tilde p}_k (L^{\tilde q}_\eta) \longrightarrow L^r_k (L^{\tilde q}_\eta) . \]
More precisely, we have
\begin{equation}\label{continuityYBbasicestimates}
\parallel \cB_{\chi}  [ w_1, w_2 ] \parallel_{L^r_k (L^{\tilde q}_\eta)} \, \leq \frac{2}{(2 \pi)^d} \, \parallel {\chi} \parallel_{L^\infty} \, \parallel w_1 \parallel_{L^p_k(L^{\infty}_\eta)} \, \parallel w_2
\parallel_{L^{\tilde p}_k (L^{\tilde q}_\eta)} .
\end{equation}
\end{lem}

\begin{proof}   Indeed, one readily observes that
\[ \parallel \cB_{\chi} \lbrack w_1, w_2 \rbrack (k,\cdot) \parallel_{L^{\tilde q}_\eta} \leq \frac{2}{(2 \pi)^d} \ \parallel {\chi} \parallel_{L^\infty} \int_{\RR^d}
\parallel \cI_{k,\eta} w_1(\tilde k,\cdot) \parallel_{L^\infty_\eta} \, \parallel \cI_{k,\eta} w_2 (k-\tilde k, \eta) \parallel_{L^{\tilde q}_\eta} \,
d \tilde k . \]
The estimate \eqref{continuityYBbasicestimates} then follows directly from Young's convolution inequality.
\end{proof}

\noindent  
In particular, consider the case where $w_1$ is independent of $\eta$, namely
\[
w_1(k,\eta)=\alpha(k),
\]
for some non-zero function $\alpha\in \cC^0_c(\RR^d)$. Then $w_1$ does not belong to
$\cS(\RR^d\times\RR^d;\CC)$, since it exhibits no decay in the $\eta$-variable. Nevertheless,
\[
w_1=\alpha \in L^p_k(\cC^0_{b,\eta}).
\]
Therefore, by the preceding Lemma~\ref{continuityYBbasic}, the bilinear operator
$\cB_\chi\lbrack w_1,w_2\rbrack$ is still well defined for $w_1=\alpha$ and for every $w_2\in L^{\tilde p}_k(L^{\tilde q}_\eta)$.


\subsubsection{The Moyal--Klein--Gordon system}\label{KGSsystemrevisited} The two equations (\ref{KGB1intermediaire}) and (\ref{SKGSyequi}) furnish the self-contained Moyal--Klein--Gordon  coupled system
\begin{subequations}\label{SKGrevisited}
\begin{eqnarray}
& & \displaystyle \part_t w - (k \cdot \nabla_\eta) w + (2 \pi)^{d/2} \cB_{\chi} \lbrack \omk^{-1/2} \alpha, w \rbrack = 0 , \label{SKGrevisited1} \\
& & \displaystyle \partial_t \alpha + i \omk \alpha + i (2 \pi)^{-d/2}  \omk^{-1/2} \, \chi (k) \, w(t,k,0) = 0 , \label{SKGrevisited2}
\end{eqnarray}
\end{subequations}
with initial data issued from (\ref{inidatatransporteqonw}) and (\ref{SKGreinter2ini}), that is
\begin{equation}\label{inidataSKGrevisited}
w(0,\cdot) = w_0 (\cdot) , \qquad \alpha(0,\cdot) = \alpha_0 (\cdot).
 \end{equation}

\begin{lem}[$ \cS $-prepared data are propagated] \label{preparedpropagation} Select any initial data $ (w_0,\alpha_0) $ which is such that
$$ (w_0,\alpha_0) \in \cS (\RR^d \times \RR^d ; \CC)
\times \cS (\RR^d ; \CC) \, , \qquad \exists \, u_0 \in
\cS (\RR^d ; \CC) \, ; \quad w_0 = (2 \pi)^{-d} \cF_{x,\xi} \circ {\rm W} (u_0) \, . $$
Then, there exists a unique solution  $ (w,\alpha) $ to (\ref{SKGrevisited})-(\ref{inidataSKGrevisited}) in $ L^\infty \bigl( \RR_+ ;\cS (\RR^d \times \RR^d ; \CC)
\times \cS (\RR^d ; \CC) \bigr) $. Moreover, we  have
\begin{equation}\label{preparedpropagationattimet}
\forall t \in \RR , \qquad \exists u (t,\cdot) \in \cS (\RR^d ; \CC) \, ; \qquad w(t,\cdot) = (2 \pi)^{-d}\cF_{x,\xi} \circ {\rm W} \bigl( u (t,\cdot) \bigr)  .
\end{equation}
In particular,  $ w $ satisfies the relation $ w = \cI w $ in (\ref{paritedew}).
\end{lem}

\begin{proof} Two preliminary remarks which follow from straightforward arguments (that can be found in \cite{MR778979,MR519635,MR515899}).
 First, the Schwartz space regularity is propagated  by both equations (\ref{PSKG})
and (\ref{SKGrevisited})\footnote{The regularity $ \chi \in \cC^\infty_b (\RR^d; \RR) $ is needed to ensure that  \eqref{SKGrevisited2} preserves the Schwartz space $ \cS (\RR^d ; \CC) $}. Secondly, in the class of smooth rapidly decreasing functions, we have uniqueness for both Cauchy problems
(\ref{PSKG})-(\ref{inidataSKGinif}) and (\ref{SKGrevisited})-(\ref{inidataSKGrevisited}).  On the other hand taking a solution $ (u,A) $ to the Cauchy problem (\ref{PSKG})-(\ref{inidataSKGinif}), in view of the preceding construction, the two expressions
(\ref{defunknownw}) and (\ref{introdealpha}) yield
a solution $ (w,\alpha) $ to  (\ref{SKGrevisited})-(\ref{inidataSKGrevisited}) which, by uniqueness, must coincide with the solution $ (w,\alpha) $
to (\ref{SKGrevisited})-(\ref{inidataSKGrevisited}). It follows that the component $ w $ is
$ \cS $-prepared for all time $ t $, as indicated in \eqref{preparedpropagationattimet}. Moreover, since
$ w = (2 \pi)^{-d}\cF_{x,\xi} \circ {\rm W}(u) $, we also have $ w = \cI_{k,\eta} w $.
\end{proof}

\begin{prop}[Equivalence of systems] \label{Equivalenceofsystems} In the context of smooth rapidly decreasing  functions, it is  equivalent to work with (\ref{PSKG})-(\ref{inidataSKGinif}) or with its  microlocal formulation (\ref{SKGrevisited})-(\ref{inidataSKGrevisited}), where $ w_0 $ is the $ \cS $-prepared data coming from $ u_0 $ through \eqref{inidatatransporteqonw} and where $ \alpha_0 $ is given by \eqref{SKGreinter2ini}.
\end{prop}

\begin{proof} We have already explained how to pass from (\ref{PSKG})-(\ref{inidataSKGinif}) to (\ref{SKGrevisited})-(\ref{inidataSKGrevisited}).
We have now to go in the opposite direction. First, consider the passage from (\ref{SKGrevisited2}) to (\ref{PSKG2}). By Lemma \ref{preparedpropagation}, we
know already that $ w $ is $ \cS $-prepared, and therefore we can exploit (\ref{paritedew}) in the form
\[ w(t,k,0) = \cI_k w (t,k,0) , \qquad \cJ_k w (t,k,0) = 0 . \]
Then, from (\ref{SKGrevisited2}), using (\ref{propdecIcJ}), we can extract (\ref{eqourcIalpha}) as well as (\ref{eqourcJalpha2}).
Now, in accordance with (\ref{llllllllll1}), define
\begin{equation}\label{BBfromaa}
A (t,x) := \frac{1}{(2 \pi)^{d/2}} \int_{\RR^d} \frac{ \alpha (t,k) e^{ik \cdot x} + \bar \alpha (t,k) e^{-ik \cdot x}}{ 2 \omk^{1/2}} \ dk .
\end{equation}
It follows from (\ref{eqourcIalpha}) and (\ref{BBfromaa}) that
\[ \part_t A = - i (2 \pi)^{d/2} \cF^{-1}_k (\omk^{1/2} \cJ_k\alpha) , \]
and then, with (\ref{eqourcJalpha2}), we have
\[ \part^2_{tt} A = - (2 \pi)^{d/2} \cF^{-1}_k (\omk^{3/2} \cI_k\alpha) - \cF^{-1}_k \bigl( \chi (k) w(t,k,0) \bigr) , \]
which furnishes  (\ref{PSKG2}). In the  last equality, $|u|^2$ is recovered from $w$ according to \eqref{recallwigff}.

\smallskip

\noindent Secondly, consider the passage from (\ref{SKGrevisited1}) to (\ref{PSKG1}). Again, knowing that $ w(t,\cdot) $ is $ \cS $-prepared, that is $w= (2 \pi)^{-d}\cF_{x,\xi} \circ {\rm W} \bigl( u (t,\cdot) \bigr) $,
we can recover the kernel $u(t,x) \bar u(t,y)$ from $ w(t,\cdot) $ through the inversion formula (\ref{recallwig}). Hence, up to a phase factor reflecting the $U(1)$-symmetry of the Schr\"odinger equation, we find that there exists a   solution $\tilde u(t,\cdot)$
of the equation (\ref{PSKG1}) which is such that $w= (2 \pi)^{-d}\cF_{x,\xi} \circ {\rm W} \bigl( \tilde u (t,\cdot) \bigr) $.
\end{proof}


\subsubsection{The Fourier--Moyal equation and its integral representation}\label{SKGSequationenfin} Now, we can exploit (\ref{integrateddealpha}) to go further
in the reduction. To this end, plug (\ref{integrateddealpha}) into \eqref{SKGrevisited1}.
Knowing that $ \alpha_0 = \chi \, \tilde \alpha_0 $, with $ f_{\tilde \alpha_0} $ and $ g_w $ as in  \eqref{int.defdecLcQt} and  \eqref{ilfautajouter}, this furnishes
$$ \part_t w - (k \cdot \nabla_\eta) w +  \cB_\chi \Bigl \lbrack \omk^{-1} \, \chi(k) \Bigl( f_{\tilde \alpha_0}(t,k,\eta) - \int_0^t g_w (s,k,0) \, ds \Bigr), w \Bigr \rbrack = 0 \, . $$
By this way, recalling that $ \tilde \chi = \omk^{-1} \, \chi(k)^2 $, we recover the equation  (\ref{transporteqonw}) with $ \cL_t $ and $ \cQ_t $ defined as in \eqref{int.defdecLcQt} and  \eqref{ilfautajouter}.

\smallskip

\noindent Fix some $ \cT \in \RR_+^* $. Duhamel's principle implies that $ w $ can be interpreted
as a solution to the Fourier--Moyal equation (\ref{transporteqonw})-(\ref{inidatatransporteqonw}) on $ [0,\cT] $ when
\begin{equation}\label{integralversionSKGS}
w(t,k,\eta) = \Phi (w) (t,k,\eta) , \qquad \forall t \in [0,\cT] ,
\end{equation}
where
\begin{equation}\label{integralversionSKGSPhi}
\qquad \Phi (w) (t,k,\eta) := w_0(k,\eta+ t k)  - \int_0^t \bigl \lbrack \cL_s w + \cQ_s (w) \bigr \rbrack \bigl(  k , \eta + (t-s) k \bigr) \, ds .
\end{equation}
The relation (\ref{integralversionSKGS}) completed by (\ref{integralversionSKGSPhi}) is  the integral version of the
Fourier--Moyal equation.


\subsection{Conserved quantities}\label{subsec:Conservedquantities}
In this subsection, we discuss about conserved quantities for the Fourier--Moyal equation.
In Paragraph \ref{subsec:Yukawaquadraticmap}, we start by exhibiting algebraic identities (useful throughout the text) satisfied by the $ \cY \cB $-map. In Paragraph \ref{consquan}, we highlight quantities (expressed in terms of $ w $ and $ \alpha $) that are invariant under the flow generated by \eqref{SKGrevisited}.
In Paragraph \ref{subsubsec:interpreattion},
we further compare these expressions with $ \cE_1 (t) $.


\subsubsection{Remarkable identities for the $ \cY \cB $-map}\label{subsec:Yukawaquadraticmap} 
We first underline some specific features of the Yukawa bilinear map 
$ \cB_{\tilde \chi} $. Recall that $ \tilde \chi $ is given by \eqref{cutofftichi}.

\begin{lem}[Remarkable identities] \label{biliasuiesti}  For  all $ w_1,w_2 \in \cS (\RR^d \times \RR^d; \CC) $, we have
\begin{subequations}\label{l2conserv}
\begin{eqnarray}
& & \displaystyle \cB_{\tilde \chi} \lbrack w_1 , w_2 \rbrack (\cdot,0) \equiv 0 , \label{l2conserv2} \\
& &  \displaystyle \int_{\RR^d} \! \int_{\RR^d} \bar w_1(k,\eta) \, \cB_{\tilde \chi} \lbrack w_2,w_1 \rbrack (k,\eta)  \, dk d\eta = 0,  \qquad
\text{ if $ w_1= \cI_{k,\eta} w_1 $ },\label{l2conserv1} \\
& & \displaystyle \Delta_\eta \bigl( \cB_{\tilde \chi} \lbrack w_2,w_1 \rbrack \bigr) (\cdot,0)  = \frac{2}{(2 \pi)^d} \int_{\RR^d} {\tilde \chi} (\tilde k) \, \cI_{k,\eta}
w_2 (\tilde k,0) \, (\tilde k \cdot \nabla_\eta)  (\cI_{k,\eta} w_1) (\cdot-\tilde k,0) \,   d \tilde k . \label{l2conserv3}
\end{eqnarray}
\end{subequations}
\end{lem}

\begin{proof} Recall the definition (\ref{bilitodefine}) of $ \cB_{\tilde \chi} $.

\smallskip

\noindent $ \circ $ The factor $ \sin (\tilde k \cdot \eta/2) $ is just zero for $ \eta = 0 $. This furnishes (\ref{l2conserv2}).

\smallskip

\noindent $ \circ $ In the first line below, we make explicit the double integral (\ref{l2conserv1}). Then, we replace the variable $ (k,\eta) $
by $ (-k,-\eta) $, and we use \eqref{paritedew} to get the second line
\renewcommand\arraystretch{2}
\[ \begin{array}{l}
\displaystyle \frac{2}{(2 \pi)^d} \int_{\RR^d} \! \int_{\RR^d} \! \int_{\RR^d} \overline {\cI_{k,\eta} w_1}(k,\eta) \, {\tilde \chi} (\tilde k) \, \sin (\tilde k \cdot \eta /2 ) \,
\cI_{k,\eta} w_2 (\tilde k,0) \, \cI_{k,\eta} w_1 (k-\tilde k, \eta) \, dk d \tilde k d\eta \\
\displaystyle \quad = - \frac{2}{(2 \pi)^d} \int_{\RR^d} \! \int_{\RR^d} \! \int_{\RR^d} \cI_{k,\eta} w_1(k,\eta) \, {\tilde \chi} (\tilde k) \, \sin (\tilde k \cdot \eta /2 ) \,
\cI_{k,\eta} w_2(\tilde k,0) \, \overline{\cI_{k,\eta} w_1} (k+\tilde k, \eta) \, dk d \tilde k d\eta .
\end{array} \]
\renewcommand\arraystretch{1}

\noindent Change $ k $ into $ k-\tilde k $ in the above second line to recognize the first expression preceded by the minus sign $ - $.
This yields (\ref{l2conserv1}).

\smallskip

\noindent $ \circ $ Then, compute
\[ \Delta_\eta \bigl \lbrack \sin (\tilde k \cdot \eta /2 ) \, \cI_{k,\eta} w (k-\tilde k, \eta) \bigr \rbrack (k,0) = (\tilde k \cdot \nabla_\eta) (\cI_{k,\eta} w)
(k-\tilde k, 0) . \]
The last line (\ref{l2conserv3})  is a straightforward consequence of this relation.
\end{proof}


\subsubsection{Invariants}\label{consquan}
The relations in (\ref{l2conserv}) give rise to  the following statement.

\begin{cor}[Conserved quantities] \label{conservedquantities}
Let $ (w,\alpha) (\cdot) $ be a smooth rapidly decreasing  solution to the
Cauchy problem (\ref{SKGrevisited})-(\ref{inidataSKGrevisited}), issued from a $ \cS $-prepared data $ w_0 $, that is $ w_0 = (2 \pi)^{-d} \cF_{x,\xi} \circ {\rm W} (u_0) $ for some $ u_0 \in
\cS (\RR^d ; \CC) $. Then, the following
quantities are constant with time.

\medskip

\noindent $ \bullet $ The $ L^2 $-norm of $ w (t,\cdot) $:
\begin{equation} \label{l2normconserved}
\forall t \in \RR , \qquad \int_{\RR^d} \! \int_{\RR^d} \vert w(t,k,\eta) \vert^2 \, dk d\eta = \int_{\RR^d} \! \int_{\RR^d} \vert
w_0(k,\eta) \vert^2 \, dk d\eta =(2\pi)^d \, \|u_0\|^4_{L^2}.
\end{equation}
$ \bullet $ The value of $ w (t,\cdot) $ at the origin:
\begin{equation} \label{l2bisnormconserved}
\forall t \in \RR , \qquad w(t,0,0) = w_0 (0,0)= \|u_0\|^2_{L^2} .
\end{equation}
$ \bullet $ The $ \text{SKG}_\chi $ energy expressed in terms of $ w $ and $ \alpha $.
For every $ t \in \RR $, define
\renewcommand\arraystretch{2}
\begin{equation} \label{l2bissnormconserved}
\begin{array}{rl}
\tilde \cE_\chi (t) \! \! \! & \displaystyle := - \frac{\Delta_\eta}{2} w(t,0,0)
+ \frac{1}{(2\pi)^{d/2}}
\int_{\RR^d} w(t,-k,0)\,\omk^{-1/2}\,\chi(k)\,\cI\alpha(t,k)\,dk \\
& \displaystyle \quad\,
+ \int_{\RR^d} \omk\,|\alpha(t,k)|^2\,dk
= \tilde \cE_\chi (0).
\end{array}
\end{equation}
\renewcommand\arraystretch{1}

\noindent Moreover, the quantity $ \tilde \cE_\chi(t) $ is precisely the total energy
$ \cE_\chi(t) $ of the Schr\"odinger--Klein--Gordon system \eqref{PSKG},
rewritten in terms of the Fourier--Moyal variables $ (w,\alpha) $.
\end{cor}

\noindent The information provided by \eqref{l2bissnormconserved} should be compared with identity (2.17) in \cite{MR515899} and with the last assertion of Theorem 3 in \cite{MR778979}.
\begin{proof} By Lemma \ref{preparedpropagation}, we know already that $ w = \cI_{k,\eta} w $. This information will be repeatedly
used in what follows.

\smallskip

\noindent $ \circ $ From (\ref{SKGrevisited1}), it is easy to infer that
\[ \part_t \vert w\vert^2 - (k \cdot \nabla_\eta) \vert w\vert^2 + (2 \pi)^{d/2}\bar w \cB_\indic \lbrack \omk^{-1/2} \alpha , w \rbrack + (2 \pi)^{d/2}  w \overline{\cB_\indic \lbrack \omk^{-1/2} \alpha ,
w \rbrack} = 0 , \]
and therefore that
\[ \part_t \int_{\RR^d} \! \int_{\RR^d} \vert w (t,k,\eta) \vert^2 \, dk d\eta + 2 (2 \pi)^{d/2} Re \ \Bigl( \int_{\RR^d} \! \int_{\RR^d} \bar w (t,k,\eta) \,
\cB_\indic \lbrack \omk^{1/2} \alpha , w \rbrack (t,k,\eta) \, dk d \eta \Bigr)= 0 . \]
Since $ w = \cI_{k,\eta} w $, we can apply (\ref{l2conserv1}) with $ w_2 = \omk^{1/2} \alpha $ to see that the right hand side disappears. This leads
directly to the left part of (\ref{l2normconserved}). The link to the $ L^2 $-norm of $ u $ can be achieved through Plancherel theorem and
Moyal's identity (\ref{Moyals}), which give rise to
\begin{equation} \label{rein}
\parallel w \parallel_{L^2} = \parallel v \parallel_{L^2} = (2 \pi)^{d/2} \parallel u \parallel^2_{L^2} .
\end{equation}

\noindent $ \circ $ In view of (\ref{l2conserv2}), for $ k = 0 $ and $ \eta = 0 $, the equation (\ref{SKGrevisited1}) reduces to $ \part_t w (t,0,0) = 0 $,
which explains (\ref{l2bisnormconserved}). On the other hand, exploiting (\ref{recallwigff}) and (\ref{introdew}), we find that
\begin{equation} \label{integradew}
\parallel u (t,\cdot) \parallel^2_{L^2} =  (2 \pi)^{-d} \int_{\RR^d} \! \int_{\RR^d} v(t,x,\xi) dx d\xi = w(t,0,0) \, ,
\end{equation}
which illustrates the second link with the mass of the nucleon particle.

\smallskip

\noindent $ \circ $ We now turn to (\ref{l2bissnormconserved}). Apply $ \Delta_\eta $ to (\ref{SKGrevisited1}), and take $ (k , \eta)
= (0,0) $ in the resulting expression. Then, exploit (\ref{l2conserv3}) to find
\[ \part_t  \Bigl(\frac{\Delta_\eta}{2} w \Bigr) (t,0,0) + \frac{1}{(2\pi)^{d/2}} \int_{\RR^d} \omk^{-1/2} \, \chi(k) \,  \cI_k \alpha (t, k) \, ( k \cdot \nabla_\eta)  (\cI_{k,\eta} w)
(t,- k,0) \,   d  k = 0 . \]
On the other hand, taking into account (\ref{l2conserv2}), the equation (\ref{SKGrevisited1}) for $ (k,\eta) = (-  k, 0) $ yields
\[ ( k \cdot \nabla_\eta)  w (t,- k,0) = - \part_t w (t,-k,0) . \]
After substitution, recalling that $ w = \cI_{k,\eta} w $, this yields
\[ \part_t  \Bigl(\frac{\Delta_\eta}{2} w \Bigr) (t,0,0) - \frac{1}{(2\pi)^{d/2}} \int_{\RR^d} \omk^{-1/2} \, \chi(k) \,  \cI_k \alpha (t, k) \, \part_t w (t,- k,0) \,   d  k = 0 . \]
Because $ \cI_k \alpha \, \part_t w
= \part_t (\cI_k \alpha \, w) -  \part_t \cI_k \alpha \, w $ whereas $ \part_t \cI_k \alpha = - i \omk \cJ_k \alpha $, there remains
\renewcommand\arraystretch{2}
\[ \begin{array}{rl}
\displaystyle \part_t \Bigl \lbrace - \Bigl(\frac{\Delta_\eta}{2} w \Bigr) (t,0,0) \! \! \! & \displaystyle + \, \frac{1}{(2\pi)^{d/2}} \int_{\RR^d} \chi(k) \, w (t,- k,0) \, \omk^{-1/2} \,
\cI_k \alpha (t, k) \,   d  k \Bigr \rbrace \\
\ & \displaystyle + \, \frac{i}{(2\pi)^{d/2}} \int_{\RR^d} \chi(k) \, w (t,- k,0) \, \omk^{1/2} \, \cJ_k \alpha (t, k) \,   d  k = 0 .
\end{array} \]
\renewcommand\arraystretch{1}

\noindent Since $ \part_t \vert \alpha \vert^2 = i (2\pi)^{-d/2} \omk^{-1/2} \chi (\alpha \bar w - \bar \alpha w) $, the last line is the same as
\[ \begin{array}{l}
\displaystyle \frac{i}{2 (2\pi)^{d/2}}  \int_{\RR^d} \chi(k) \, w (t,- k,0) \omk^{1/2} \displaystyle \bigl \lbrack \alpha (t, k) - \bar \alpha (t, -k) \bigr \rbrack \,  d  k  \\
\displaystyle \quad = \frac{i}{2 (2\pi)^{d/2}}
\Bigl(
\int_{\RR^d} \chi(k) \, \omk^{1/2} \alpha (t, k) \bar w (t, k,0) \, dk -  \int_{\RR^d} \chi(k) \, \omk^{1/2} \bar \alpha (t, k)  w (t,k,0) \, dk \Bigr) \\
\displaystyle \quad = \part_t \int_{\RR^d} \omk \vert \alpha (t,k) \vert^2 \, dk .
\end{array} \]
To put it briefly, we have checked that $ \part_t \tilde \cE_\chi (t) = 0 $, that is $ \tilde \cE_\chi (t) = \tilde \cE_\chi (0) $ as claimed.

\smallskip

\noindent To complete the proof, assuming that $ (w,\alpha) $ is a (smooth rapidly decreasing) solution to \eqref
{SKGrevisited} issued from a solution $ (u,A) $ to \eqref{PSKG}, we have to compare the computation of $
\tilde \cE_\chi (t) $
in \eqref{l2bissnormconserved} with the one of $ \cE_\chi (t) $ in \eqref
{eq.skg.energytr}. First,
it should be noted that the second relation of (\ref{recallwigff}) indicates that
\renewcommand\arraystretch{2}
\begin{equation} \label{integraxidewbis}
\begin{array}{rl}
\displaystyle - \frac{\Delta_\eta}{2} w (t,0,0) \! \! \! & \displaystyle = \frac{1}{2 (2\pi)^d} \int_{\RR^d} \! \int_{\RR^d}
\vert \xi \vert^2 \, v(t,x,\xi) dx d\xi \\
\ & \displaystyle = \frac{1}{2 (2\pi)^d} \int_{\RR^d} \vert \xi \vert^2 \, \vert \hat u(t,\xi) \vert^2 \, d\xi = \frac{1}{2} \parallel \nabla_x u (t,\cdot)
\parallel^2_{L^2} .
\end{array}
\end{equation}
Secondly, recall \eqref{pasaoublierbis} which establishes the link between the $ L^2_{1/2} $-norm of $ \alpha $ and the free energy of the Klein--Gordon equation.

\smallskip

\noindent Thirdly, we focus on
 the interaction term.
Observe
from \eqref{recallwigff}-\eqref{recallwigffl1}, that we have
\begin{equation}
\begin{array}{rl}
\displaystyle w (t,-k,0) \! \! \! & \displaystyle
=\frac{1}{(2\pi)^d}
\int_{\RR^d}\int_{\RR^d} e^{i k\cdot x} v(t,x,\xi) dx d\xi \\
\ & \displaystyle =\int_{\RR^d} e^{i k\cdot x} |u(t,x)|^2 dx = \overline{\cF_x \bigl(\vert u(t,\cdot) \vert^2 \bigr)}(k) .
\end{array}
\end{equation}
It follows with \eqref{llllllllll} that
$$ \begin{array}{rl}
\displaystyle \frac{1}{(2\pi)^{d/2}} \int_{\RR^d} w (t,-k,0) \, \frac{\chi(k)}{\omk^{1/2}} \, \cI \alpha (t,k) \, dk \! \! \! & \displaystyle =
\frac{1}{(2\pi)^{d}} \int_{\RR^d} \overline{\cF_x \bigl(\vert u(t,\cdot) \vert^2 \bigr)} (k) \, \chi(k) \, \cF_x \bigl( A(t,\cdot) \bigr) (k) \, dk \\
& \displaystyle = \int_{\RR^d} \chi(D) A (t,x) \ |u(t,x)|^2 \, dx .
\end{array}
$$
By picking up the pieces, we recover that $ \tilde \cE_\chi (t) = \cE_\chi (t) $. 
For $ \chi \equiv \indic_\RR $, it is well known that the energy $ \cE_\indic $ is conserved along the flow of SKG. The same property holds for $ \cE_\chi $ in the case of $ \text{SKG}\chi $. This can be established directly from \eqref{PSKG}. Our contribution is rather to rewrite $ \cE\chi $ in terms of the variables $ (w,\alpha) $, and to prove that the resulting energy $ \tilde{\cE}_\chi $ is conserved along the flow of \eqref{PSKG} expressed in the Fourier--Moyal variables.
\end{proof}


\subsubsection{Interpretation in terms of $ u $ and $ A $}\label{subsubsec:interpreattion} In this paragraph, for the sake of simplicity, we work with $ d=3 $ and we take $ \chi \equiv 1 $.
It is interesting to discuss  the content of Corollary \ref{conservedquantities} in terms of $u$ and $A$.
On the one hand, the two identities  \eqref{l2normconserved} and  (\ref{l2bisnormconserved})  express both the conservation of the $ L^2 $-norm of the nucleon wave function $ u $, as it can be easily
deduced from the Schr\"odinger equation (\ref{PSKG1}). On the other hand, the expression  (\ref{l2bissnormconserved})
corresponds to the total energy of the SKG system \eqref{PSKG}. In \eqref{eq.skg.energytr}, the interaction term (in the last line) inherits  apparently no sign condition. Still, it can be absobed since (as can be seen below) the expression $ \cE_1 $ is comparable to the  energy of the non-interacting particle-field system.

\begin{cor}[Comparison of the total and free energies] \label{interpretationofE}  Assume that $ d = 3 $.
Fix a smooth rapidly decreasing solution $ (w,\alpha) (\cdot) $ of the Cauchy problem (\ref{SKGrevisited})-(\ref{inidataSKGrevisited}),
which is  generated by a smooth rapidly decreasing solution $ (u,A) $ to \text{\rm SKG}-(\ref{inidataSKGinif}) . Then, we can find a constant
$ C \in \RR_+^* $ depending only on the $ L^2 $-norm of $ w_0 $ (or $ u_0 $) such that, uniformly in $ t \in \RR_+ $, we have
\begin{equation} \label{equivalencedeEeten}
\parallel \nabla_x u(t,\cdot) \parallel^2_{L^2 (\RR^3)} + \parallel \part_t A(t,\cdot) \parallel^2_{L^2 (\RR^3)} + \parallel A(t,\cdot)
\parallel^2_{H^1 (\RR^3)} \, \sim \cE_1 (t) + C .
\end{equation}
\end{cor}

\begin{proof} We have already seen that the first and third term inside the definition of $ \tilde \cE_1 \equiv \cE_1 $ are exactly the same to the left hand side of (\ref{equivalencedeEeten}). It is enough then to explain how to estimate the second remaining term in \eqref{l2bissnormconserved}. From (\ref{l2normconserved}) or
(\ref{l2bisnormconserved}), we know that
\[ \forall t \in \RR_+ , \qquad \parallel u(t,\cdot) \parallel_{L^2 (\RR^3)} = \parallel u_0 (\cdot) \parallel_{L^2 (\RR^3)} = C(w_0) < + \infty . \]
Taking into account (\ref{recallwigff}), this furnishes
\renewcommand\arraystretch{2}
\begin{equation} \label{estilinftysurmew}
\begin{array}{rl}
\displaystyle \vert w (t,-k,0) \vert \! \! \! & \displaystyle = (2 \pi)^{-3} \vert \cF_x \Bigl( \int_{\RR^3} v(t,\cdot,\xi) d\xi
\Bigr) (-k) \vert = \vert \cF_x \bigl( \vert u(t,\cdot) \vert^2 \bigr) (-k) \vert  \\
\ & \leq \parallel \vert u(t,\cdot) \vert^2 \parallel_{L^1(\RR^3)} = \parallel u(t,\cdot) \parallel^2_{L^2(\RR^3)} = C(w_0)^2 < + \infty .
\end{array}
\end{equation}
\renewcommand\arraystretch{1}
H\"older's inequality gives rise to
\[ \vert \int_{\RR^3} w (t,-k,0) \, \omk^{-1} \cI \alpha (t,k) \, dk \vert \leq \parallel w (t,\cdot,0) \parallel_{L^4 (\RR^3)}
\parallel \omdotk^{-1} \cI \alpha (t,\cdot) (t,\cdot) \parallel_{L^{4/3} (\RR^3)} . \]
On the other hand, the sup norm estimate (\ref{estilinftysurmew}) allows to deal with the $ L^4 $-norm of $ w(t,\cdot,0) $ through
tame estimates according to
\renewcommand\arraystretch{2}
\[ \begin{array}{rl}
\displaystyle \parallel w (t,\cdot,0) \parallel_{L^4 (\RR^3)} \! \! \! & \displaystyle \leq C(w_0) \Bigl( \int_{\RR^3}
\vert w (t,k,0) \vert^2 dk \Bigr)^{1/4} = C(w_0) \parallel \cF_x \bigl( \vert u(t,\cdot) \vert^2 \bigr) \parallel_{L^2 (\RR^3)}^{1/2} \\
\ & \displaystyle \leq (2 \pi)^{3/4} \, C(w_0) \parallel \vert u(t,\cdot) \vert^2 \parallel_{L^2 (\RR^3)}^{1/2} \lesssim \parallel u (t,\cdot)
\parallel_{L^4 (\RR^3)} .
\end{array} \]
\renewcommand\arraystretch{1}
Then, by Gagliardo-Nirenberg interpolation inequality, we can assert that
\[ \parallel u (t,\cdot) \parallel_{L^4 (\RR^3)} \leq \parallel \nabla_x u (t,\cdot) \parallel^{3/4}_{L^2 (\RR^3)} \parallel u (t,\cdot)
\parallel^{1/4}_{L^2 (\RR^3)} \lesssim \parallel \nabla_x u (t,\cdot) \parallel^{3/4}_{L^2 (\RR^3)} . \]
On the other hand, by H\"older's inequality, we have (since $ d = 3 $)
\[ \parallel \omdotk^{-1} \cI \alpha (t,\cdot) \parallel_{L^{4/3} (\RR^3)} \leq \Bigl( \int_{\RR^3}  \omk^{-4} dk \Bigr)^{1/4}
\parallel \cI \alpha (t,\cdot) (t,\cdot) \parallel_{L^2 (\RR^3)} \lesssim  \parallel \alpha (t,\cdot) \parallel_{L^2 (\RR^3)} . \]
Combining the above information, there remains
\[ \vert \int_{\RR^3} w (t,-k,0) \, \omk^{-1} \cI \alpha (t,k) \, dk \vert \lesssim \parallel \nabla_x u (t,\cdot) \parallel^{3/4}_{L^2 (\RR^3)}
\parallel \alpha (t,\cdot) \parallel_{L^2 (\RR^3)} . \]
Remark that
\[ \forall \eta \in \RR_+^* , \quad \exists C(\eta) \in \RR_+^* ; \qquad \forall (a,b) \in (\RR_+^*)^2 , \quad a^{3/4} b \leq C(\eta) + \eta a^2
+ \eta b^2 . \]
In particular, by choosing $ \eta $ sufficiently small, we can guarantee that
\[ \vert 2 \int_{\RR^3} w (t,-k,0) \, \omk^{-1} \cI \alpha (t,k) \, dk \vert \leq C  + \frac{1}{2} \parallel \nabla_x u (t,\cdot)
\parallel^2_{L^2 (\RR^3)} + \frac{1}{2}  \parallel \alpha (t,\cdot) \parallel^2_{L^2 (\RR^3)} , \]
from which (\ref{equivalencedeEeten}) follows easily .
\end{proof}

\noindent At this stage,  several remarks are in order:
\begin{enumerate}
\item Consider SKG-(\ref{inidataSKGinif}). Assuming (\ref{regularinidata}),
the upper bound
\begin{equation*} \label{supequivalencedeEeten}
\sup_{t \in \RR} \ \bigl( \parallel u(t,\cdot) \parallel_{H^1 (\RR^d)} + \parallel \part_t A(t,\cdot) \parallel_{L^2 (\RR^d)} +
\parallel A(t,\cdot) \parallel_{H^1 (\RR^d)} \bigr) < + \infty .
\end{equation*}
 is well-known, see for instance the articles \cite{MR778979,MR519635} where it is related to the construction of
 {\it strong} global solutions $ (u,A) $ to (\ref{PSKG}).
\item Knowing that $ w (t,\cdot) $ is in $ L^2 $ is far from sufficient in order to give a sense to (\ref{l2bisnormconserved}), and
even more to (\ref{l2bissnormconserved}). This is because the trace of $ w (t,\cdot) $ or $ \Delta_\eta w (t,\cdot) $ at the origin
is not accessible via some $ L^2 $-information
concerning $ w $. However, we will see in Paragraph \ref{subsubsec:structureofprepared data} that these difficulties disappear in the
case of $ L^2 $-prepared data.
\item It should be kept in mind that generic solutions to
(\ref{transporteqonw}) or (\ref{SKGrevisited}) that are not generated by prepared initial data
$w_0$ need not correspond to solutions of (\ref{PSKG}). Even when such a correspondence exists,
the regularity and integrability properties of $w$ and $\alpha$ cannot, in general, be readily
interpreted in terms of the original variables $u$ and $A$. This point will be discussed in more
detail in Section~\ref{sec:L2preparedsol}.
\end{enumerate}

\noindent From this perspective, the system~\eqref{SKGrevisited} may be regarded as a genuine extension of
\eqref{PSKG}: it contains the latter as a particular case while allowing for additional microscopic
degrees of freedom.


\section{Fourier--Moyal Cauchy problem}\label{sec:cauchyproblem}
In this section, we investigate the Fourier--Moyal (FM) equation
\eqref{transporteqonw}--\eqref{inidatatransporteqonw} in its own right.
Throughout this section, we work under the following assumptions.

\begin{assump}[Conditions on the cutoff $\tilde\chi$ and on the modulus of continuity $\omega$]\label{assumpFM} \
\begin{enumerate}
    \item The cutoff $\tilde\chi$ satisfies
    \[
    \tilde\chi \in L^{\tilde p}_{\tilde s},
    \qquad
    (\tilde p,\tilde s)\in [1,+\infty]\times\mathbb R.
    \]
    \item The modulus of continuity $\omega$ is concave\footnote{In particular,
    $\omega$ is subadditive, uniformly continuous, and sublinear.}
    and satisfies \eqref{modconw} for some
    $ \tilde\imath\in]0,1] $.
\end{enumerate}
\end{assump}

\medskip
\noindent Given $ (l,s,\imath)\in \mathbb N \times [1,+\infty ] \times [0,\tilde \imath] $, we start with initial data satisfying
$$ \exists \, q \in [1,+\infty ] \, ; \qquad \tilde \alpha_0 \in L^{q}_{l+s + \imath+\frac 1 2} \, , \qquad w_0 \in W^{l,q}_{s,\imath}
(\cC^{\omega}_b) \, . $$

\noindent The purpose of this section is to prove that the initial value problem \eqref{transporteqonw}-\eqref{inidatatransporteqonw} is strongly well-posed in the space $W^{l,q}_{s,\imath}(\cC^{\omega}_b)$, provided that the exponent $q$ is suitably chosen in terms of $(\tilde p,\tilde s,\tilde\imath,\imath)$. In Subsection~\ref{contiproYukawamap}, we investigate the mapping properties of the operator $\cY\cB$ in this functional framework. Then, in Subsection~\ref{subsec:locmildsol}, we establish the local well-posedness of the problem by proving existence and uniqueness of solutions. Notably, $w(t,k,\cdot)$ remains uniformly $\omega$-continuous in the variable $\eta$, locally in time, where $\omega$ is the same modulus of continuity as that of the initial datum $w_0$. In particular, this shows that the corresponding microlocal regularity is propagated by the flow.


\subsection{Continuity properties of the Yukawa bilinear map}\label{contiproYukawamap} A common way to solve the integral equation (\ref{integralversionSKGS}) is to apply a fixed point argument. This means to study the continuity
properties of the $ \cY \cB $-map. Lemma \ref{continuityYBbasic} provides
already some guidance in this direction. But the conclusions of Lemma \ref{continuityYBbasic} are not sufficient and far from
optimal for two main reasons:
\begin{itemize}
  \item First, they do not fully exploit the assumption $ \tilde \chi \in L^{\tilde p}_{\tilde s} $ which, for $ \tilde s \in \mathbb R_+^* $ or $ \tilde p \in [1,+\infty[ $, adds supplementary elements compared to the use of $ \tilde \chi \in L^\infty $ as in \eqref{continuityYBbasicestimates}. In particular, when dealing with $ \text{SKG}_{\chi} $, the function $ \tilde \chi $ is as in \eqref{cutofftichi} with a cutoff $ \chi $  satisfying (at least) the condition  \eqref{voevenchi}, so that $ \tilde \chi \in L^\infty_1 $. It is important to take advantage of the condition $ \tilde s = 1 \in \mathbb R_+^* $.
  \item Secondly, they do not guarantee that the function $ \cB_{\tilde \chi}[ w_1, w_2 ] (k,\cdot) $ is still continuous near the position $ \eta = 0 $. However, in order to implement a Picard scheme, this information is repeatedly needed because the definition of $ \cB_{\tilde \chi}[ w_1, w_2 ] $ requires to compute the trace at $ \eta = 0 $ of the first entry $ w_1 $. The uniform continuity of $ w_2 $ with respect to $ \eta $ helps to address this shortfall, as seen below.
\end{itemize}

\begin{lem}[Bilinear continuity properties of the $\cY\cB$-map]
\label{continuityYBfirstcase}
Suppose that Assumption~\ref{assumpFM} holds. 
Let
\[
(l,s,\imath) \in \NN \times \RR_+ \times [0,\tilde\imath],
\]
and choose $(p,q,r) \in [1,+\infty]^3$ such that
\[
\frac{1}{\tilde p}+\frac{1}{p}+\frac{1}{q}
=1+\frac{1}{r}.
\]
 Then the
$\cY\cB$-map defines a continuous bilinear map
\begin{equation}\label{intoitself}
\cB_{\tilde\chi} :
L^p_{l+s-\tilde s+\tilde\imath}(\cC^0_b)
\times
W^{l,q}_{s,\imath}(\cC^\omega_b)
\longrightarrow
W^{l,r}_{s,\imath}(\cC^\omega_b).
\end{equation}
More precisely, one has the estimate
\begin{equation}\label{continuityYBevolutedestimates}
\bigl\|
\cB_{\tilde\chi}[w,\slw]
\bigr\|_{W^{l,r}_{s,\imath}(\cC^\omega_b)}
\lesssim
\|\tilde\chi\|_{L^{\tilde p}_{\tilde s}}\,
\|w\|_{L^p_{l+s-\tilde s+\tilde\imath}(L^\infty)}\,
\|\slw\|_{W^{l,q}_{s,\imath}(\cC^\omega_b)} .
\end{equation}
\end{lem}

\begin{proof} We start by discussing the case $ l = 0 $. From the definition \eqref{calculnormcontinuouscondition} and the triangle inequality, we get that
\begin{equation} \label{diffautilisercarinhoooraa}
\parallel \cB_{\tilde \chi}  [ w, \slw ] \parallel_{L^r_{s,\imath} (\cC^{\omega}_b) } \leq \parallel \cB_{\tilde \chi}  [ w, \slw ] \parallel_{L^r_{s + \imath}
(L^\infty)} + \parallel k \mapsto \vert \cB_{\tilde \chi}  [ w, \slw ] (k,\cdot) \! \downharpoonright_\omega \parallel_{L^r_s} .
\end{equation}
We proceed in three distinct steps. First ($i$), we estimate the first term in the right hand side. Secondly ($ii$), we consider the part involving $ \vert \, \cdot  \downharpoonright_\omega $. Thirdly ($iii$), we examine the case $ l \in \NN^* $.

 \smallskip

 \noindent $(i) - $ {\it Control of the sup norm} - By assumption
\begin{equation} \label{assontilchi} \exists \, \check \chi \in L^{\tilde p} \, ; \qquad \tilde \chi (k)= \omk^{-\tilde s} \ \check \chi (k) \, .
\end{equation}
 Observe that
\begin{equation} \label{astucepeetre}
\begin{array}{rl}
\displaystyle \omk^{s+\imath} \, \vert \cB_{\tilde \chi} \lbrack w , \slw \rbrack (k,\eta) \vert \lesssim  \int_{\RR^d} \frac{\omk^{s+\imath}}{\omdifk^{s+\imath}
\omtilk^{s+\imath}} \! \! \! & \displaystyle \vert \check \chi(\tilde k) \vert \, \omtilk^{s-\tilde s +\imath} \vert \cI_{k,\eta}  w (\tilde k,0) \vert \\
\ & \displaystyle \times  \omdifk^{s+\imath} \vert \cI_{k,\eta} \slw ( k-\tilde k, \eta) \vert \ d \tilde k .
\end{array}
\end{equation}
Recall Peetre's inequality
\begin{equation} \label{peetres}
 \forall \, \check s \in \RR_+ , \qquad \frac{\omk^{\check s}}{\omdifk^{\check s} \ \omtilk^{\check s}} \leq 2^{\, \check s /2} .
 \end{equation}
 Since $ s+\imath \geq 0 $, we can apply \eqref{peetres} with $ \check s = s +\imath $.
 It follows that
 $$
\begin{array}{rl}
\displaystyle \omk^{s+\imath} \, \parallel  \cB_{\tilde \chi} \lbrack w , \slw \rbrack (k,\cdot) \parallel_{L^\infty} \lesssim  \int_{\RR^d} \! \! \! & \displaystyle \vert \check \chi(\tilde k) \ \vert \, \omtilk^{s-\tilde s+\imath} \  \bigl( \parallel w (\tilde k,\cdot) \parallel_{L^\infty} +  \parallel w (-\tilde k,\cdot) \parallel_{L^\infty} \bigr) \\
\ & \displaystyle \times  \omdifk^{s+\imath} \ \bigl( \parallel \slw ( k-\tilde k, \cdot) \parallel_{L^\infty} + \parallel \slw ( -k+\tilde k, \cdot) \parallel_{L^\infty} \bigr) \ d \tilde k .
\end{array} $$
By Young's and H\"older's
inequality, there remains
$$ \parallel \cB_{\tilde \chi} \lbrack w , \slw \rbrack \parallel_{L^r_{s + \imath} (L^\infty)}\, \lesssim \, \parallel \check \chi \parallel_{L^{\tilde p}} \ \parallel
w \parallel_{L^p_{s -\tilde s+\imath}(L^\infty)} \ \parallel \slw \parallel_{L^{q}_{s+\imath}(L^\infty)} . $$
Recall that
$ \imath \leq \tilde \imath $ and that $ \parallel \check \chi \parallel_{L^{\tilde p}} = \parallel \tilde \chi \parallel_{L^{\tilde p}_{\tilde s}} $. Thus, this can be bounded by
\begin{equation} \label{astucepeetreprim}
\qquad \parallel \cB_{\tilde \chi} \lbrack w , \slw \rbrack \parallel_{L^r_{s + \imath} (L^\infty)}\, \lesssim \, \parallel \tilde \chi \parallel_{L^{\tilde p}_{\tilde s}} \ \parallel
w \parallel_{L^p_{s -\tilde s+ \tilde \imath}(L^\infty)} \ \parallel \slw \parallel_{L^{q}_{s,\imath}(\cC^\omega_b)} .
\end{equation}

\smallskip

 \noindent $(ii) - $ {\it Control of the H\"older coefficient} -
 Let $ (\eta_1,\eta_2) \in \RR^d \times \RR^d $ with $ \eta_1 \not = \eta_2 $. If $ \bar z \leq \vert \eta_1 - \eta_2 \vert $,
 coming back to \eqref{modconw}, we can assert that $ 0 < c \, \bar z^{\tilde \imath} \leq \omega (\vert \eta_1 - \eta_2 \vert) $. As a consequence
$$
  \vert \cB_{\tilde \chi} \lbrack w , \slw \rbrack (k,\cdot) \! \downharpoonright_\omega \, \leq \! \! \!
\sup_{0 < \vert \eta_1 - \eta_2 \vert \leq \bar z} \! \! \! \frac{\vert \cB_{\tilde \chi} \lbrack w , \slw \rbrack (k,\eta_1) - \cB_{\tilde \chi} \lbrack w , \slw \rbrack (k,\eta_2) \vert}{\omega( \vert \eta_1 - \eta_2 \vert)} + \frac{2}{c \, \bar z^{\tilde \imath}} \parallel \cB_{\tilde \chi} \lbrack w , \slw \rbrack (k,\cdot) \parallel_{L^\infty} . $$
By repeating the arguments of ($ i $), with $ s \geq 0 $ (in place of $ s + \imath $), it is easy to infer that
\begin{equation} \label{repeatingrim} \parallel \cB_{\tilde \chi} \lbrack w , \slw \rbrack  \parallel_{L^r_s(L^\infty)} \lesssim \, \parallel \tilde \chi \parallel_{L^{\tilde p}_{\tilde s}} \ \parallel
w \parallel_{L^p_{s -\tilde s}(L^\infty)} \ \parallel \slw \parallel_{L^{q}_{s}(L^\infty)} .
\end{equation}
From now on, we focus on the case $ 0 < \vert \eta_1 - \eta_2 \vert \leq  \bar z $.
The best way to proceed is to directly examine differences. We have
\renewcommand\arraystretch{2}
\begin{equation} \label{diffautiliser}
 \begin{array}{rl}
\displaystyle \omega \! \! \! \! \! & \displaystyle (\vert \eta_1 -\eta_2 \vert)^{-1} \, \vert \cB_{\tilde \chi} \lbrack w , \slw \rbrack (k,\eta_1) -
\cB_{\tilde \chi} \lbrack w , \slw \rbrack (k,\eta_2) \vert \\
\ & \displaystyle \lesssim \int_{\RR^d} \omtilk^{-\tilde s} \, \vert \check  \chi(\tilde k) \vert \, \vert \cI_{k,\eta} w(\tilde k,0) \vert \, \vert \cI_{k,\eta} \slw (k-\tilde k, \cdot)
\downharpoonright_\omega \, d \tilde k  \\
\ & \displaystyle \ \ \ + \int_{\RR^d} \frac{\vert \sin (\tilde k \cdot \eta_1 /2 ) - \sin (\tilde k \cdot \eta_2 /2 ) \vert}{\omega (\vert \eta_1 -\eta_2 \vert)} \, \vert \tilde \chi(\tilde k) \vert \, \vert \cI_{k,\eta} w(\tilde k,0) \vert \, \vert \cI_{k,\eta} \slw (k-\tilde k, \eta_2) \vert \, d \tilde k  .
\end{array}
\end{equation}
\renewcommand\arraystretch{1}

\begin{rem}[Need for the modulus $ \omega $]
The estimate \eqref{diffautiliser} remains valid when $ \omega \equiv 1 $, provided that the seminorm
$ \vert \cI_{k,\eta} \slw (k-\tilde k, \cdot) \downharpoonright_\omega $
is replaced by the difference
\[
\bigl| \slw (k-\tilde k, \eta_1) - \slw (k-\tilde k, \eta_2) \bigr|
=o \bigl( |\eta_1-\eta_2| \bigr).
\]
However, the corresponding little-$o$ estimate may depend on the triple
$ (k-\tilde k,\eta_1,\eta_2) $. In contrast, estimating the integral in the
second line of \eqref{diffautiliser} requires a quantitative bound that is
uniform with respect to $ \eta_1 $ and $ \eta_2 $. This is one of the main
reasons for introducing the seminorm
$ \vert \, \cdot \downharpoonright_\omega $, which provides precisely such
uniform control.
\end{rem}

\noindent Observe that
 \begin{equation} \label{diffsinaexploiter}
 \sup_{0 < \vert \eta_1 - \eta_2 \vert \leq \bar z} \frac{\vert \sin (\tilde k \cdot \eta_1 /2 ) - \sin (\tilde k \cdot \eta_2 /2 ) \vert}{\omega (\vert \eta_1 -\eta_2 \vert)}
 \leq 2 \sup_{0 < z \leq \bar z} \ \frac{\min (1 \, ; \vert \tilde k \vert \, z)}{\omega (z)} .
 \end{equation}
Since $ \omega $ is increasing, we can assert that
\[ \sup_{1 / \vert \tilde k \vert \leq z} \ \frac{\min (1 \, ; \vert \tilde k \vert \, z)}{\omega (z)} = \sup_{1 / \vert \tilde k \vert \leq z} \
\frac{1}{\omega (z)} = \frac{1}{\omega (1/\vert \tilde k \vert)} \, .  \]
 Since $ \omega $ is concave, we can also retain that
 \[ \sup_{0 < z \leq 1/ \vert \tilde k \vert} \ \frac{\min (1 \, ; \vert \tilde k \vert \, z)}{\omega (z)} = \sup_{0 < z \leq 1/ \vert
 \tilde k \vert} \ \frac{\vert \tilde  k \vert \, z}{\omega (z)} = \frac{1}{\omega (1/\vert \tilde k \vert)} \, .  \]
 In short, exploiting \eqref{modconw}, we find that
  \begin{equation} \label{diffsinaexploiterfinaal}
  \sup_{z \in \RR^*_+} \ \frac{\min (1 \, ; \vert \tilde k \vert \, z)}{\omega (z)} = \frac{1}{\omega (1/\vert \tilde k \vert)} \leq \frac{\vert \tilde k \vert^{\tilde \imath}}{c} \lesssim \ \omtilk^{\tilde \imath} .
  \end{equation}
\begin{rem}[On the role of $ \tilde \imath $]\label{remimpactimath}
For $1 \ll \vert \tilde k \vert$, the computation of
$\cB_{\tilde \chi}\lbrack w , \slw \rbrack (k,\cdot)$
involves the rapidly oscillating sine function. As shown above, these
oscillations lead to a loss quantified by the factor
$\omtilk^{\tilde \imath}$. Nevertheless, by suitably modifying the modulus of
continuity $\omega$, as explained in
Remark~\ref{remtildeimath}, the exponent
$\tilde \imath \in \mathbb{R}_+^*$ can be chosen arbitrarily small.
\end{rem}

\noindent Multiply \eqref{diffautiliser}
by $ \omk^s $ and exploit (\ref{diffsinaexploiterfinaal})
to extract
\renewcommand\arraystretch{2}
\[ \begin{array}{l}
\displaystyle
\omk^s \ \sup_{0 < \vert \eta_1 - \eta_2 \vert \leq \bar z} \! \! \! \frac{ \vert \cB_{\tilde \chi} \lbrack w , \slw \rbrack (k,\eta_1) - \cB_{\tilde \chi} \lbrack w , \slw \rbrack (k,\eta_2) \vert}{\omega( \vert \eta_1 - \eta_2 \vert)} \\
\displaystyle \qquad \qquad \quad \lesssim \int_{\RR^d} \frac{\omk^{s}}{\omdifk^{s}
\omtilk^{s}} \ \vert \check \chi(\tilde k) \vert \ \omtilk^{s-\tilde s} \, \vert \cI_{k,\eta}  w (\tilde k,0) \vert \   \omdifk^{s} \, \vert \cI_{k,\eta} \slw ( k-\tilde k, \cdot) \! \downharpoonright_\omega \, d \tilde k \smallskip \\
\displaystyle \qquad \qquad \quad
+ \int_{\RR^d}  \frac{\omk^{s}}{\omdifk^{s}
\omtilk^{s}} \ \vert \check \chi(\tilde k) \vert \ \omtilk^{s-\tilde s+\tilde  \imath} \, \vert \cI_{k,\eta}  w (\tilde k,0) \vert \   \omdifk^{s} \parallel \cI_{k,\eta} \slw ( k-\tilde k, \cdot) \parallel_{L^\infty_\eta} \, d \tilde k \, .
\end{array}
\]
\renewcommand\arraystretch{1}

\noindent We apply  \eqref{peetres} with $ \check s = s \geq  0 $. Then, we follow the same  steps as in part ($ i )$ to control the right hand sides. With \eqref{repeatingrim}, this furnishes
  \renewcommand\arraystretch{1.2}
\begin{equation} \label{diffautilisercarinhoooajout}
 \begin{array}{rl}
\displaystyle \parallel \vert \cB_{\tilde \chi} \lbrack w , \slw \rbrack (k,\cdot) \! \downharpoonright_\omega \parallel_{L^r_{s}} \! \! \! \! & \lesssim  \, \parallel \tilde \chi \parallel_{L^{\tilde p}_{\tilde s}} \ \parallel
w \parallel_{L^p_{s -\tilde s}(L^\infty)} \ \parallel \slw \parallel_{L^{q}_{s}(L^\infty)} \\
\ & \displaystyle + \,
\parallel \check \chi \parallel_{L^{\tilde p}} \ \parallel w \parallel_{L^p_{s-\tilde s}(L^\infty)} \, \parallel \! \vert \slw (k,\cdot) \! \downharpoonright_\omega \parallel_{L^{q}_{s}} \\
\ & \displaystyle +
\parallel \check \chi \parallel_{L^{\tilde p}} \ \parallel w \parallel_{L^p_{s-\tilde s+\tilde \imath}(L^\infty)} \, \parallel \slw \parallel_{L^{q}_{s} (L^\infty)} \\
\! \! & \lesssim \, \parallel \tilde \chi \parallel_{L^{\tilde p}_{\tilde s}} \ \parallel
w \parallel_{L^p_{s -\tilde s+ \tilde \imath}(L^\infty)} \ \parallel \slw \parallel_{L^{q}_{s,\imath}(\cC^\omega_b)} .
\end{array}
\end{equation}
\renewcommand\arraystretch{1}

\noindent By adding  \eqref{astucepeetreprim} and  \eqref{diffautilisercarinhoooajout} as indicated in \eqref{diffautilisercarinhoooraa}, we recover \eqref{continuityYBevolutedestimates} for $ l = 0 $.

\medskip

 \noindent $(iii) -$ {\it The case $ l \in \NN^* $} - Introduce
 \[ \tilde \chi_\beta (\tilde k) := \tilde k^\beta \ \tilde \chi (\tilde k) \in L^{\tilde p}_{\tilde s - \vert \beta \vert} \, , \qquad \forall \beta \in \mathbb N^d \, . \]
 Applying Leibniz rule, for all $ \alpha \in \NN^d $ with $ \vert \alpha \vert \leq l $, we
 find that
\[ \begin{array}{rl}
\displaystyle \part^\alpha_\eta \cB_{\tilde \chi} \lbrack w , \slw \rbrack (k,\eta) = \sum_{\beta \leq \alpha} \, 2^{1-\vert \beta \vert}
\left( \begin{array}{c}
\alpha \\
\beta
\end{array} \right) \int_{\RR^d} \! \! \! & \displaystyle \sin^{(\vert\beta \vert)} (\tilde k \cdot \eta/2) \ \tilde \chi_\beta (\tilde k) \\
\ & \displaystyle \times \, \cI_{k,\eta} w (\tilde k,0) \ \part^{\alpha-\beta}_\eta (\cI_{k,\eta} \slw ) (k-\tilde k,\eta)  \ d\tilde k .
\end{array} \]
Each term in the right hand side is of the type $ \cB_{\tilde \chi_\beta} $ with, when appropriate, the function $ \sin $ replaced by $ \cos $.
Since \eqref{diffsinaexploiter} is still true with $ \sin $ replaced by $ \cos $, the preceding method applies with
$ s + l - \vert \alpha \vert $ (instead of $ s $) which, as required, is such that $ s + l - \vert \alpha \vert \geq s \geq 0 $. As a consequence
 \[ \begin{array}{l}
 \displaystyle \parallel \omk^{l-\vert \alpha \vert} \,  \part^\alpha_\eta \cB_{\tilde \chi} \lbrack w , \slw \rbrack \parallel_{L^r_{s,\imath} (\cC^{\omega}_b)} = \, \parallel \part^\alpha_\eta \cB_{\tilde \chi} \lbrack w , \slw \rbrack \parallel_{L^r_{s+l-\vert \alpha \vert,\imath} (\cC^{\omega}_b) } \smallskip \\
 \qquad \qquad \lesssim \,
\parallel \tilde \chi_\beta \parallel_{L^{\tilde p}_{\tilde s-\vert \beta \vert}} \, \displaystyle \sum_{\beta \leq \alpha} \parallel w \parallel_{L^p_{l+ s- \tilde s + \imath - \vert \alpha \vert + \vert \beta \vert}
 (L^\infty_\eta) } \, \parallel \part_\eta^{\alpha-\beta} \slw \parallel_{L^{q}_{s+l-\vert \alpha \vert,\imath} (\cC^{\omega}_b) } \\
 \qquad \qquad \lesssim \,
\parallel \tilde \chi \parallel_{L^{\tilde p}_{\tilde s}} \ \displaystyle \parallel w \parallel_{L^p_{l+ s- \tilde s + \imath}
 (L^\infty) } \, \sum_{\beta \leq \alpha}   \parallel
 \omk^{l-\vert \alpha -\beta \vert} \, \part_\eta^{\alpha-\beta} \slw \parallel_{L^{q}_{s,\imath} (\cC^{\omega}_b) } \\
 \qquad \qquad \lesssim \,
\parallel \tilde \chi \parallel_{L^{\tilde p}_{\tilde s}} \ \displaystyle \parallel w \parallel_{L^p_{l+ s- \tilde s + \imath}
 (L^\infty) } \, \parallel \slw \parallel_{W^{l,q}_{s,\imath} (\cC^{\omega}_b) } .
 \end{array}   \]
After summation on $ \vert \alpha \vert \leq l $, we recover (\ref{continuityYBevolutedestimates}).
\end{proof}

\noindent Before turning to the local well-posedness theory, let us recall the embedding
\eqref{defdenormcomegabimathaajouter}, namely
\begin{equation}\label{defdenormcomegabimathaajouter-1}
\| w \|_{L^p_{\,l+s-\tilde s+\tilde\imath}(L^\infty)}
\leq
\| w \|_{L^p_{\,l+s-\tilde m}(\cC^\omega_b)},
\qquad
\tilde m:=\tilde s-\tilde\imath+\imath.
\end{equation}
In view of the estimate \eqref{continuityYBevolutedestimates}, no loss in powers
of $\omk$ occurs provided that $\tilde m\geq0$. When $\tilde s>0$, this condition
is satisfied by choosing $\tilde\imath=\tilde s$ and $\imath=0$. On the other
hand, if $\tilde s=0$, one must take
$\imath=\tilde\imath\in\mathbb{R}_+^*$. Throughout the remainder of this section,
we therefore choose the parameters
$(\tilde s,\tilde\imath,\imath)$ so that
\[
\tilde m=\tilde s-\tilde\imath+\imath\geq0.
\]

\begin{rem}[On the introduction of $L^p_{s,\imath}(\cC^\omega_b)$]\label{remintroani}
The additional weight $\omk^\imath$, with $\imath \in \mathbb{R}_+^*$, in the definition of the anisotropic norm on
$L^p_{s,\imath}(\cC^\omega_b)$ is introduced precisely to compensate for the loss
$\omtilk^{\tilde\imath}$ arising in the  estimates above. The trade-off is that one must impose stronger regularity assumptions when working in
$L^p_{s,\imath}(\cC^\omega_b)$ rather than in $L^p_{s,0}(\cC^\omega_b)$.
\end{rem}


\subsection{Local well-posedness}\label{subsec:locmildsol}
Recall that $\tilde \imath \in ]0,1]$ and $0\leq \imath\leq\tilde\imath$.
We further assume that $\tilde s\geq 0$ and set
\[
\tilde m:=\tilde s-\tilde\imath+\imath\geq 0.
\]
We decompose the parameter region for $(\tilde p,\tilde m)$ as
\[
\bigl\{(\tilde p,\tilde m)\,;\,1\leq\tilde p,\ 0\leq\tilde m\bigr\}
=L_0\cup \mathfrak B\cup L_1 \cup \mathfrak Q \, ,
\]
where
\[
\begin{array}{ll}
L_0
:=\bigl\{(\tilde p,0)\,;\,1\leq\tilde p\bigr\},
& \text{is a horizontal half-line,}
\\[1ex]
\mathfrak B
:=\bigl\{(\tilde p,\tilde m)\,;\,
1<\tilde p,\ 0<\tilde p\,\tilde m<(\tilde p-1)d\bigr\},
& \text{is an open region,}
\\[1ex]
L_1
:=\bigl\{(1,\tilde m)\,;\,0<\tilde m\bigr\},
& \text{is a vertical half-line,}
\\[1ex]
\mathfrak Q
:=\bigl\{(\tilde p,\tilde m)\,;\,
1<\tilde p,\ (\tilde p-1)d\leq\tilde p\,\tilde m\bigr\},
& \text{is the remaining region.}
\end{array}
\]

\begin{theo}[Local well-posedness]\label{mildtotal}
Let $\tilde\imath$, $\imath$ and $\tilde s$ be as above and let
$(\tilde p,\tilde m)\in L_0\cup \mathfrak B \cup L_1 \cup \mathfrak Q $.
Assume that $q$ satisfies
\begin{equation}\label{determineq}
\left\{
\begin{array}{ll}
\displaystyle
q=\frac{\tilde p}{\tilde p-1},
& \text{if }(\tilde p,\tilde m)\in L_0,
\\[2ex]
\displaystyle
\frac{\tilde p}{\tilde p-1}
\leq q
<
\frac{\tilde p d}
{(\tilde p-1)d-\tilde p\tilde m},
& \text{if }(\tilde p,\tilde m)\in \mathfrak B,
\\[3ex]
\displaystyle
\frac{\tilde p}{\tilde p-1}
\leq q\leq+\infty,
& \text{if }(\tilde p,\tilde m)\in \mathfrak Q,
\\[2ex]
q=+\infty,
& \text{if }(\tilde p,\tilde m)\in L_1.
\end{array}
\right.
\end{equation}
Fix any
\[
(l,s)\in\mathbb N\times\mathbb R_+,
\]
and assume that, for some $M\in\mathbb R_+$, the initial data satisfy
\begin{equation}\label{inidataboundMi}
\|\tilde\alpha_0\|_{L^q_{l+s+\imath+1/2}}
\leq M,
\qquad
\|w_0\|_{W^{l,q}_{s,\imath}(\cC^\omega_b)}
\leq M.
\end{equation}
Then there exists $\cT>0$ such that the Fourier-Moyal initial value problem
\eqref{transporteqonw}--\eqref{inidatatransporteqonw}
admits a unique solution
\[
w\in L^\infty\bigl([0,\cT];
W^{l,q}_{s,\imath}(\cC^\omega_b)\bigr).
\]
Moreover, $\cT$ can be chosen so that
\begin{equation}\label{boundMaftersolvingdeb}
\sup_{t\in[0,\cT]}
\|w(t,\cdot)\|_{W^{l,q}_{s,\imath}(\cC^\omega_b)}
\leq 2M.
\end{equation}
\end{theo}

\begin{proof} 
Our aim is to apply a fixed-point argument, for which we first need to set up an appropriate iterative scheme. The main point is to ensure that the nonlinear term
$\cB_{\tilde \chi}[w, \slw]$ belongs to the same functional space as $w$ and  $\slw$.
In view of \eqref{continuityYBevolutedestimates}, the first step is to control the
$L^p_{l+s-\tilde s+\tilde\imath}(L^\infty)$-norm in terms of the norm used in the iteration space.
Combining successively \eqref{defdenormcomegabimathaajouter-1} and \eqref{clearthatbis}, we can assert that
\begin{equation}\label{renot}
\begin{aligned}
\|w\|_{L^p_{l+s-\tilde s+\tilde\imath}(L^\infty)}
&\leq
\|w\|_{L^p_{l+s-\tilde s+\tilde\imath,0}(\cC^\omega_b)}
\\
&\leq
\|w\|_{L^p_{l+s-\tilde m,\imath}(\cC^\omega_b)}
\\
&\leq
\|w\|_{W^{l,p}_{s-\tilde m,\imath}(\cC^\omega_b)}.
\end{aligned}
\end{equation}
Now, we write  \eqref{continuityYBevolutedestimates} with $ q = r $ and $ p = \tilde p / (\tilde p - 1) $ to get
\begin{equation}\label{vvvcasegeneral} \parallel \cB_{\tilde \chi}[ w, \slw ] \parallel_{W^{l,q}_{s,\imath}  (\cC^{\omega}_{b})} \, \lesssim \, \parallel w \parallel_{W^{l,p}_{s-\tilde m,\imath} (\cC^\omega_b)} \, \parallel \slw \parallel_{W^{l,q}_{s,\imath} (\cC^{\omega}_b)} .
 \end{equation}
 When $\tilde m=0$, corresponding to the case $(\tilde p,\tilde m)\in L_0$, we simply choose in \eqref{vvvcasegeneral}
 \[ 
 q=p=\frac{\tilde p}{\tilde p-1}, 
 \] 
 and we obtain
\begin{equation}\label{vvvcasetildem=0}
\parallel \cB_{\tilde \chi}[ w, \slw ] \parallel_{W^{l,q}_{s,\imath}  (\cC^{\omega}_{b})} \, \lesssim \, \parallel w \parallel_{W^{l,q}_{s,\imath} (\cC^\omega_b)} \, \parallel \slw \parallel_{W^{l,q}_{s,\imath} (\cC^{\omega}_b)} .
\end{equation}
We now turn to the case $\tilde m>0$.  The case $\tilde p=1$, corresponding to $(\tilde p,\tilde m)\in L_1$, must be treated separately. Then, we can apply \eqref{continuityYBevolutedestimates} with $ p=q = r= +\infty $ together with \eqref{renot} where $ p = +\infty $ in order to get
\[\begin{aligned}
\parallel \cB_{\tilde \chi}[ w, \slw ] \parallel_{W^{l,\infty}_{s,\imath}  (\cC^{\omega}_{b})} 
&\lesssim \, 
\|w\|_{L^\infty_{l+s-\tilde s+\tilde\imath}(L^\infty)} \, \parallel \slw \parallel_{W^{l,\infty}_{s,\imath} (\cC^{\omega}_b)}
\\
&\lesssim \, \parallel w \parallel_{W^{l,\infty}_{s-\tilde m,\imath} (\cC^\omega_b)} \, \parallel \slw \parallel_{W^{l,\infty}_{s,\imath} (\cC^{\omega}_b)}
\\
&\lesssim \, 
\|w\|_{W^{l,\infty}_{s,\imath}(\cC^\omega_b)} \, \parallel \slw \parallel_{W^{l,\infty}_{s,\imath} (\cC^{\omega}_b)} .
\end{aligned} 
\]
When $\tilde m>0$ and $ \tilde p>1 $, so that $ p < +\infty $, the idea is to exploit the additional margin provided by $\tilde m$ in order to increase the integrability exponent $p$ appearing in \eqref{vvvcasegeneral}. As long as 
 \[ 
 0<p\tilde m=\frac{\tilde p}{\tilde p-1} \tilde m<d, 
 \] 
 which corresponds to the region $\mathfrak B$, the weighted embedding 
 \[ 
 \|w\|_{L^p_{s-\tilde m}(B)} \lesssim \|w\|_{L^q_s(B)} 
 \] 
 holds provided that \[ p\leq q<\frac{dp}{d-p\tilde m}. \] Indeed, the latter condition is equivalent to 
 \[ 
 \frac{d}{p}-\tilde m<\frac{d}{q}\leq\frac{d}{p}. 
 \] 
 The same argument, applied to the derivatives entering the anisotropic norms \eqref{defdenormcomegabimathaajouter}, yields 
 \[ 
 \|w\|_{W^{l,p}_{s-\tilde m,\imath}(\cC^\omega_b)} \lesssim \|w\|_{W^{l,q}_{s,\imath}(\cC^\omega_b)}, \qquad \text{ when } p\leq q<\frac{dp}{d-p\tilde m}. 
 \] 
 Combining this embedding with \eqref{vvvcasegeneral}, we recover the estimate \eqref{vvvcasetildem=0} in the $\mathfrak B$ region. Finally, when 
 \[ 
 d\leq p\tilde m, 
 \] 
 corresponding to the region $\mathfrak Q$, the upper restriction on $q$ disappears, and the same conclusion holds for every $q\in[p,+\infty]$. Consequently, for every $q$ satisfying the conditions prescribed in \eqref{determineq}, we obtain the bilinear estimate \eqref{vvvcasetildem=0}.

 Now, consider the Banach space $ E $ introduced below
\[ E:= L^\infty \bigl([0,T]; W^{l,q}_{s,\imath} (\cC^\omega_b)\bigr) \, , \qquad \parallel w \parallel_E \, := \sup_{t \in [0,T]} \ \parallel w (t,\cdot) \parallel_{W^{l,q}_{s,\imath} (\cC^\omega_b)}  . \]
With $ M $ as in (\ref{inidataboundMi}), we work on the complete metric space $ B_E (0,2M] $. We can interpret the expression
$ f_{\tilde \alpha_0} $ in \eqref{int.defdecLcQt} as a function of both variables $ k $ and $ \eta $. Then, the derivatives $ \part_\eta^\alpha f_{\tilde  \alpha_0} $ (with
$ 0 < \vert \alpha \vert $) as well as the semi-norm $ \vert f_{\tilde \alpha_0} \! \downharpoonright_\omega $ are simply zero. In view of
(\ref{defdenormcomegabdepare}) and (\ref{calculnormcontinuousconditionregular}), there remains
\begin{equation}\label{thereremains}
\parallel f_{\tilde \alpha_0} \parallel_{W^{l,q}_{s,\imath} (\cC^\omega_b)} = \parallel \omk^{l+\imath} \, f_{\tilde  \alpha_0} (k) \parallel_{L^q_{s}} =
\parallel \tilde \alpha_0 \parallel_{L^q_{l+s+\imath+(1/2)}} \leq M < + \infty .
\end{equation}
Coming back to the definition \eqref{int.defdecLcQt} of $ \cL_s $, by applying \eqref{vvvcasetildem=0} and then \eqref{thereremains}, we find that
\begin{equation}\label{l2totou1}
 \parallel \cL_s w \parallel_{W^{l,q}_{s,\imath} (\cC^\omega_b)} \lesssim \, M
\, \parallel w(s,\cdot) \parallel_{W^{l,q}_{s,\imath} (\cC^\omega_b)} .
\end{equation}
On the other hand, in view of
\eqref{ilfautajouter}, we have
\begin{equation}\label{l2totou2inter}
\parallel \cQ_s w \parallel_{W^{l,q}_{s,\imath} (\cC^\omega_b)} \lesssim \, \Bigl( \int_0^s \parallel g_w(r,s,\cdot) \parallel_{W^{l,q}_{s,\imath} (\cC^\omega_b)} \, dr \Bigr) \,
\parallel w(s,\cdot) \parallel_{W^{l,q}_{s,\imath} (\cC^\omega_b)} .
\end{equation}
Since $ g_w(r,s,\cdot) $ does not depend on $ \eta $, we can assert that
\[ \parallel g_w(r,s,\cdot) \parallel_{W^{l,q}_{s,\imath} (\cC^\omega_b)} = \parallel \omk^l \,  g_w(r,s,\cdot) \parallel_{L^q_{s+\imath}} = \,  \parallel w(r,\cdot) \parallel_{L^q_{l+s+\imath}} . \]
It follows that (for $ 0 \leq s \leq T $)
\begin{equation}\label{l2totou2}
\parallel \cQ_s w \parallel_{W^{l,q}_{s,\imath} (\cC^\omega_b)} \lesssim \, s  \ \parallel w \parallel_{E} \
\parallel w(s,\cdot) \parallel_{W^{l,q}_{s,\imath} (\cC^\omega_b)} .
\end{equation}
Now, consider the integral formulation (\ref{integralversionSKGS})-(\ref{integralversionSKGSPhi}), and especially the content of $ \Phi (w) $. Due to the transport part inside \eqref{KGB1intermediaire}, the map $ \Phi $ implies translations with respect to $ \eta $, by $ t \, k $ at the level of $ w_0 $ and by $ (t-s) \, k $ at the level of $ \cL_s w + \cQ_s w $. These two  operations have no effect on
$ \parallel w(t,k,\cdot) \parallel_{L^\infty} $ and $ \vert w(t,k,\cdot) \! \downharpoonright_\omega $. As long as $ 0 \leq t \leq T $, we obtain that
\[ \begin{array} {rl}
\displaystyle \parallel \Phi (w)(t,\cdot) \parallel_{W^{l,q}_{s,\imath} (\cC^\omega_b)} \! \! \! & \leq \, M \displaystyle + \int_0^t M
\parallel w(s,\cdot) \parallel_{W^{l,q}_{s,\imath} (\cC^\omega_b)} ds \\
\ & \displaystyle \qquad \ \, +  \parallel w \parallel_E \int_0^t s \, \parallel w(s,\cdot) \parallel_{W^{l,q}_{s,\imath} (\cC^\omega_b)}
ds \smallskip \\
\ & \leq \, M \ \bigl( 1 + 2 \, M \, T \, (1+T)\bigr) .
\end{array} \]
A bootstrap argument indicates that, for $ T $ small enough, namely such that $ 2 \, M \, T \, (1+T) < 1 $, the map $ \Phi $ sends $ B_E (0,2M] $ into itself. On the
other hand, a repetition of the above arguments indicates that, for some constant $ C \in \mathbb R_+ $, we have
\[ 
\parallel \Phi(w) - \Phi(\slw) \parallel_E \, \leq  C \, T \parallel w - \slw \parallel_E \, (M + \parallel w \parallel_E + \parallel \slw \parallel_E) \, , \qquad \forall (w,\slw) \in E^2 . 
\]
For $T$ sufficiently small, namely for
\[
T \leq \frac{1}{6CM},
\]
the map $\Phi$ is a contraction on $ B_E (0,2M] $. Hence, it admits a unique fixed point, which yields a solution to the integral formulation \eqref{integralversionSKGS}--\eqref{integralversionSKGSPhi} of the Fourier--Moyal equation \eqref{transporteqonw}--\eqref{inidatatransporteqonw}. By construction, this solution satisfies the a priori bound \eqref{boundMaftersolvingdeb}.

Moreover, since $\Phi$ is locally Lipschitz, an application of Gronwall's inequality yields unconditional uniqueness: the solution is unique in $E$, without the additional assumption of the a priori bound \eqref{boundMaftersolvingdeb}. Finally, one can define a maximal solution in the usual way, together with the corresponding blow-up criterion.
\end{proof}


\section{$L^P $-prepared solutions to the Fourier--Moyal equation} \label{sec:L2preparedsol}
For simplicity, throughout this section we restrict ourselves to the
three-dimensional case $d=3$.  Our objective is twofold: to construct weak
solutions to the $\mathrm{FM}$ equations directly, and to establish a precise
correspondence between these solutions and the associated weak solutions of the
$\mathrm{SKG}$ system. To this end, we fix $ \chi $, $ \tilde \alpha_0 $, $ u_0 $ and $ w_0 $. We assume that $ u_0 \in L^p $ and that $ w_0 = (2 \pi)^{-3} \, \cF \text{W}(u_0)  $. The choice of the exponent $p$ is dictated by the available SKG well-posedness
theory. On the one hand, one may take $p=2$, since the $L^2$ norm of solutions
to $\mathrm{SKG}$ is conserved \cite{MR519635}. On the other hand, taking into
account the dispersive estimates established in \cite{MR778979}, one may also
consider the case $p=4$.
 The issue is to:
\begin{itemize}
 \item solve globally in time the initial value problem built with $ \text{SKG}_{\chi} $ (resp. $ \text{FM}^{\tilde \alpha_0}_{\tilde \chi} $) and the initial data $ u_0 $ (resp. $ w_0 $) to obtain the existence of a solution $ u $ (resp. $ w $);
 \item show that $ w $ remains $ L^p $-prepared, in the sense\footnote{By Lemma \ref{preparedpropagation}, this property is satisfied for smooth rapidly decreasing  solutions. But it is not granted for weak solutions.} that
 \begin{equation}\label{wellprepatt}
  w(t,\cdot) = (2 \pi)^{-3} \, \cF \text{W} \bigl(u(t,\cdot)\bigr) \, , \qquad u(t,\cdot) \in L^p \, , \qquad \forall \, t \in \RR_+ \, ;
 \end{equation}
\item prove that adequate norms remain under control at the level of $ w $ and, from there, deduce the propagation of microlocal properties related to $ u $;
\item Exploit this information to recover the uniqueness of the weak solutions $ u $ and $ w $.
\end{itemize}

\smallskip

In preparation for the global well-posedness result for weak solutions, we introduce in Subsection \ref{subsec:preparatorywork} several preliminary tools. The properties of the correspondence  between $u(t,\cdot)$ and $w(t,\cdot)$ depends on the assumptions imposed on $\chi$, $\tilde \alpha_0$, and $u_0$.  We impose condition \eqref{voevenchi} on $\chi$, which ensures that $\tilde \chi \in L^\infty$. We further assume that $\tilde \alpha_0 \in L^\infty_{1/2}$ and define $\alpha_0$ accordingly by
\[
\alpha_0 = \chi \,\tilde \alpha_0 .
\]
Then, in Subsections \ref{subsec:propagationinfty} and \ref{subsec:propagationl2}, we successively consider the cases $\tilde \chi \in L^1_+$ with $u_0 \in L^2$, and $\tilde \chi \in L^2$ with $u_0 \in L^2 \cap L^4$.


\subsection{Preparatory work}\label{subsec:preparatorywork}
In Paragraph \ref{subsubsec:approximationscheme}, we propose an approximation scheme allowing to solve simultaneously the $ \text{SKG} $ and $ \text{FM} $ equations. This furnishes two sequences of smooth solutions $ (u^n)_n $ and $ (w^n)_n $ satisfying \eqref{wellprepatt}.
In Paragraph
\ref{subsubsec:2FWtransfrom},
we explain why the nonlinear constraint  \eqref{wellprepatt} is weakly continuous with respect to $ L^2 $.
In Paragraphs
\ref{subsubsec:structureofprepared data} and \ref{subsubsec:structureofprepared data4}, we describe the regularity and integrability properties inherited by $ w $ when $ w = (2 \pi)^{-3} \, \cF \text{W} (u) $ with  respectively $ u \in L^2 $ and $ u \in L^4 $.


\subsubsection{Approximation scheme}\label{subsubsec:approximationscheme} Let $ \psi $ be a smooth even compactly supported function on $ \RR $, which is a symmetric positive mollifier, meaning that
\[ 0 \leq \psi \leq 1 , \qquad
\exists \, \delta \in ]0,1[ \, ; \quad \psi_{\mid [-\delta,\delta]} \equiv 1 , \qquad \psi_{\mid [-1,1]^c} \equiv 0 \, , \qquad \int_\RR \psi(r) \, dr = 1 \, . \]
For $ n \in \NN^* $, we denote by $ \psi^n \in \cC^\infty_c (\RR^3) $ its rescaled $ L^1 $-version on $ \RR^3 $, which is given by
\begin{equation}\label{apprschemes}
 \psi^n (x) := \frac{n^3}{c} \, \psi \bigl( n \, \vert x \vert \bigr) \, , \qquad c := 4 \, \pi \, \int_0^{+\infty} \! \psi(r) \, r^2 \, dr  \, , \qquad \int_{\RR^3}  \psi^n(x) \, dx = 1 \, .
\end{equation}
We start with $ \chi $ as in \eqref{voevenchi} and with initial conditions $ \tilde \alpha_0 $ and $ u_0 $ adjusted such that
\begin{equation}\label{inidataboundMicutcut}
\exists M \in \RR_+ ; \qquad \parallel \chi \parallel_{L^\infty} \leq M \, , \qquad \parallel \tilde \alpha_0 \parallel_{L^\infty_{1/2}} \leq M \, , \qquad
\parallel u_0 \parallel_{L^2} \leq M \, .
\end{equation}
We can approximate $ \chi $, $ \tilde \alpha_0 $ and $ u_0 $ according to
\[ 
\chi^n (k) := \psi \Bigl( \frac{\vert k \vert}{n} \Bigr) \ \chi * \psi^n (k) \, , \quad \ \tilde \alpha_0^n (k) := \psi \Bigl( \frac{\vert k \vert}{n} \Bigr) \ \tilde \alpha_0 * \psi^n (k) \, , \quad \ u_0^n (x) := \psi \Bigl( \frac{\vert x \vert}{n} \Bigr) \ u_0 * \psi^n (x) \, . 
\]
By construction, we get that
\[ 
\chi^n (k) = \chi^n (-k) \, , \quad \chi^n \in \cC^\infty_c (\RR^3; \RR) \, , \quad \ \tilde \alpha^n_0 \in \cC^\infty_c (\RR^3; \CC) \, , \quad \ u^n_0 \in \cC^\infty_c (\RR^3; \CC) \, . 
\]
Define
\[ w^n_0 := (2 \pi)^{-3}\, \cF_{x,\xi} \circ \text{W} (u^n_0) \, , \qquad \alpha^n_0 := \chi^n \, \tilde \alpha^n_0 \, , \qquad  \tilde \chi^n (k) := \omk^{-1} \, \chi^n(k) \, . \]
For all $ n \in \NN^* $, we can retain the following bounds
\begin{equation}\label{inidataboundMicutcutreg} \
\parallel \chi^n \parallel_{L^\infty} \leq M \, , \quad \parallel \tilde \alpha_0^n \parallel_{L^\infty_{1/2}} \leq M \, , \quad \
\parallel u_0^n \parallel_{L^2} \leq M \, , \quad \
\parallel w_0^n \parallel_{L^2} \leq (2 \pi)^{d/2} \, M^2 .
\end{equation}
From $ \alpha^n_0 $, we can  deduce $ A^n_0 $ and $ A^n_1 $ as indicated in \eqref{recoveralpha0}. Then:
\begin{itemize}
  \item[$ \circ $] The $ \text{SKG}_{\chi^n} $-system \eqref{PSKG} with initial data $ (u_0^n,A_0^n,A_1^n) $ has a global smooth solution $ (u^n , A^n) $;
  \vskip 1mm
  \item[$ \circ $] The $ \text{FM}^{\tilde \alpha_0^n}_{\tilde \chi^n} $-system  \eqref{transporteqonw} with initial data $ w_0^n $ has a global smooth solution $ w^n $.
\end{itemize}

\noindent By Lemma \ref{preparedpropagation}, we can assert that
\begin{equation}\label{preparedpropagationattimetn}
\forall t \in \RR_+ \, , \qquad w^n (t,\cdot) = (2 \pi)^{-3} \, \cF_{x,\xi} \circ {\rm W} \bigl( u^n (t,\cdot) \bigr)  .
\end{equation}
Moreover, for all $ t \in \RR_+ $, we know that\footnote{See Lemma \ref{Propertiesofprepareddata}}
\begin{equation}\label{conservnn}
\qquad \forall \, n \in \NN_+ \, , \quad \ \parallel u^n (t,\cdot) \parallel_{L^2} \lesssim M \, , \quad \
\parallel w^n (t,\cdot) \parallel_{L^2} \lesssim M^2 \, , \quad \
\parallel w^n (t,\cdot) \parallel_{L^\infty} \lesssim M^2 .
\end{equation}


\subsubsection{Weak continuity on $ L^2 $ of the Fourier--Wigner transform}\label{subsubsec:2FWtransfrom} Let $ t \in \RR_+ $. The uniform bounds inside \eqref{conservnn} guarantee that, modulo the extraction of a subsequence (which is not specified), the two
sequences $ \bigl(u^n(t,\cdot) \bigr)_n $ and $ \bigl( w^n (t,\cdot) \bigr)_n $ converge weakly in $ L^2 $ to functions $ u (t,\cdot) $ and $ w (t,\cdot) $. A first step is to check that $ w (t,\cdot) $ is necessarily $ L^2 $-prepared.

\begin{lem} \label{lemmedeweakconti}
Let $ (u^n )_n $ be a sequence that converges  weakly in $ L^2 (\RR^3) $ to the function $ u $. Then, the sequence $ (w^n)_n $ with $ w^n = (2 \pi)^{-3} \, \cF_{x,\xi} \circ {\rm W} (u^n) $ converges weakly in $ L^2 (\RR^3 \times \RR^3) $ to $ w $ with:
\begin{equation}\label{apreslimite}
 u \in L^2 (\RR^3) , \qquad w \in L^2 (\RR^3 \times \RR^3) , \qquad w = (2 \pi)^{-3} \, \cF_{x,\xi} \circ {\rm W} (u) \, .
\end{equation}
\end{lem}

\begin{proof} The starting point is \eqref{autreformulapourwbis}, which guarantees that
\[ w^n (k,\eta) = \int_{\RR^3} e^{-i k \cdot x} \, u^n \Bigl( x - \frac{\eta}{2} \Bigr) \, \bar u^n \Bigl( x + \frac{\eta}{2} \Bigr) \, dx . \]
The weak limit of a product (as above) is not necessarily the product of the weak limits. The last relation inside (\ref{apreslimite})
cannot be directly deduced without exploiting the special (tensor) structure appearing in the integrand. Let $ \varphi \in \cS (\RR^3 \times \RR^3)  $. We can test $ w_n $ against
$ \varphi $ to obtain
\[ \int_{\RR^3} \! \int_{\RR^3} w^n (k,\eta)  \, \varphi (k,\eta) \, dk d\eta = \int_{\RR^3} \! \int_{\RR^3}  (\cF_k \varphi) \bigl( (y+z)/2,z-y \bigr) \,
u^n (y) \, \bar u^n (z) \, dy dz . \]
Since $ \cF_k \varphi $ is rapidly decreasing, modulo some arbitrarily small error term, we can work on a compact set $ K $, on which
the function $ \cF_k \varphi $ may be approximated in the sup norm by a sum of tensor products according to
\[ (\cF_k \varphi) \bigl( (y+z)/2,z-y \bigr) = \sum_{j=0}^{+\infty} \psi_j^1 (y) \, \psi_j^2 (z) , \qquad (\psi_j^1, \psi_j^2) \in \cC^0 (K)^2 . \]
By Fubini's theorem, for all $ J \in \NN $, we have
\[ \begin{array}{rl}
\displaystyle \lim_{n \rightarrow +\infty} \  \sum_{j=0}^J  \int_{\RR^3} \! \int_{\RR^3} \psi_j^1 (y) \, \psi_j^2 (z) \, u^n (y) \, u^n (z) \, dy dz \! \! \! &
\displaystyle = \lim_{n \rightarrow +\infty} \ \sum_{j=0}^J \Bigl( \prod_{l=1}^2 \int_{\RR^3} \psi_j^l (y) \, u^n (y) \, dy \Bigr) \\
\ & \displaystyle = \sum_{j=0}^J  \int_{\RR^3} \! \int_{\RR^3} \psi_j^1 (y) \, \psi_j^2 (z) \, u (y) \, u (z) \, dy dz \, .
\end{array} \]
Passing to the limit with an exhaustion of $ \RR^3 $ by compact sets  and by letting $ J $ goes to $ + \infty $, this yield
\renewcommand\arraystretch{2}
\[ \begin{array}{rl}
\displaystyle \lim_{n \rightarrow +\infty} \ \int_{\RR^3} \! \int_{\RR^3} \! \! \! & \displaystyle w^n (k,\eta)  \, \varphi (k,\eta) \, dk d\eta =
\int_{\RR^3} \! \int_{\RR^3} w (k,\eta)  \, \varphi (k,\eta) \, dk d\eta \\
\ & \displaystyle = \int_{\RR^3} \! \int_{\RR^3} \Bigl( \int_{\RR^3} e^{-i k \cdot x} \, u \Bigl( x - \frac{\eta}{2} \Bigr) \, \bar u
\Bigl( x + \frac{\eta}{2} \Bigr) \, dx \Bigr) \, \varphi (k,\eta) \, dk d\eta .
\end{array} \]
\renewcommand\arraystretch{1}

\noindent This means that $ w $ must coincide with $ (2 \pi)^{-3} \, \cF_{x,\xi} \circ {\rm W} (u) $ in $ \cS' (\RR^3 \times \RR^3) $. The last identity
inside (\ref{apreslimite}) is therefore also true in the sense of $ L^2 $-functions.
  \end{proof}


\subsubsection{Structure of $ L^2 $-prepared data}\label{subsubsec:structureofprepared data} 
In view of \eqref{apreslimite}, one might be led to believe that the $L^2$ condition on $w$ is the fundamental requirement. However, an $L^2$ bound alone is far from sufficient to characterize functions of the form
\[
w=(2\pi)^{-3}\,\cF\mathrm{W}(u), \qquad u\in L^2.
\]
In fact, the preparation condition \eqref{apreslimite} imposes several additional constraints on $w$, some of which are listed below.

\begin{lem}[Properties of $ L^2 $-prepared data] \label{Propertiesofprepareddata}   Assume that $ w = (2 \pi)^{-3} \, \cF {\rm W} (u) $
 with $ u \in L^2 (\RR^3) $. Then:
\begin{enumerate}
\item [(a)] The function $ w $ is in $ L^2 (\RR^3 \times \RR^3) $ with $ \parallel w \parallel_{L^2} \lesssim \parallel u
\parallel_{L^2}^2 $.
\item [(b)] The function $ w $ is in $ L^\infty (\RR^3 \times \RR^3) $ with $ \parallel w \parallel_{L^\infty} \leq \parallel u
\parallel_{L^2}^2 $.
\item [(c)] For all $ \eta \in \RR^3 $, the function $ w(\cdot , \eta) $ is in the Wiener algebra $ \cF (L^1) $.
\item [(d)] The function $ w $ is uniformly continuous on $ \RR^3 \times \RR^3 $, and it vanishes at infinity.
\item [(e)] There exists a concave bounded modulus of continuity $ \omega $ which is adjusted in such a way that $ w $ is in $ L^\infty (\cC^\omega_b) $ with moreover $ \parallel w \parallel_{L^\infty(\cC^\omega_b)} \lesssim \parallel u
\parallel_{L^2}^2 $.
\end{enumerate}
\end{lem}

\begin{proof} Since $ w = (2 \pi)^{-3} \, \cF W (u,u) $, the point {\it (a)} is a direct consequence of \eqref{crosswignermoyalidentityb}. From \eqref{autreformulapourwbis}, by H\"older's inequality, we can easily deduce {\it (b)} and {\it (c)}. Let $ h \in \mathbb R^3 $. Observe that
\[ \begin{array}{rl}
\vert w(k , \eta + 2 h) - w(k , \eta)  \vert \leq \! \! \! &  \displaystyle \int_{\RR^3} \vert u ( x - h) \, \bar u ( x + \eta+h) - u ( x) \, \bar u ( x + \eta) \vert \ dx  \smallskip \\
\leq \! \! \! & \displaystyle \int_{\RR^3} \vert u ( x - h) - u ( x) \vert \, \vert \bar u ( x + \eta +h) \vert \ dx \smallskip \\
& \displaystyle + \int_{\RR^3} \vert u ( x) \vert \, \vert \bar u ( x + \eta +h) - \bar u ( x + \eta ) \vert \ dx \, .
\end{array}
\]
Hence, we have
\begin{equation}\label{hencewehave}
\quad \vert w(k , \eta + 2 h) - w(k , \eta)  \vert \leq \parallel u \parallel_{L^2} \, \bigl( \parallel u ( \cdot - h) - u (\cdot) \parallel_{L^2} + \parallel u ( \cdot + h) - u (\cdot) \parallel_{L^2} \bigr) \, .
\end{equation}
The right-hand side goes to $ 0 $ when $ \vert h \vert $ tends to $ 0 $, showing the continuity of $ w(k,\cdot) $.
Due to the inclusion $ \cF (L^1) \hookrightarrow \cC^0 $, the function $ w(\cdot,\eta) $ is also continuous. Now, when $ u $ is smooth (say of class $ \cC^\infty $) with support inside the ball $ B(0,R] $, the support of $ w $ is in the
cylinder $ \RR^3 \times B(0,2R] $. On the other hand, it is easy to show (by integration by parts) that $ w $ is (uniformly in $ \eta $)
rapidly decreasing when $ \vert k \vert $ goes to infinity. For $u \in L^2(\RR^3)$, a standard density argument shows that $w$ still vanishes at infinity. Since $w$ is continuous, it follows that $w$ is uniformly continuous on the whole phase space $\RR^3 \times \RR^3$. This proves {\it (d)}.

\smallskip

\noindent Statement (e) provides more quantitative information. If $u \equiv 0$, then $w \equiv 0$, and there is nothing to prove. We may therefore assume that $u \not\equiv 0$. Define the function
\[ \mathbb R_+ \ni s \mapsto \tilde \omega (s) := \frac{1}{\parallel u \parallel_{L^2}} \ \sup_{\vert h \vert \leq s} \ \parallel u ( \cdot + h) - u (\cdot) \parallel_{L^2} \, , \qquad \lim_{s \rightarrow + \infty } \ \tilde \omega (s) := 2 \, , \]
and consider the map
\begin{equation}\label{translatioonmap}
\tau_u : \mathbb R^3 \rightarrow L^2(\mathbb R^3) \, , \qquad \tau_u(h) := u(\cdot-h) \, .
\end{equation}
We can apply  \eqref{defdenormcomegabdepare} with $ B \equiv L^2 (\RR^3) $. Since the $ L^2 $-norm is invariant by translation and due to the definition of $ \tilde \omega $, we find that
\begin{equation}\label{normtrinvariant} \qquad \vert \tau_u \! \downharpoonright_{\tilde \omega} \, = \sup_{\eta \in \mathbb R^3} \ \sup_{0 < \vert h \vert} \frac{\parallel \tau_u(\eta+h)- \tau_u (\eta) \parallel_{L^2}}{\tilde \omega (\vert h \vert)} = \sup_{0 < \vert h \vert} \frac{\parallel u(\cdot -h)- u (\cdot) \parallel_{L^2}}{\tilde \omega (\vert h \vert)} \leq \parallel u \parallel_{L^2} .
\end{equation}
By construction, the function $ \tilde \omega $ is non-zero, increasing on $ \mathbb R_+ $, and such that $ \tilde \omega (0) = 0 $. It can be replaced by its concave hull $ \omega $ which can serve as a modulus of continuity. Since
$ \tilde \omega \leq \omega $ and because $ \tilde \omega (\vert h \vert/2) \leq \tilde \omega (\vert h \vert ) $, we have
\[ \vert w(k,\cdot) \! \downharpoonright_{\omega} \, \leq \vert w(k,\cdot) \! \downharpoonright_{\tilde \omega} \leq \sup_{\eta \in \mathbb R^3} \ \sup_{0 < \vert h \vert} \frac{ \vert w(k,\eta+h)-w(k,\eta) \vert}{\tilde \omega (\vert h \vert/2)} \, .  \]
Exploiting  \eqref{hencewehave} together with \eqref{normtrinvariant}, it follows that
\[ \vert w(k,\cdot) \! \downharpoonright_{\omega} \, \leq 2 \ \parallel u \parallel_{L^2} \ \sup_{0 < \vert h \vert} \frac{\parallel u(\cdot -h/2)- u (\cdot) \parallel_{L^2}}{\tilde \omega (\vert h \vert/2)} = 2 \ \parallel u \parallel_{L^2} \ \vert \tau_u \! \downharpoonright_{\tilde \omega} \leq 2 \ \parallel u \parallel_{L^2}^2 . \]
With {\it (b)}, this furnishes
\begin{equation}\label{expliinfcre}
 \parallel w \parallel_{L^\infty(\cC^\omega_b)} 
\leq 2 \, \parallel u \parallel^2_{L^2} \, < + \infty \, .
\end{equation}
By this way, we recover (e).
\end{proof}

\noindent It is worth noting that the preceding construction of $ \omega $  depends on $ u $. Indeed, given $ u \in L^2(\mathbb R^3) $, the definition of $ \tilde \omega $ (and therefore of $ \omega $) is based  on $ u $. Now, given some $ \omega $, it is interesting to investigate the meaning in terms of $ u $ of
$ \vert w(k,\cdot) \! \downharpoonright_{\omega} $. Clearly, the control of $ \vert w(k,\cdot) \! \downharpoonright_{\omega} $ provides some  information on $ u $ that is not only inherited from the $ L^2 $-bound on $ u $. This is confirmed by the remark below.

\begin{rem}[Some regularity and concentration properties detected by the semi-norm $ \vert \, \cdot \! \downharpoonright_{\omega} $] \label{detectseminorm}
Recall that
\begin{equation}\label{transrot}
\forall \, (z,\zeta) \in \RR^3 \times \RR^3 \, , \qquad \cF {\rm W} \bigl( e^{i \zeta \cdot x} u(x-z) \bigr) = e^{-i \zeta \cdot \eta -i k \cdot z} \,
\cF {\rm W} (u) (k,\eta) \, .
\end{equation}
  The above operation does not modify the $ L^2 $-norm of $ u $. But, given $ \omega $,  the modulation operator can increase the size of $ \vert w \! \downharpoonright_{\omega} $, especially when $ \vert \zeta \vert $ is large. This indicates that the quantity $ \vert w \! \downharpoonright_{\omega} $ measures the regularity of $ u $. The expression $ \vert w \! \downharpoonright_{\omega} $ also detects the presence of concentration. Indeed, select a function $ u $ whose $ L^2 $-norm is one, and consider its $ L^2 $-rescaled version
 $$ u_\eps (x) := \frac{1}{\eps^{d/2}} \ u \Bigl( \frac{x}{\eps}\Bigr) \, , \qquad w_\eps (k,\eta) := \cF \, {\rm W} (u_\eps) (k,\eta) \, , \qquad w \equiv w_1 \, , \qquad \eps \in ]0,1] \, . $$
Just take $ k = 0 $. Since $ \vert w(0,\cdot) \! \downharpoonright_{\omega} \not = 0 $, we can find $ (\eta_1,\eta_2) \in \RR^3 \times \RR^3 $ such that
\[ \begin{array}{rl}
\displaystyle \vert w_\eps (0, \eps \, \eta_2) - w_\eps(0,\eps \, \eta_1) \vert \! \! \! & \displaystyle = \vert  \int_{\RR^3}
u_\eps ( x) \ \bigl \lbrack u_\eps ( x + \eps \, \eta_2) - u_\eps ( x + \eps \, \eta_1) \bigr \rbrack \ dx \, \vert \smallskip \\
\ & \displaystyle =
 \vert w(0,\eta_2) - w(0,\eta_1) \vert \not = 0 \, .
 \end{array}
\]
It follows that
\[ \vert w_\eps (0,\cdot) \! \downharpoonright_{\omega} \,  \geq \frac{\vert w_\eps (0,\eps \, \eta_2) - w_\eps (0,\eps \, \eta_1) \vert}{\omega \bigl( \eps \, \vert \eta_2 -\eta_1 \vert \bigr)} \geq \frac{\vert w (0,\eta_2) - w (0, \eta_1) \vert}{\omega \bigl( \eps \, \vert \eta_2 -\eta_1 \vert \bigr)} \, , \]
showing that $ \vert w_\eps (0,\cdot) \! \downharpoonright_{\omega} $ becomes as large as wanted when $ \eps \in \mathbb R_+^* $ goes to zero. Looking at $ \vert w \! \downharpoonright_{\omega} $ is therefore also a way to control how $ u $ can  concentrate in space.
\end{rem}

\noindent The first global solutions to the Schr\"odinger--Klein--Gordon system
(that is, $\mathrm{SKG}_{\indicf}$) were constructed by Baillon and Chadam
\cite{MR519635,MR515899} in the functional setting $ \cC^0 \bigl(\mathbb R; H^2(\mathbb R^3)) $. Motivated by this, we describe below how the $H^2$-regularity of $u(t,\cdot)$ translates into properties of $w(t,\cdot)$. More generally, we characterize, at the level of $w(t,\cdot)$, the counterpart of the $H^l$-regularity of $u(t,\cdot)$.

\begin{lem}[Properties of  prepared data coming from higher levels of Sobolev regularity] \label{Propertiesofhigherprepareddata}   Assume that $ w = (2 \pi)^{-3} \, \cF_{x,\xi} \circ {\rm W} (u) $
 with $ u \in H^l (\RR^3) $ and $ l \in \NN^* $. Then, we can find a concave modulus of
continuity $ \omega $ such that $ w \in W^{l,\infty}_{0,0} (\cC^\omega_b) $.
\end{lem}

\begin{proof}
\noindent In view of the definition (\ref{calculnormcontinuousconditionregular}),
for $ \vert \alpha \vert \leq l $, using Leibniz rule, we have to compute
\[ \begin{array}{rl}
\displaystyle \omk^{l-\vert \alpha \vert} \, \part^\alpha_\eta w (k,\eta) = \frac{1}{2^{\vert \alpha \vert}} \sum_{\beta \leq \alpha} (-1)^{\vert \beta \vert}
\left( \begin{array}{c}
\alpha \\
\beta
\end{array} \right) \int_{\RR^3} \! \! \! & (1-\Delta_x)^{(l-\vert \alpha \vert)/2} \bigl( e^{-i k \cdot x} \bigr) \\
\ & \displaystyle \times \, (\part^\beta_x u) \Bigl( x - \frac{\eta}{2} \Bigr) \, (\part^{\alpha-\beta}_x \bar u) \Bigl( x + \frac{\eta}{2} \Bigr) \, dx .
\end{array} \]
Integration by parts allow to transfer the $ l-\vert \alpha \vert $ derivatives of the oscillatory factor $ e^{-i k \cdot x} $ on the right hand
side, leading to at most $ l $ spatial derivatives distributed between $ u $ and $ \bar u $. By this way, we recover the discussion leading
to Lemma \ref{Propertiesofprepareddata}-$ (e) $. By implementing the concave hull $ \omega $ of the moduli of continuity associated to the derivatives (up to the order $ l $) of $ u $, we obtain that  $ w \in
W^{l,\infty}_{0,0} (\cC^\omega_b) $.  Note that the $ \omega $ of Lemma \ref{Propertiesofhigherprepareddata} may be different from the one of Lemma \ref{Propertiesofprepareddata}.
\end{proof}


\subsubsection{Structure of $ (L^2 \cap L^4) $-prepared data}\label{subsubsec:structureofprepared data4} When $ w $ is as in \eqref{apreslimite} with moreover $ u \in L^4 (\RR^3) $,
some additional properties become available.

\begin{lem}[Properties of $ (L^2 \cap L^4) $-prepared data] \label{Propertiesofprepareddata24} Fix any $ u \in (L^2 \cap L^4) (\RR^3) $, and assume that $ w = (2 \pi)^{-3} \, \cF {\rm W} (u) $. Then, in addition to (a), $ \cdots $, (e), we can assert that:
\begin{enumerate}
\item [(f)] There exists a concave bounded modulus of continuity $ \omega $ which is adjusted in such a way that $ w $ is in $ \cC^\omega_{b,\eta}(L^2_k) $ with moreover $ \parallel w \parallel_{\cC^\omega_{b,\eta}(L^2_k)} \lesssim \parallel u
\parallel_{L^4}^2 $.
\end{enumerate}
\end{lem}

\begin{proof} Unlike before, when looking at $ \cC^\omega_{b,\eta}(L^2_k) $, the roles of $ k $ and $ \eta $ are reversed. From  \eqref{autreformulapourwbis} and Plancherel theorem, we can assert that
\[ \forall \, \eta \in \RR^3 \, , \qquad \parallel w(\cdot,\eta)\parallel_{L^2} \lesssim \parallel u(\cdot-\eta/2) \, \bar u (\cdot+\eta/2) \parallel_{L^2} \leq \parallel u\parallel_{L^4}^2 \, . \]
Let $ h \in \mathbb R^3 $.  From the relation  \eqref{autreformulapourwbis}, we get also that
\[ \begin{array}{rl}
w(k , \eta + 2 h) - w(k , \eta)  = \! \! \! &  \displaystyle \int_{\RR^3} e^{-i \, k \cdot x} \ \Bigl \lbrack u \Bigl( x - \frac{\eta}{2} - h \Bigr) - u \Bigl( x - \frac{\eta}{2} \Bigr) \Bigr \rbrack \ \bar u \Bigl( x + \frac{\eta}{2} + h \Bigr) \ dx \smallskip \\
\ & \displaystyle + \int_{\RR^3} e^{-i \, k \cdot x} \ u \Bigl( x - \frac{\eta}{2} \Bigr) \
\Bigl \lbrack \bar u \Bigl( x + \frac{\eta}{2} + h \Bigr) - \bar u \Bigl( x + \frac{\eta}{2} \Bigr) \Bigr \rbrack \ dx
\end{array}
\]
Hence, by Plancherel theorem, we have
\begin{equation}\label{hencewehavebis}
\parallel w(\cdot , \eta + 2 h) - w(\cdot , \eta)  \parallel_{L^2} \lesssim \, \parallel u \parallel_{L^4} \, \bigl( \parallel u ( \cdot - h) - u \parallel_{L^4} + \parallel u ( \cdot + h) - u \parallel_{L^4} \bigr) \, .
\end{equation}
From there, the proof
follows the same lines as
the one of Lemma \ref{Propertiesofprepareddata}. It suffices to replace $ \tilde \omega $ by
\[ \mathbb R_+ \ni s \mapsto \check \omega (s) := \frac{1}{\parallel u \parallel_{L^4}} \ \sup_{\vert h \vert \leq s} \ \parallel u ( \cdot + h) - u \parallel_{L^4} \, , \qquad \lim_{s \rightarrow + \infty } \ \check \omega (s) := 2 \, , \]
and to introduce the concave hull $ \omega $ of $ \check \omega $.
\end{proof}


\subsection{Construction of weak solutions to Fourier--Moyal equation}
\label{subsec:propagationinfty}
In this section, we establish a first link between $ \text{SKG}_{\chi} $ and FM equations in the context of $ L^2 $-data. To this end, we will have to impose a {\it strong cutoff} in the sense that $ \tilde \chi \in L^1_+ $, which is equivalent to
\begin{equation}\label{strongcutoff}
\exists \, \tilde s>0 \, ; \quad \tilde \chi = \omk^{-1} \, \chi^2 \in L^1_{\tilde s}
\quad \Longleftrightarrow \quad \chi \in L^2_{-(1/2)+} = \cup \, \bigl \lbrace L^2_s \, ; \, s  > -1/2 \bigr \rbrace \, .
\end{equation}
Let $ T \in \mathbb R_+^* $.
Consider a solution $ u \in L^\infty([0,T];L^2_x) $  to $ \text{SKG}_{\chi} $ whose $ L^2 $-norm is conserved. From Lemma \ref{Propertiesofprepareddata}-(e),
for all $ t \in \mathbb [0,T] $, we can extract from $ u(t,\cdot) $ a modulus of continuity $ \omega_t $ which is such that
\[\forall \, t \in [0,T] \, , \qquad \parallel w(t,\cdot) \parallel_{L^\infty(\cC^{\omega_t}_b)} \lesssim \parallel u_0 \parallel_{L^2}^2 \, . \]
The question is whether  $ \omega_t $ can be replaced by $ \omega_0 \, $? In such case, some  microlocal information (independent from the time evolution) concerning $ u $ is propagated because the instrument $ \omega \equiv \omega_0 $ used to measure  the $ L^\infty(\cC^{\omega}_b) $-norm of $ u(t,\cdot) $ is no longer dependent on $ u(t,\cdot) $ but only on the  initial data $ u_0 $.
With this in mind, given the fixed modulus of continuity $ \omega \equiv \omega_0$, it is interesting to investigate how the quantity $ \parallel w \parallel_{L^\infty(\cC^\omega_b)} $ is modified by the Fourier--Moyal equation.
In other words, we want to give a sense to the following diagram (D):
\[ 
\begin{tikzcd}
L^2 (\mathbb R^2) \ni \arrow[d] & u_0
 \arrow[r,"\text{Schr\"odinger}"] \arrow[d,"\displaystyle \begin{array}{c}
\text{$ \ $} \\
\text{                                                                    Fourier--Wigner}\\
\text{transform}\\
\text{$ \ $}
\end{array}"'] &
u(t,\cdot) \arrow{d}{\displaystyle \begin{array}{c}
\text{$ \ $} \\
\text{                                                                    Fourier--Wigner}\\
\text{transform}\\
\text{$ \ $}
\end{array}} & \in L^2(\mathbb R^3) \arrow[d]
\\
L^\infty(\cC^\omega_b) \ni & w_0 = (2 \pi)^{-3} \cF {\rm W} (u_0)
 \arrow[r,"\text{ Moyal }"] &  w(t,\cdot) = (2 \pi)^{-3} 
\cF {\rm W} \bigl(u(t,\cdot) \bigr) & \in L^\infty(\cC^\omega_b)
\end{tikzcd} 
\]
The idea is to exploit the bottom arrow in order to extract some information about $ \text{SKG}_{\chi} $. Recall that the set $ \mathscr W_{\tilde \imath} $ is defined according to \eqref{modconw}, whereas
$ \tau_u $ is given by  \eqref{translatioonmap}.
Given $ \omega \in \mathscr W_{\tilde \imath} $,  define the vector space
\begin{equation}\label{vectoespaceEo}
 E_\omega := \bigl \lbrace \, u \in L^2(\mathbb R^3) \, ; \, \parallel \tau_u \parallel_{\cC^\omega_b (L^2(\mathbb R^3))} <  + \infty \, \bigr \rbrace \, ,
\end{equation}
equipped with the distance
\begin{equation}\label{distanceequi}
 d_\omega (u,\tilde u) := \,  \parallel \cF  \text{W}(u) - \cF  \text{W}(\tilde u) \parallel_{L^\infty(\cC^\omega_b)} .
\end{equation}
Combining Lemma \ref{Propertiesofprepareddata} and Remark \ref{remtildeimath}, we can assert that
\begin{equation}\label{coverofL2}
 L^2(\mathbb R^3) = \cup \, \lbrace E_\omega \, ; \, \omega \in \mathscr W_{\tilde \imath}\rbrace \, .
\end{equation}
This cover of $ L^2(\mathbb R^3) $ by the subspaces $ E_\omega $ gives a way to control all $ L^2 $-solutions to $ \text{SKG}_{\chi} $ by following the time evolution of $ d_\omega $ (with $ \omega \equiv \omega_0 $) at the level of the Fourier--Moyal equation.

\begin{theo}[Global existence, uniqueness, and stability for $L^2$-prepared data]\label{GESPD}
Assume that $\tilde \chi \in L^1_+$. Let $\tilde \alpha_0 \in L^\infty_{1/2}(\mathbb R^3)$, and choose $(A_0,A_1)$ according to \eqref{SKGreinter2ini}, with
\[
\alpha_0=\chi\,\tilde \alpha_0.
\]
Then the following statements hold:
\begin{itemize}
\item For every $u_0\in L^2(\mathbb R^3)$, the Cauchy problem \eqref{PSKG}--\eqref{inidataSKGinif} for the $\mathrm{SKG}_{\chi}$ system admits a unique global solution 
which is such that
\[
u\in L^\infty\bigl(\mathbb R_+;L^2(\mathbb R^3)\bigr) \, , \qquad (A, \part_tA) \in L^\infty \bigl(\mathbb R_+;H^{1/2+}(\mathbb R^3) \times H^{-1/2+}(\mathbb R^3)\bigr) .
\]
\item For every
\[
w_0=(2\pi)^{-3}\,\cF\mathrm{W}(u_0),
\qquad u_0\in L^2(\mathbb R^3),
\]
the Cauchy problem \eqref{transporteqonw}--\eqref{inidatatransporteqonw} for the $\mathrm{FM}^{\tilde \alpha_0}_{\tilde \chi}$ equation admits a unique global solution
\[
w\in \cC^0\bigl(\mathbb R_+;L^\infty(\mathbb R^3\times\mathbb R^3)\bigr)
\]
satisfying \eqref{wellprepatt}.
\end{itemize}

\noindent Moreover, given $\tilde s$ as in \eqref{strongcutoff}, for every $\omega\in\mathscr W_{\tilde \imath}$ with $0<\tilde \imath\leq\tilde s$ and every $t\in\mathbb R_+$, the flow map
\begin{equation}\label{flowmapp}
\Phi_t:E_\omega\longrightarrow E_\omega,
\qquad
\Phi_t(u_0):=u(t,\cdot),
\end{equation}
is well defined and locally Lipschitz with respect to the distance $d_\omega$.
\end{theo}
\begin{rem} 
In view of \eqref{coverofL2},
we can find $ \omega_0 \in\mathscr W_{\tilde \imath} $ and $ \tilde \omega_0 \in\mathscr W_{\tilde \imath} $ such that $ u_0 \in E_{\omega_0} $ and $ \tilde u_0 \in E_{\tilde \omega_0} $. Let $ \omega $ be the modulus of continuity whose graph is the concave hull of the graph of the function $ \max \, (\omega_0 ; \tilde \omega_0) $. By construction, on $ [0,\bar z] $ with $ 0 < \bar z $ small enough, we have
\[ \begin{array}{ll}
0 \leq c \, z^{\tilde \imath} \leq \omega_0  (z) \leq \omega (z) \, , \quad & \forall z \in [0,\bar z] , \\
0 \leq \tilde c \, z^{\tilde \imath} \leq \tilde \omega_0  (z) \leq \omega  (z) , \ & \forall z \in [0,\bar z] , 
\end{array} \]
and therefore
\[ 0 \leq \min(c;\tilde c) \, z^{\tilde \imath} \leq \omega (z) \, , \quad \forall z \in [0,\bar z] . \]
This means that $ \omega $ is in $ \mathscr W_{\tilde \imath}  $. Moreover, since $ \omega $ is a relaxed version of both $ \omega_0 $ and $ \tilde \omega_0 $, we can assert that $ u_0 \in E_\omega $ and $ \tilde u_0 \in E_\omega $. This choice of $ \omega $ is (at time $ t=0 $) a tool which allows to control simultaneously the $ L^\infty(\cC^\omega_b) $-regularity of $ \cF {\rm W} (u_0)$ and $ \cF {\rm W} (\tilde u_0) $. Knowing that $ (u_0,\tilde u_0)\in E_\omega^2 $, it follows from Theorem \ref{GESPD} that the distance $d_\omega$ between the corresponding solutions $u(t,\cdot)$ and $\tilde u(t,\cdot)$ is controlled in terms of the initial distance $d_\omega(u_0,\tilde u_0)$. Notice, however, that this stability estimate is not uniform on $L^2(\RR^3)$, as it depends on the choice of an adequate subspace $E_\omega$ containing the initial data.
\end{rem}

\begin{proof} 
We begin with a few comments regarding the existence theory related to mild solutions of $\mathrm{SKG}_{\chi}$. This subject can fall within the scope of \cite{MR2035502}, whose approach gives access (for $ SKG $) to solutions satisfying 
$ u \in L^\infty(\mathbb R_+;L^2 ) $. From there, let us describe the regularity of $ (A, \part_tA) $. We can seek $ \alpha $ in the form $ \alpha = \chi \, \tilde \alpha $.
Exploiting \eqref{SKGSyequibis}, we deduce that 
\begin{equation}\label{sec5:evol_eq_alpha_tilde}
 \partial_t \bigl( \omk^{1/2} \, e^{i \omk t} \, \tilde \alpha \bigr) = - i \, (2 \pi)^{-3/2} \, e^{i \omk t} \, \cF_x \bigl( \vert u(t,\cdot) \vert^2 \bigr)(k) \in L^\infty_{loc}\bigl(\RR_+; L^\infty(\RR^3) \bigr) , 
 \end{equation}
while by assumption, we start with
\[ \omk^{1/2} \,\tilde \alpha_{\mid t= 0} = \omk^{1/2} \,\tilde \alpha_0 \in L^\infty(\RR^3) . \]
It folows that $ \tilde \alpha (t,\cdot) \in L^\infty_{1/2} $ for all $ t \in \RR_+ $. Now, using \eqref{recoveralpha01}, we get that
\[ \omk^{1/2} \, \cF_x \bigl( A (t,\cdot) \bigr) = (2 \pi)^{d/2} \ \underbrace{\, \omk^{-1/2} \chi \,}_{\in L^2_{0+}} \ 
\underbrace{\, \cI \bigl(\omk^{1/2} \tilde \alpha (t,\cdot) \bigr)}_{\in L^\infty} \in L^2_{0+}. \]
This implies that $ A (t,\cdot) \in H^{1/2+}$ as claimed. A similar argument based this time on \eqref{recoveralpha02} allows to assert that  $ \part_t A (t,\cdot) \in H^{-1/2+}$.

We now adopt the Fourier--Moyal perspective.
We start with data satisfying \eqref{inidataboundMicutcut} for some $ M $. We exploit the approximation scheme of Paragrah \ref{subsubsec:approximationscheme} to obtain sequences $ (u^n)_n $ and $ (w^n)_n $
satisfying \eqref{preparedpropagationattimetn} and \eqref{conservnn}. Apply  Lemma \ref{Propertiesofprepareddata}-$ (e) $ to extract from $ u_0 $ a  modulus of continuity $ \omega $ yielding $ \parallel w_0 \parallel_{L^\infty(\cC^\omega_b)} \lesssim M^2 $.
As explained in the proof of Lemma \ref{Propertiesofprepareddata}, see especially \eqref{normtrinvariant}, the $ L^\infty(\cC^\omega_b) $-norm of $ w $ is controlled by the $ L^2 $-norm of $ u $ and 
$ \vert \tau_u \! \downharpoonright_{\tilde \omega} $. By young's inequality, we have
\[ \parallel u^n_0 \parallel_{L^2} \lesssim \parallel u_0 * \psi^n \parallel_{L^2} \lesssim \parallel u_0 \parallel_{L^2} \, , \qquad \vert \tau_{u^n_0} \! \downharpoonright_{\tilde \omega} \lesssim \vert \tau_{u_0 * \psi^n} \! \downharpoonright_{\tilde \omega}  \lesssim \parallel u_0 \parallel_{L^2} .
\]
As a consequence, we start with the uniform (with respect to $ n $) control
\begin{equation}\label{conservnnomega}
\qquad \forall \, n \in \NN_+ \, , \quad \
\parallel w^n_0  \parallel_{L^\infty(\cC^\omega_b)} \lesssim M^2 \, .
\end{equation}
Coming back to the definition \eqref{int.defdecLcQt} of $ \mathcal L_t $, we can apply Lemma \ref{continuityYBfirstcase} with $ \tilde \chi^n \in L^1_{\tilde s} $ (so that $ \tilde p = 1 $) and $ \tilde \iota = \tilde s \in \RR_+^* $. We select $ (l,s,\imath) = (0,0,0) $ and $ (p,q,r) = (\infty,\infty,\infty) $. This furnishes
\[ \parallel \cL_t w^n \parallel_{L^\infty (\cC^{\omega}_b)} \, \lesssim \, \parallel \tilde \chi^n \parallel_{L^1_{\tilde s}} \ \parallel \tilde \alpha_0^n \parallel_{L^\infty_{1/2}} \
\parallel w^n(t,\cdot) \parallel_{L^\infty (\cC^{\omega}_b)} \, \lesssim \, \parallel w^n(t,\cdot) \parallel_{L^\infty (\cC^{\omega}_b)} \, . \]
The same argument gives rise to
\[ \parallel \cB_{\tilde \chi} \bigl \lbrack g_{w^n}(s,t,\cdot), w^n (t,\cdot) \bigr \rbrack \parallel_{L^\infty (\cC^{\omega}_b)} \, \lesssim \, \parallel w^n (s,\cdot) \parallel_{L^\infty} \
\parallel w^n(t,\cdot) \parallel_{L^\infty (\cC^{\omega}_b)} \, . \]
Now, the a priori sup-norm estimate inside \eqref{conservnn} allows to break the quadratic growth on the right hand side. We can retain that
\[ \parallel \cB_{\tilde \chi} \bigl \lbrack g_{w^n}(s,t,\cdot), w^n (t,\cdot) \bigr \rbrack \parallel_{L^\infty (\cC^{\omega}_b)} \, \lesssim \,
\parallel w^n(t,\cdot) \parallel_{L^\infty (\cC^{\omega}_b)} \, . \]
Taking into account the integral version \eqref{integralversionSKGS}-\eqref{integralversionSKGSPhi}, we can find a constant $ C $ (not depending on $ n $ and $ t $) such that, for all $ t \in \RR_+ $, we have
\begin{equation}
 \label{foralltjgzefa;} \forall \, n \in \NN^*_+ \, , \qquad \parallel w^n(t,\cdot) \parallel_{L^\infty (\cC^{\omega}_b)} \leq C \, e^{C \, (1+t^2)} .
\end{equation}
We implement (the proof of) Theorem  \ref{mildtotal} with $ (\tilde p,\tilde m) = (1,0) $ so that $ q = +\infty $. We consider initial data $ w_0 $ contained in the ball of $ L^\infty (\cC^{\omega}_b) $ of radius $ C $. Then, for some $ \cT \in \RR_+^* $, there is a unique local solution $ w \in L^\infty \bigl([0,\cT]; L^\infty (\cC^\omega_b)\bigr) $ to the
$ \text{FM} $-equation starting from $ w(0,\cdot) = w_0 $. Moreover, on this ball, the $ \text{FM} $-flow is locally Lipschitz with respect to the $ L^\infty (\cC^{\omega}_b) $-norm. This implies that $ (w^n)_n $ converges on $ [0,\cT] $ to $ w $. Then, passing to the limit at the level of \eqref{foralltjgzefa;}, we recover that
\begin{equation}
 \label{foralltjgzefa;limit} \forall \, t \in [0,\cT] \,  , \qquad \parallel w (t,\cdot) \parallel_{L^\infty (\cC^{\omega}_b)} \leq C \, e^{C \, (1+t^2)} .
\end{equation}
Let $ \cT_m \in \RR_+ \cup \{+\infty \} $ be the lifespan associated with $ w $. The bound \eqref{foralltjgzefa;limit} remains true as long as $ w $ stays in $ L^\infty (\cC^{\omega}_b) $, that is on $ [0,\cT_m ]$. Since the theory of continuation of solutions shows that, blow-up is the only possibility for a maximal solution to be
defined only on a finite time interval, the solution is in fact global ($ \cT_m = + \infty $).

\smallskip

\noindent Coming back to
 \eqref{integralversionSKGS}-\eqref{integralversionSKGSPhi}, the bound \eqref{foralltjgzefa;} guarantees also that the $ w^n $ are uniformly locally Lipschitz in time with values in $ L^\infty (\cC^{\omega}_b) $, and hence locally in $ L^2 $. By Arzel\`a-Ascoli theorem, passing to the weak limit through Lemma \ref{lemmedeweakconti}, we recover that, for all $ t\in \RR_+ $, the expression $ w(t,\cdot) $ is as in \eqref{apreslimite}
for some $ u (t,\cdot) $. Since $ w $ is unique, the same applies to $ u $. On the other hand, from  \cite{MR2035502}, we know that $ (u^n)_n $ converge to the solution $ u $ of  the $ \text{SKG}_{\chi} $-system. Thus, we have found an alternative way (based on completely different ideas) of constructing unique weak $ L^2 $-solutions to the $ \text{SKG}_{\chi} $-system. The properties of $ \Phi_t $ are just a traduction of what happens at the level of $ w $. Incidentally, our method allows to exhibit the family of metrics $ d_\omega $ allowing to recover uniqueness.
\end{proof}


\subsection{Relaxation of the condition on the cutoff}\label{subsec:propagationl2} Knowing that $ u_0 \in L^2 \cap L^4 $, the aim here is to improve the constraint $ \tilde \chi \in L^1_+ $
inside Theorem \ref{GESPD} into $ \tilde \chi \in L^2 $. To this end, the main  difficulty is to pass to the weak limit in the FM-equation
when there is no (local) result of strong well-posedness (as could be provided by Theorem \ref{mildtotal}). As noted below, the uniform estimates inside \eqref{conservnn} are not sufficient to deal with the computation of a trace, as involved in the $ \cY \cB $-map.

\begin{rem}[Defect of weak continuity when composing $ \cF_{x,\xi} \circ {\rm W} $ with the trace
at $ \eta = 0 $] \label{lemmedeweakcontitrace} The framework is as in Lemma \ref{lemmedeweakconti}. By Lemma \ref{Propertiesofprepareddata}-(d), we can give
a meaning to
\[ \begin{array}{rcl}
Tr : L^2 (\RR^3) & \longrightarrow & L^\infty (\RR^3) \\
u & \longmapsto & \begin{array}{rcl}
Tr(u) : \RR^3 & \longrightarrow & \RR^3 \\
k &  \longmapsto & w(k,0) := (2 \, \pi)^{-3} \, \cF_{x,\xi} \circ {\rm W} (u) (k,0 ).
\end{array}
\end{array} \]
However, this map is not weakly continuous. To see why, it suffices to produce a counterexample. Consider the sequence
\[ u^n (x) := e^{i n \xi \cdot x} \tilde u(x) , \qquad n \in \NN, \qquad \xi \in \RR^3 \setminus \{ 0 \} , \qquad 0 \not \equiv
\tilde u \in \cC^0_0 .\]
In view of (\ref{transrot}), we have
\[ w^n (k,\eta) := (2 \, \pi)^{-3} \, \cF_{x,\xi} \circ {\rm W} (u^n) (k,\eta) = e^{-i n \xi \cdot \eta} \, \tilde w (k,\eta) , \qquad
\tilde w  := \cF_{x,\xi} \circ {\rm W} (\tilde u) \not \equiv 0 .  \]
The sequence $ (u^n)_n $ converges weakly in $ L^2 $ to $ u^\infty = 0 $ and, in coherence with Lemma \ref{lemmedeweakconti},
the sequence $ (w^n)_n $ converges weakly in $ L^2 $ to $ w^\infty = 0 $. On the other hand, we find
\[ \forall n \in \NN, \qquad Tr (u^n) (k) = w^n(k,0) = \tilde w (k,0) = Tr (\tilde u) (k) = \cF_x \bigl( \vert \tilde u \vert^2 \bigr)(k)
\not \equiv 0 =  Tr(u^\infty) . \]
Thus, the sequence $ \bigl( Tr(u_n) \bigr)_n $ does not converge weakly-$ \star $ in $ L^\infty $ to $ Tr(u_\infty) $.
\end{rem}

\noindent In the context of Lemma \ref{lemmedeweakconti}, the sequence $ ( w^n(. , 0) ) _n $ is bounded in $ L^\infty (\RR^3) $.
After extraction of a subsequence, the expressions $ w^n (. , 0) = \cF ( \vert u^n \vert^2)$ converge weakly-$ \star $ to some limit $ \tilde w (k) $, while $ ( \vert u_n \vert^2 )_n $
does converge weakly-$ \star $ to some Radon measure $ \mu $ (with finite mass). We have $ \tilde w = \cF_x (\mu) $ but, in
general, there is no link between $ \mu $ and $ \vert u \vert^2 $. The trace $ w(\cdot, 0) $ is not well defined, and it cannot be associated with $ \tilde w $. This implies that, when
looking at (\ref{transporteqonw}) through compactness arguments, care should be taken to the
meaning of limit terms.

\begin{theo}[Global existence of weak solutions]\label{globalexistence}
Assume that
\begin{equation}\label{inidataboundMicutcutbis}
\tilde\chi \in L^2,
\qquad
\tilde\alpha_0 \in L^\infty_{1/2},
\qquad
u_0 \in L^2 \cap L^4,
\qquad
w_0 := (2\pi)^{-3}\,
\cF_{x,\xi} \circ {\rm W}(u_0).
\end{equation}
Then the Fourier--Moyal initial value problem
\eqref{transporteqonw}--\eqref{inidatatransporteqonw}
admits a global  weak $ L^2 $-solution  $w$.

\noindent Moreover, there exists
\[
u \in L^\infty_{\rm loc}(\RR_+;L^2)
\]
such that, for every $t\in\RR_+$,
\begin{equation}\label{issuchthat}
w(t,\cdot) \in \cC^\omega_{b,\eta}(L^2_k),
\qquad
w(t,\cdot)
= (2\pi)^{-3}\,
\cF_{x,\xi} \circ {\rm W}\bigl(u(t,\cdot)\bigr).
\end{equation}
\end{theo}
\noindent The above result provides in a weaker setting\footnote{The word {\it weak} means here that we implement $ \tilde\chi \in L^2 \cap L^\infty $ instead of $ \tilde\chi \in L^1_+ \cap L^\infty $ as in Theorem \ref{GESPD}.} a notion of weak $L^2$-solution to the
$\mathrm{SKG}_{\chi}$ system.

\begin{proof} Let $ (w^n)_n $ be the sequence which has been  constructed in Paragraph \ref{subsubsec:approximationscheme}.
We work modulo the extraction of subsequences which are not mentioned. Exploiting \eqref{conservnn} together with Banach-Alaoglu theorem, we can suppose that $ (u^n)_n $ converges weakly in $ L^2 $ to some $ u $, and that $ (w^n)_n $ converges simultaneously weakly in $ L^2 $ and weakly-$ \star $ in
$ L^\infty $ to some $ w \in L^2 \cap L^\infty $. The goal is to explain why this $ w $ is still a weak  solution to $\mathrm{SKG}_{\chi}$.

\smallskip

\noindent From \eqref{preparedpropagationattimetn} and Lemma 
\ref{lemmedeweakconti}, we can already deduce the right part of \eqref{issuchthat}.
As already mentioned, the difficulty is to pass to the limit in the bilinear terms. 
We want to check that (\ref{transporteqonw}) is satisfied in a weak sense. To this end, we test (\ref{transporteqonw}) against a compactly supported smooth function $ \varphi \in \cC^\infty_c ( [0,T \lbrack \times \RR^3 \times \RR^3) $. Then, the matter is to study the limit (when $ n $ goes to $ +\infty $) of the sequences $ \bigl( \Xi^n_\diamond (0) \bigr)_n $ where 
\[ \Xi^n_\diamond (\tilde \eta) := \int_0^T \! \! \int_{\RR^3} g_\diamond^n (t,\tilde k,\tilde \eta) \ h^n (t,\tilde k)  \ dt \, d \tilde k , \qquad \tilde \eta \in \RR^3 \, , \qquad \diamond \in \{ \alpha,w\} , \]
 where (coming from the $ \cL_t $-part)
 \[ g^n_\alpha (t,\tilde k,\tilde \eta) := \frac{2}{(2 \, \pi)^{3/2}} \ \tilde  \chi^n (\tilde k) \ \cI_k \bigl( e^{-i \omtilk t} \, \omtilk^{1/2} \, \tilde \alpha_0^n ( \tilde k) \bigr)  , \]
 or (coming from the $ \cQ_t $-part)
 \[ g^n_w (t,\tilde k,\tilde \eta) := - \frac{2}{(2 \, \pi)^3} \ \int_0^t \tilde \chi^n (\tilde k) \ \cI_k \bigl( i \,  e^{i \omtilk (s-t)} \, w^n(s,\tilde k,\tilde \eta) \bigr) \, ds , \]
 as well as
 \[ h^n (t,\tilde k)  := \int_{\RR^3} \! \int_{\RR^3} \sin (\tilde k \cdot \eta /2 ) \, w^n (t,k, \eta) \, \varphi(t,k+\tilde k, \eta) \ d k \, d \eta . \]
The assumptions made on $ \tilde \alpha_0 $ and $ \tilde \chi $ together with (\ref{conservnn}) give rise to
\[ \sup_n \ \parallel \tilde \chi^n  \parallel_{L^2} < + \infty \, , \qquad \sup_n \ \parallel \omtilk^{1/2} \, \tilde \alpha^n_0 \parallel_{L^\infty} < + \infty \, , \qquad \sup_n \ \sup_{t\in \RR} \ \parallel w^n (t,\cdot)  \parallel_{L^\infty} < + \infty \, . \]
It follows that, for all $ \diamond \in \{ \alpha,w \} $ and all $ t \in \RR $, we have
\[ \sup_n \ \sup_{\tilde \eta \in \RR^3} \parallel g^n_\diamond (t,\cdot,\tilde \eta) \parallel_{L^2} \lesssim (1+t) < + \infty \, . \]
On the one hand, we have $ g^n_\alpha (t,\tilde k, \tilde \eta) = e^{-i \omtilk t} \, \tilde g^n_\alpha (k) $. Since the sequence $ (\tilde g^n_\alpha)_n $ is bounded in $ L^2 $, we can extract from it a weak limit $ \tilde g \in L^2 $. Then
\begin{equation}\label{passlimgnalpha}
 \forall \, (t,\tilde \eta) \in \RR_+ \times \RR^3 \, , \qquad g^n_\alpha (t,\cdot,\tilde \eta)  \underset{n \rightarrow + \infty}{\overset{\ast}{\rightharpoonup}} g^\infty_\alpha (t,k) := e^{-i \omtilk t} \, \tilde g (k) \, .
\end{equation}
On the other hand, in view of the definition of $ g^n_w $, we have
\[ \begin{array}{l}
\displaystyle \sup_n \ \sup_{\tilde \eta \in \RR^3}\parallel \partial_t g^n_w (t,\cdot,\tilde \eta) \parallel_{L^2} \lesssim \sup_n \  \sup_{\tilde \eta \in \RR^3} \ \bigl( \parallel \tilde \chi^n \, w^n (t,\cdot,\tilde \eta)  \parallel_{L^2} + \, t \sup_{s\in[0,t]} \parallel (\chi^n)^2  \, w^n (s,\cdot,\tilde \eta)  \parallel_{L^2}
\bigr) \\
\displaystyle \qquad \lesssim \sup_n \  \sup_{\tilde \eta \in \RR^3} \ \bigl( \parallel \tilde \chi^n \parallel_{L^2}\, \parallel w^n (t,\cdot)  \parallel_{L^\infty} + \, t \, \parallel \chi^n \parallel_{L^\infty}^2 \, \sup_{s\in[0,t]} \parallel w^n (s,\cdot,\tilde \eta)  \parallel_{L^2}
\bigr) \lesssim 1 +t\, .
   \end{array} \]
As proved in Subsection \ref{subsec:propagationinfty}, the conservation of some $ L^2 $- control on the solutions $ u $ to $\mathrm{SKG}_{\chi}$ is related to the propagation of the $ L^\infty (\cC^{\omega}_b) $-norm of the solutions $ w $ to $\mathrm{FM}_{\chi}$. Along similar lines (details are not specified), the conservation of some $ L^4 $-control on $ u $ is linked to the propagation of the $ \cC^{\omega}_b(L^2) $-norm of $ w $. These correspondances between the $ L^2 $ and $ L^4 $-conditions on $ u $ and the $ L^\infty (\cC^{\omega}_b) $ and $ \cC^{\omega}_b(L^2) $-estimates on $ w $ are exhibited in Lemmas \ref{Propertiesofprepareddata} and \ref{Propertiesofprepareddata24}. Now, in the continuity of the $ L^4 $-stability established in \cite{MR778979}, we can already assert that
\begin{equation}\label{stabilityL4}  
\forall \, T \in \RR_+ \, , \qquad \forall \, t \in[0,T] \, , \qquad \sup_{t \in[0,T]} \, \sup_n \ \parallel u^n (t,\cdot) \parallel_{L^4} < + \infty \, . 
\end{equation}
From there, we can deduce that\footnote{To be strictly accurate, starting from \eqref{stabilityL4}, Lemma \ref{Propertiesofprepareddata24} furnishes a modulus of continuity $ \omega^n_t $ which is induced by the $ L^4 $-information on $ w^n (t,\cdot) $. But, as previously mentioned, this $ \omega^n_t $ is issued from the truncated regularized initial data. It follows that it can be chosen uniform with respect to $ (t , n) $, say $ \omega^n_t = \omega $ not depending on $ (t , n) $.}
\begin{equation}\label{supcomega} 
\forall \, t \in \RR_+ \, , \qquad \sup_n \ \parallel w^n (t,\cdot) \parallel_{\cC^\omega_{b,\eta}(L^2_k)} < + \infty \, . 
\end{equation}
As stated in \eqref{issuchthat}, this implies that $ w(t,\cdot) \in \cC^\omega_{b,\eta}(L^2_k) $. As a consequence, we can assert that the sequence $ \bigl(g^n_w (t,\cdot,\tilde \eta) \bigr)_n \in L^2(\RR^3)^\NN $ is both uniformly bounded and uniformly equicontinuous with respect to $ (t,\tilde \eta) \in \RR_+ \times \RR^3 $. From there, we can apply (a weak version of) Arzel\`a-Ascoli theorem to get
\[ \exists \, g_w^\infty \in \cC^0  \bigl( \RR_+ \times \RR^3_\eta ;L^2_k (\RR^3) \bigr)
 \, ; \qquad \forall \, (t,\tilde \eta) \in \RR_+ \times \RR^3 \, , \qquad  g^n_w (t,\cdot,\tilde \eta)  \underset{n \rightarrow + \infty}{\overset{\ast}{\rightharpoonup}} g^\infty_w (t,k,\tilde \eta) \, .
\]
Let $ \tilde \varphi \in \cC^\infty_c (\RR^3 \times \RR^3 ) $. Coming back to the definition of $ g^n_w $, it is easy to see that
\[ \begin{array}{l}
\displaystyle - \frac{2}{(2 \, \pi)^3} \ \int_{\RR^3} \int_{\RR^3} \int_0^t \tilde \chi (\tilde k) \ \cI_k \bigl( i \,  e^{i \omtilk (s-t)} \, w(s,\tilde k,\tilde \eta) \bigr) \ \tilde \varphi (\tilde k , \tilde \eta) \ ds \, dk \, d\eta \\
\qquad \qquad \qquad \qquad \displaystyle = \int_{\RR^3} \int_{\RR^3} g^\infty_w (t,\tilde k , \tilde \eta) \  \varphi (\tilde k , \tilde \eta) \ ds \, dk \, d\eta \, .
   \end{array} \]
We can take $ \tilde \varphi $ in the form of the tensor product $ \check \phi (\tilde k) \, \psi^m (\tilde \eta) $ with $ \check \varphi \in \cC^\infty_c (\RR^3) $ and $ \psi^m $ as in \eqref{apprschemes}. Using the above continuity properties with respect to $ \tilde \eta $, passing to the limit when $ m $ goes to $ +\infty $, we
find that
\[ - \frac{2}{(2 \, \pi)^3} \ \int_{\RR^3} \int_0^t \tilde \chi (\tilde k) \ \cI_k \bigl( i \,  e^{i \omtilk (s-t)} \, w(s,\tilde k,0) \bigr) \ \check \phi (\tilde k) \ ds \, dk =  \int_{\RR^3} g^\infty_w (t,\tilde k , 0) \ \check \phi (\tilde k) \ ds \, dk \, . \]
This works for all $ \check \phi $. Thus, we have identified how to express $ g^\infty_w $ in terms of $ w $. Retain that, for all $ t \in \RR_+ $, we have (in the sense of $ L^2 $-weak limit) that
\begin{equation}\label{passlimgnw}
  \qquad g^n_w (t,\cdot,0)  \underset{n \rightarrow + \infty}{\overset{\ast}{\rightharpoonup}}
  g^\infty_w (t,\tilde k , 0) = - \frac{2}{(2 \, \pi)^3} \ \int_0^t \tilde \chi (\tilde k) \ \cI_k \bigl( i \,  e^{i \omtilk (s-t)} \, w(s,\tilde k,0) \bigr) \ ds \, .
 \end{equation}
Let us now consider what happens at the level of $ (h^n)_n $. It is clear that
\[ \forall \, (t,\tilde k) \, , \qquad \lim_{n \rightarrow + \infty} h^n (t,\tilde k) = h^\infty (t,\tilde k) := \int_{\RR^3} \! \int_{\RR^3} \sin (\tilde k \cdot \eta /2 ) \, w (t,k, \eta) \, \varphi(t,k+\tilde k, \eta) \ d k \, d \eta . \]
We claim that (modulo the extraction of subsequences), for all compact subset $ K $ of $ \RR^3 $, this convergence is strong in $ L^\infty \bigl( [0,T] ; L^2 (K) \bigr) $
. This is a consequence of Kondrachov embedding theorem (in $ \tilde k $) combined with Arzel\`a-Ascoli theorem (in $ t $). By H\"older's inequality, we have
\[ \begin{array}{rl}
    \vert h^n(t,\tilde k) \vert \! \! \! & \displaystyle \leq \int_{\RR^3} \Bigl( \int_{\RR^3} \vert \varphi(t,\tilde k + k, \eta) \vert \ \vert w^n (t,k, \eta) \vert \ d \eta  \Bigr) \, d k \smallskip \\
    \ & \displaystyle \leq \int_{\RR^3} \parallel \varphi(t,\tilde k + k, \cdot) \parallel_{L^2_\eta} \ \parallel w^n (t,k, \cdot) \parallel_{L^2_\eta} \, d k \, .
   \end{array}
\]
On the right-hand side, we can recognize a convolution product. By  Young's inequality, we have
\begin{equation}\label{unifdf}
 \sup_n \parallel h^n (t,\cdot) \parallel_{L^2} \leq \parallel \varphi(t,\cdot) \parallel_{L^1_k (L^2_\eta)} \ \sup_n \parallel w^n (t,\cdot) \parallel_{L^2} < + \infty \, . 
 \end{equation}
Looking at the definition of $ h^n $, we see that the derivatives of $ h^n (t,\cdot) $ produce powers of $ \eta $ and derivatives of $ \varphi $. Then,  as described above, we obtain that
\[ \sup_n \parallel h^n (t,\cdot) \parallel_{H^m} \leq \parallel \langle \eta \rangle^m \varphi(t,\cdot) \parallel_{L^1_k (H^m_\eta)} \ \sup_n \parallel w^n (t,\cdot) \parallel_{L^2} < + \infty \, . \]
Thus, for all $ t \in \RR_+ $, there is a  subsequence $ \bigl( h^n(t,\cdot)\bigr)_n $ which converges strongly in $ L^2 $.
There remains to estimate uniformly the time derivative of $ h^n $. Remark that
\[ \begin{array}{rl}
\displaystyle \vert \part_t h^n (t,\tilde k)  \vert \lesssim \! \! \! & \parallel w^n (t,\cdot) \parallel_{L^2} \ \parallel \part_t \varphi \parallel_{L^2} \smallskip \\
\ & \displaystyle + \, \Bigl \vert \int_{\RR^3} \! \int_{\RR^3} \sin (\tilde k \cdot \eta /2 ) \, \part_t w^n (t,k, \eta) \, \varphi(t,k+\tilde k, \eta) \ d k \, d \eta \Bigr \vert \, .
\end{array} 
\]
We can use the equation (\ref{transporteqonw}) to replace $ \part_t w^n $. We perform an integration by parts with respect to $ \eta $ to absorb $ \nabla_\eta $. This furnishes a power of $ \tilde k $ but $ \tilde k \in K $ is bounded. 
Now, let us look at $ \cQ_t (w^n) $. 
Coming back to the definition \eqref{bilitodefine} and exploiting Young's inequality, it is clear that
\[ \begin{array}{rl}
\displaystyle \parallel \cB_{\tilde \chi} [w_1,w_2](t,\cdot) \parallel_{L^\infty} \! \! \! & \displaystyle \lesssim \, \parallel \tilde \chi \, w_1 (t,\cdot,0) \parallel_{L^2} \ \sup_\eta \, \parallel w_2 (t,\cdot,\eta) \parallel_{L^2} \\ 
\ & \displaystyle \lesssim \,
\parallel \tilde \chi \parallel_{L^2} \ 
\parallel w_1 (t,\cdot,0) \parallel_{L^2}
 \sup_\eta \, \parallel w_2 (t,\cdot,
\eta) \parallel_{L^2}  \, .
\end{array} \]
In particular, due to \eqref{supcomega}, we can assert that
\[ \sup_n \, \parallel \cB_{\tilde \chi} \bigl \lbrack g_{w_n}(s,t,\cdot), w_n (t,\cdot) \bigr \rbrack](t,\cdot) \parallel_{L^\infty} \lesssim \, 1 \, . \]
It follows that $ \cQ_t (w^n) $ is uniformly bounded. For similar reasons, the same applies concerning $ \cL_t w_n $. Briefly, we  have just seen that
 $ \part_t h_n (t,\tilde k) $ with $ \tilde k \in K $ is uniformly bounded in $ L^\infty $.  This gives rise to some equicontinuity in time. By Arzel\`a-Ascoli theorem, by extracting a second
 subsequence if necessary, it can be ensured that the subsequence  $ (h_{\psi(n)})_n $ of continuous functions converges uniformly (as well as all its
 derivatives) on all compact subsets of $ [0,T] \times \RR^3 $ to the continuous function
\begin{equation}\label{hazekfhk}
 h^\infty (t,\tilde k)  := \int_{\RR^3} \! \int_{\RR^3} \sin (\tilde k \cdot \eta /2 ) \, w (t,k, \eta) \, \varphi(t,k+\tilde k, \eta) \, d k d \eta .
 \end{equation}
Note that,  for all $ t \in [0,T] $, we have
 \begin{equation}\label{l2uniboundninfty}
\parallel h^\infty (t,\cdot) \parallel_{L^2} \leq \parallel w(t,\cdot) \parallel_{L^2_{k,\eta}} \, \parallel \varphi (t,\cdot) \parallel_{L^1_k (L^2_\eta)} .
\end{equation}
The difficulty is to pass to the limit (when $ n \rightarrow + \infty $) in the nonlinear terms, and to recognize at the limit the corresponding nonlinear expressions of the weak limits. Since the correct link between $ g_\diamond^\infty $ and $ h_\infty $ with $ w $ has been established at the level of \eqref{passlimgnalpha}, \eqref{passlimgnw} and \eqref{hazekfhk}, it suffices to show that
\[ \lim_{n \rightarrow + \infty} \ \int_0^T \! \! \int_{\RR^3} g^n_\diamond (t,\tilde k,0) \, h^n (t,\tilde k)  \, dt d \tilde k = \int_0^T \! \! \int_{\RR^3} g_\diamond^\infty
(t,\tilde k,0) \, h^\infty (t,\tilde k)  \, dt d \tilde k . \]
When restricted to compact sets, this convergence is easy. Indeed, the weak
$ L^2 $-convergence of $ (g^\diamond_n)_n $ and the strong $ L^2 $-convergence of $ (h_n)_n $ are sufficient to conclude.
The difficulty is to get rid of the large values of $ \vert \tilde k \vert $. Let $ M \in \RR_+^* $. By H\"older's inequality, exploiting the uniform bound \eqref{unifdf} on the $ L^2 $-norm of $ h^n $, we get that
\[ \Bigl \vert \int_{M \leq \vert \tilde k \vert} g^n_\diamond (t,\tilde k,0) \ h^n (t,\tilde k) \, d \tilde k \Bigr \vert \leq \int_{M \leq \vert \tilde k \vert} \vert g^n_\diamond (t,\tilde k,0)   \vert^2 \, d \tilde k . \]
Since the sequences $ (\omtilk^{1/2} \, \tilde \alpha^n_0)_n $ and $ (w^n)_n $ are bounded in $ L^\infty $, coming back to the definition of $ g^n_\diamond $, we find that
\[ \int_{M \leq \vert \tilde k \vert} \vert g^n_\diamond (t,\tilde k,0) \vert^2 \, d \tilde k \lesssim \int_{M \leq \vert \tilde k \vert} \tilde \chi^n (\tilde k)^2 \, d \tilde k \lesssim \int_{M-1 \leq \vert \tilde k \vert}
 \tilde \chi (\tilde k)^2 \, d \tilde k . \]
We have assumed that $ \tilde \chi \in L^2 $. Thus, for $ M $ large enough, the right hand side can be made as small as wanted. This completes the proof.
\end{proof}

\appendix
\section{Phase-space analysis}\label{sec:appendix}

The aim of this subsection is to collect standard tools from time-frequency analysis which are exploited in our text, and which are discussed at length for instance in \cite{Chaichenets,MR3800635,MR3643624,MR2294795}.

\addtocontents{toc}{\setcounter{tocdepth}{-10}}
\subsection{About the Wigner transform}\label{aboutwigner}
The reader can simply refer to \cite{MR2502626}, or to the book \cite{MR3643624} for more details. Let $ d \in \NN^* $.
The (cross-)Wigner transform  is the continuous sesquilinear map
\begin{equation}\label{crosswigner}
\begin{array}{rcl}
\mathbb W : \cS (\RR^d ; \CC) \times \cS (\RR^d ; \CC) & \longrightarrow & \cS (\RR^d \times \RR^d ; \CC) \\
\bigl(u(x) , \slu (x) \bigr) & \longmapsto & \mathbb W[u , \slu ] (x,\xi)
\end{array}
\end{equation}
which is defined according to
\begin{equation}\label{crosswignerdef}
\mathbb W[u , \slu ] (x,\xi) := \int_{\RR^d} e^{-i \xi \cdot y} \, u \bigl( x + \frac{y}{2} \bigr) \, \overline{\slu} \bigl( x - \frac{y}{2} \bigr) \,
dy = \cF \Bigl( u \bigl( x + \frac{\cdot}{2} \bigr) \, \overline{\slu} \bigl( x - \frac{\cdot}{2} \bigr)  \Bigr) (\xi) .
\end{equation}
From Plancherel theorem, we can easily deduce the relation
\begin{equation}\label{crosswignerplancherelidentity}
\mathbb W[u , \slu ] (x,\xi) = \frac{1}{(2 \pi)^d} \, \mathbb W[ \hat u , \hat \slu ] (\xi,-x) ,
\end{equation}
as well as the so called Moyal identity
\begin{equation}\label{crosswignermoyalidentity}
\langle \mathbb W[u , \slu ] , \mathbb W[ \tilde u , \tilde \slu ] \rangle_{L^2 \times L^2} = (2 \pi)^d \, \langle u , \tilde u \rangle_{L^2 \times L^2}  \,
\overline{\langle \slu ,\tilde \slu \rangle}_{L^2 \times L^2}  ,
\end{equation}
which implies
\begin{equation}\label{crosswignermoyalidentityb}
\parallel \mathbb W[u , \slu ] \parallel_{L^2} = (2 \pi)^{d/2} \parallel u \parallel_{L^2} \, \parallel \slu \parallel_{L^2} .
\end{equation}
On the other hand, by applying H\"older inequality to (\ref{crosswignerdef}) one gets
\begin{equation}\label{crosswignerholderdentity}
\parallel \mathbb W[u , \slu ] \parallel_{L^\infty} \leq 2^d \parallel u \parallel_{L^2} \, \parallel \slu \parallel_{L^2} .
\end{equation}
From the right part
of (\ref{crosswignerdef}), we get the inversion formula
\begin{equation}\label{recallwig}
u(x) \, \overline{\slu} (y) = \frac{1}{(2 \pi)^d}  \int_{\RR^d} \mathbb W[u,\slu] \Bigl( \frac{x+y}{2},\xi \Bigr) \, e^{i \xi \cdot (x-y)} \, d \xi .
\end{equation}
The Wigner function $  v \in \cS (\RR^d \times \RR^d;\RR) $ associated with $ u \in \cS (\RR^d ;\CC) $ is simply
\begin{equation}\label{wignersimple}
 v ( x,\xi ) \equiv {\rm W} (u) ( x,\xi ) := \mathbb W \lbrack u , u \rbrack (x,\xi) = \bar v  ( x,\xi ) .
 \end{equation}
In (\ref{recallwig}), replace $ \slu $ by $ u $ and $ y $ by $ x $ to find that
 \begin{equation}\label{recallwigff}
\qquad \quad \vert u(x) \vert^2 = \frac{1}{(2 \pi)^d}  \int_{\RR^d} v ( x ,\xi ) d \xi , \qquad \vert \hat u (\xi) \vert^2
= \int_{\RR^d} v ( x ,\xi ) d x ,
\end{equation}
and therefore
\begin{equation}\label{recallwigffl1}
\parallel u \parallel^2_{L^2 (\RR^d)} = \frac{1}{(2 \pi)^d} \parallel \hat u \parallel^2_{L^2 (\RR^d)} = \frac{1}{(2 \pi)^d}
\int_{\RR^d} \! \int_{\RR^d} v ( x ,\xi ) dx d \xi .
\end{equation}
On the other hand, looking at (\ref{crosswignermoyalidentity}), we know that
 \begin{equation}\label{Moyals}
\int_{\RR^d} \vert u(x) \vert^2 \, dx = \frac{1}{(2 \pi)^{d/2}}  \left( \int_{\RR^d} \! \int_{\RR^d} \vert v ( x ,\xi ) \vert^2 \, dx d \xi \right)^{1/2}  .
\end{equation}
Applying (\ref{recallwig}) to $ \hat u $  rather than to  $ u $ and  $ \slu $, and then using (\ref{crosswignerplancherelidentity}), we obtain
\begin{equation}\label{wignerdenucleaon}
\cF_{\! x} \bigl( v (\cdot,\xi) \bigr) (k) =  \hat u \Bigl( \xi + \frac{k}{2} \Bigr) \, \overline{\hat u} \Bigl( \xi - \frac{k}{2} \Bigr)  .
\end{equation}
On the other hand, the Fourier--Wigner transform of $ u $ as in  
definition~\eqref{defunknownw} is equivalent to
\begin{equation}\label{Fourier_Wigner_formula_bis}
\begin{aligned}
w(k,\eta)
&= \int_{\RR^d} e^{-ik\cdot x}\,
\cF_\xi \circ \cF^{-1}_{y\to \xi}
\bigl(
u(x-y/2)\,\overline{u}(x+y/2)
\bigr)(\eta)\,dx \,.
\end{aligned}
\end{equation}
Making the change of variables $y\mapsto -y$ in
\eqref{Fourier_Wigner_formula_bis} yields
\begin{equation}\label{alternativeformulaforw}
w(k,\eta)= (2\pi)^{-d} \, \cF_{x,\xi} \circ
{\rm W}(u)(k,\eta) \, .
\end{equation}
By exchanging the order of integration in
\eqref{defunknownw}, we obtain the alternative representation
\begin{equation}\label{autreformulapourwbis}
w(k,\eta) = \int_{\RR^d} e^{-i k \cdot x} \ u \Bigl( x - \frac{\eta}{2} \Bigr) \ \bar u \Bigl( x + \frac{\eta}{2} \Bigr) \ dx \, .
\end{equation}
Finally, the cross-Wigner transform $ \mathbb W $ can be extended to a map $ \cS' (\RR^d ; \CC) \times \cS' (\RR^d ; \CC) \longrightarrow \cS' (\RR^d \times \RR^d ; \CC) $. Moreover,  from (\ref{crosswignermoyalidentityb}) and (\ref{crosswignerholderdentity}), we can infer that the restriction of $ \mathbb W $ to the
subspace $ L^2 (\RR^d ; \CC) \times L^2 (\RR^{d} ; \CC) $ maps into $ (L^2 \cap \cC^0)(\RR^{2d} ; \CC) $.

\subsection{About the short-time Fourier transform}\label{aboutstft} 
We briefly recall the definition and some basic properties of the short-time Fourier transform (STFT); see, for instance, \cite{Chaichenets,MR2294795}. Given $ x,\xi \in \RR^d $, define the translation and modulation operators by
  \begin{equation}\label{trans-mod}
 T_x u (t) := u(t-x) , \qquad M_\xi u (t) := e^{i \xi \cdot t} u(t) .
\end{equation}
The short-time Fourier transform (STFT) of the function $ u $ with respect to the window $ \slu $ is defined by
\begin{equation}\label{STFT}
 V_{\slu} u (x,\xi) := \int_{\RR^d} u(t) \, \bar \slu(t-x) \, e^{-it\cdot \xi} \, dt = \langle u , M_\xi  T_x \slu \rangle_{L^2 \times L^2} .
\end{equation}
An immediate consequence of the Plancherel theorem is the identity
\begin{equation}\label{fondaidenity}
V_{\slu}u(x,\xi)
=e^{-ix\cdot\xi}\,V_{\hat{\slu}}\hat{u}(\xi,-x).
\end{equation}
Finally, writing $\check{\slu}(t):=\slu(-t)$ for the reflected window and performing the change of variables
$t=x+\frac{y}{2}$ in \eqref{crosswignerdef}, we see that the Wigner distribution is essentially a short-time Fourier transform. More precisely,
\begin{equation}\label{linkWigSTFT}
\mathbb W[u,\slu](x,\xi)
=
2^d\,e^{2ix\cdot\xi}\,
V_{\check{\slu}}u(2x,2\xi).
\end{equation}


\subsection{About the modulation spaces}\label{aboutmodulation} 
The presentation below is largely inspired by Chapter~2 (Preliminaries)
of \cite{MR3800635}. Starting from the definition
\eqref{crosswignerdef}, we consider the Fourier transform of the cross-Wigner
distribution with respect to both the space and frequency variables
$x$ and $\xi$, namely
\begin{equation}\label{fourierdewig}
\begin{aligned}
w_c(k,\eta)
&:= \cF_{x,\xi}\bigl(\mathbb W[u,\slu]\bigr)(k,\eta) \\
&= \int_{\RR^d}\!\!\int_{\RR^d}
e^{-i(k\cdot x+\eta\cdot\xi)}
\,\mathbb W[u,\slu](x,\xi)\,dx\,d\xi.
\end{aligned}
\end{equation}
\begin{lem}[Relation between the cross-Wigner transform and the STFT]
\label{LinkbetweenWignerandSTFT}
The Fourier transform of the cross-Wigner distribution satisfies
\begin{equation}\label{linkFwignertostft}
w_c(k,\eta)
=e^{i k\cdot\eta/2}\,
V_{\hat{\slu}}\hat{u}(k,\eta).
\end{equation}
In particular, when $\slu=u$, identity \eqref{linkFwignertostft} reduces to
\eqref{wignerdenucleaon}.
\end{lem}

\begin{proof} Using \eqref{STFT}, \eqref{linkWigSTFT} and then integrating with respect to $x$, yields 
\begin{equation}\label{handledthrough}
\begin{array}{rl}
\displaystyle \int_{\RR^d} e^{-i k \cdot x} \, \mathbb W[u , \slu ] (x,\xi) \, dx \! \! \! & = \displaystyle 2^d \int_{\RR^d} \! \int_{\RR^d}
e^{i x\cdot  (2 \xi-k)} \, u(t) \, \bar \slu(2x-t) \, e^{-2 it\cdot \xi} \, dt  \, dx \\
\ &  \displaystyle =  \int_{\RR^d} e^{- i t \cdot  (\xi+ \frac{k}{2})} \, u(t) \, \overline {\hat \slu} \Bigl( \xi - \frac{k}{2} \Bigr) \, dt  =
\hat u \Bigl( \xi + \frac{k}{2} \Bigr) \, \overline {\hat \slu} \Bigl( \xi - \frac{k}{2} \Bigr) .
\end{array}
\end{equation}
When $ \slu = u $, we recover (\ref{wignerdenucleaon}). Applying Fubini's theorem and the change of variables $ t = \xi + \frac{k}{2} $ gives
\[
\int_{\RR^d}
e^{-i\eta\cdot\xi}
\hat u\!\left(\xi+\frac{k}{2}\right)
\overline{\hat\slu\!\left(\xi-\frac{k}{2}\right)}
\,d\xi
=
e^{ik\cdot\eta/2}
\int_{\RR^d}
e^{-it\cdot\eta}
\hat u(t)\,
\overline{\hat\slu(t-k)}
\,dt
=
e^{ik\cdot\eta/2}
V_{\hat\slu}\hat u(k,\eta),
\]
which is exactly \eqref{linkFwignertostft}.
\end{proof}

Denote by $\alpha(\cdot)$ and $\beta(\cdot)$ two even positive weight
functions on $\RR^d$. Typical examples are polynomial weights, such as
$\alpha(z)=\langle z\rangle^r$ and $\beta(z)=\langle z\rangle^s$ for some
$r,s\in\RR$. Let $g\in\cS(\RR^d)\setminus\{0\}$ be a fixed window function,
and let $1\leq p,q\leq+\infty$. The \emph{modulation space}
$M^{p,q}_{\alpha\otimes\beta}(\RR^d)$ is defined as the space of all tempered
distributions $u\in\cS'(\RR^d)$ such that the short-time Fourier transform
$V_g u$ belongs to the weighted mixed-norm Lebesgue space
$L^q_{\beta,\xi}(L^p_{\alpha,x})$, namely,
(with the usual modifications when $p=+\infty$ or $q=+\infty$)
\begin{equation}\label{normmodulation} \  
\|u\|_{M^{p,q}_{\alpha\otimes\beta}}
:=
\|V_g u\|_{L^q_{\beta,\xi}(L^p_{\alpha,x})}
=
\left(
\int_{\RR^d}
\left(
\int_{\RR^d}
|V_g u(x,\xi)|^p\,\alpha(x)^p\,dx
\right)^{q/p}
\beta(\xi)^q\,d\xi
\right)^{1/q} \! \! 
<+\infty.
\end{equation}
The definition of $M^{p,q}_{\alpha\otimes\beta}(\RR^d)$ is independent of the
choice of the non-zero window $g\in\cS(\RR^d)$. More precisely, different
choices of $g$ give rise to equivalent norms; see for instance
\cite{Chaichenets}. Moreover, Proposition~2.24 of \cite{Chaichenets} shows
that, in the particular case $\alpha=1$ and
$\beta(\xi)=\langle\xi\rangle^s$, $s\in\RR$, the Schwartz assumption on the
window can be substantially relaxed. Indeed, the norm
\eqref{normmodulation} remains equivalent for every non-zero window
$g\in M^{1,1}_{1\otimes\langle\xi\rangle^s}$. Finally, the choice
$(p,q)=(\infty,1)$ yields the celebrated Sj\"ostrand class
\cite{MR1266757}.


\subsection{About the Wiener amalgam spaces}\label{aboutwiener}
For a concise introduction to Wiener amalgam spaces, we refer the reader
to \cite{MR3800635}. The Wiener amalgam space
$W(\cF L^p_\alpha,L^q_\beta)(\RR^d)$ is defined as the space of all tempered
distributions $u\in\cS'(\RR^d)$ satisfying
\begin{equation}\label{normamalgamsp}
\|u\|_{W(\cF L^p_\alpha,L^q_\beta)}
:=
\left(
\int_{\RR^d}
\left(
\int_{\RR^d}
|V_g u(x,\xi)|^p\,\alpha(\xi)^p\,d\xi
\right)^{q/p}
\beta(x)^q\,dx
\right)^{1/q}
<+\infty,
\end{equation}
or, equivalently,
\[
V_g u\in L^q_{\beta,x}(L^p_{\alpha,\xi}).
\]
Comparing \eqref{normmodulation} with \eqref{normamalgamsp}, we observe that
the roles of the variables $x$ and $\xi$ are interchanged. The connection
between modulation and Wiener amalgam spaces follows immediately from
\eqref{fondaidenity}. Indeed, since the weight $\beta$ is even, we obtain
\begin{equation}\label{link-mod-Wien}
\|u\|_{M^{p,q}_{\alpha\otimes\beta}}
=
\|\hat u\|_{W(\cF L^p_\alpha,L^q_\beta)}.
\end{equation}
Thus, Wiener amalgam spaces are precisely the Fourier images of modulation
spaces.

\noindent 
Next, let $\slu=u$ in \eqref{linkFwignertostft}. In view of
\eqref{alternativeformulaforw}, we recover
\[
|w(k,\eta)|
=
(2\pi)^{-d}
\bigl|
\cF_{x,\xi}\bigl(\mathbb W[u,u]\bigr)(k,\eta)
\bigr|
=
|V_{\hat u}\hat u(k,\eta)|.
\]
Using \eqref{fondaidenity}, this can be rewritten as
\[
\langle k\rangle^s |w(k,\eta)|
=
\langle k\rangle^s |V_u u(-\eta,k)|.
\]
Therefore, provided that $u$ is admissible as a window function, the
modulation norm \eqref{normmodulation} yields
\begin{equation}\label{equivnormwigsjo}
\|w\|_{L^p_{s,k}(L^\infty_\eta)}
=
\|u\|_{M^{\infty,p}_{1\otimes\langle\xi\rangle^s}}.
\end{equation}
In other words, the phase-space framework
$L^p_{s,k}(L^\infty_\eta)$ naturally corresponds to the modulation space
$M^{\infty,p}_{1\otimes\langle\xi\rangle^s}$, which reduces to
Sj\"ostrand's class when $p=1$. The functional framework $ L^\infty(\cC^\omega_b) $ highlighted in the present paper (and microlocally propagated in the context of FM equations) is a refinement of $ L^\infty_{s,k}(L^\infty_\eta) $.


\bibliographystyle{abbrv}

 \newcommand{\noop}[1]{}

\end{document}